\documentclass[11pt]{article}

\usepackage{amsmath}
\usepackage{amssymb}
\usepackage{latexsym}
\usepackage{enumitem}
\usepackage{mathrsfs}
\usepackage{comment}
\usepackage{color}
\usepackage{bbm}

\newtheorem{assumption}{Assumption}

\def\qed{ \ \vrule width.2cm height.2cm depth0cm\smallskip}
\newenvironment{proof}{\noindent {\bf Proof.\/}}{$\qed$\vskip 0.1in}

\newcommand{\ba}{\begin{array}}
\newcommand{\ea}{\end{array}}
\newcommand{\be}{\begin{equation}}
\newcommand{\ee}{\end{equation}}
\newcommand{\bea}{\begin{eqnarray}}
\newcommand{\eea}{\end{eqnarray}}
\newcommand{\beaa}{\begin{eqnarray*}}
\newcommand{\eeaa}{\end{eqnarray*}}

\def\dbE{\mathbb{E}}
\def\dbF{\mathbb{F}}

\def\dbL{\mathbb{L}}

\def\dbP{\mathbb{P}}
\def\dbR{\mathbb{R}}

\def\dbT{\mathbb{T}}

\def\dbZ{\mathbb{Z}}

\def\a{\alpha}
\def\b{\beta}

\def\d{\delta}
\def\e{\varepsilon}
\def\z{\zeta}
\def\k{\kappa}
\def\l{\lambda}

\def\si{\sigma}
\def\t{\tau}
\def\f{\varphi}
\def\th{\theta}
\def\o{\omega}
\def\h{\widehat}
\def\G{\Gamma}
\def\D{\Delta}
\def\Th{\Theta}

\def\F{\Phi}
\def\O{\Omega}
\def\cA{{\cal A}}

\def\cF{{\cal F}}
\def\cG{{\cal G}}

\def\cJ{{\cal J}}

\def\cL{{\cal L}}
\def\cM{{\cal M}}

\def\cP{{\cal P}}

\def\cT{{\cal T}}

\def\cW{{\cal W}}
\def\cX{{\cal X}}
\def\cY{{\cal Y}}
\def\cZ{{\cal Z}}

\def\ch{\textsc{h}}

\def\no{\noindent}

\def\ms{\medskip}
\def\bs{\bigskip}
\def\q{\quad}
\def\qq{\qquad}

\def\pa{\partial}
\def\cd{\cdot}
\def\cds{\cdots}

\def\td{\nabla}

\def\tr{\hbox{\rm tr}}

\def\qed{ \hfill \vrule width.25cm height.25cm depth0cm\smallskip}

\newcommand{\basa}{\begin{assumption}}
\newcommand{\easa}{\end{assumption}}

\newcommand{\bas}{\begin{assum}}
\newcommand{\eas}{\end{assum}}

\def\limsup{\mathop{\overline{\rm lim}}}

\def\pa{\partial}
\def\h{\widehat}

 \def\cd{\cdot}
\def\cds{\cdots}

\def\supp{\hbox{\rm supp$\,$}}

\def\tr{\hbox{\rm tr$\,$}}

\def\dis{\displaystyle}

\def\wh{\widehat}

\def\bx{{\bf x}}

\def\1{{\bf 1}}

\def\:{\!:\!}
\def\reff#1{{\rm(\ref{#1})}}
\def \proof{{\noindent \bf Proof\quad}}

at 9pt

\newtheorem{thm}{Theorem}[section]
\newtheorem{lem}[thm]{Lemma}

\newtheorem{prop}[thm]{Proposition}
\newtheorem{rem}[thm]{Remark}
\newtheorem{eg}[thm]{Example}
\newtheorem{defn}[thm]{Definition}
\newtheorem{assum}[thm]{Assumption}

\begin{document}

\title{\bf Wellposedness of Second Order Master Equations for Mean Field Games with Nonsmooth Data
\thanks{An earlier version of this paper is entitled "Weak Solutions of Mean Field Game Master Equations", see arXiv:1903.09907v1.  The authors would like to thank all the feedbacks received on that version.} } 
\author{Chenchen Mou\thanks{\noindent  Department of Mathematics, 
City University of Hong Kong. E-mail: chencmou@cityu.edu.hk. This author is supported by CityU Start-up Grant 7200684.} ~ and ~ Jianfeng Zhang\thanks{\noindent  Department of Mathematics,
University of Southern California. E-mail:
jianfenz@usc.edu.  This author is supported in part by NSF grant DMS-1908665.}  
}
\date{}
\maketitle

\begin{abstract}
In this paper we study second order master equations arising from mean field games with common noise over arbitrary time duration. A classical solution typically requires the monotonicity condition (or small time duration) and sufficiently smooth data. While keeping the monotonicity condition, our goal is to relax the regularity of the data, which is an open problem in the literature. In particular, we do not require any differentiability in terms of the measures, which prevents us from obtaining classical solutions. We shall propose three weaker notions of solutions, named as {\it good solutions}, {\it weak solutions},  and {\it weak-viscosity solutions}, respectively, and establish the wellposedness of the master equation under all three notions. We emphasize that, due to the game nature, one cannot expect comparison principle even for classical solutions.  The key for the global (in time) wellposedness is the uniform a priori estimate for the Lipschitz continuity of the solution in the measures. The monotonicity condition is crucial for this uniform estimate and thus is crucial for the existence of the global solution, but is not needed for the uniqueness in such Lipschitz class. To facilitate our analysis,  we construct a smooth mollifier for functions on Wasserstein space, which is new in the literature and is interesting in its own right. 

Following the same approach of our wellposedness results, we prove the convergence of the Nash system, a high dimensional system of PDEs arising from the corresponding $N$-player game, under mild regularity requirements. We shall also prove  a propagation of chaos property for the associated optimal trajectories.
\end{abstract}

\no{\bf Keywords.}  Mean field game, $N$-player game, master equation, Nash system, 
forward-backward SDEs, good solutions, weak solutions, weak-viscosity solutions, Wasserstein spaces 

\ms
\no{\it 2020 AMS Mathematics subject classification:}  49N80, 35Q89, 91A16, 60H30  


\tableofcontents

\vfill\eject
\section{Introduction}
\label{sect-Introduction}
\setcounter{equation}{0} 

\subsection{Literature review}
Initiated independently by Caines-Huang-Malhame \cite{CHM} and Lasry-Lions \cite{LL}, mean field games have received very strong attention in the past decade. We refer to Lions \cite{Lions} and Cardaliaguet \cite{Cardaliaguet} for  introduction of the subject in early stage and Camona-Delarue \cite{CD1, CD2} for more recent developments. Such problems consider limit behavior of large systems where the agents interact with each other in some symmetric way, with the systemic risk as a notable application. The master equation, introduced by Lions \cite{Lions}, is a powerful tool in this framework, which plays the role of the PDE in the standard literature of controls/games:
\bea
\label{masterint}
&\dis \cL V(t,x,\mu) :=\pa_t V + \frac{\b_1^2+\b^2}{2}\tr(\pa_{xx} V) + H(x,\partial_x V) + F(x,\mu) + \cM V =0,\nonumber\\
&\dis V(T,x,\mu) = G(x,\mu),\q \mbox{where}\q\\
&\dis   \cM V(t,x,\mu)  :=\tr\Big( \tilde  \dbE\Big[\frac{\b_1^2+\b^2}{2} \pa_{\tilde x} \pa_\mu V(t,x, \mu, \tilde \xi) + \pa_\mu V(t, x, \mu, \tilde \xi)(\pa_pH)^\top(\tilde \xi, \pa_x V(t, \tilde \xi, \mu))\nonumber\\
&\dis  +\b^2\pa_x\pa_\mu V(t,x,\mu,\tilde \xi)+\frac{\b^2}{2}\bar{\mathbb E}\big[\pa_{\mu\mu}V(t,x,\mu,\bar\xi,\tilde\xi)\big]\Big]\Big).\nonumber
\eea
Here $\b_1, \b$ are two constants, $\partial_t,\partial_x,\partial_{xx}$ are standard temporal and spatial derivatives, $\partial_{\mu},\partial_{\mu\mu}$ are Wasserstein derivatives with respect to the measure $\mu$, $\tilde \xi$ and $\bar\xi$ are independent random variables with the same law $\mu$, and $\tilde {\mathbb E}$ and $\bar{\mathbb E}$ are (conditional) expectations corresponding to $\tilde \xi$ and $\bar\xi$ respectively. The main feature of the master equation is that its state variables include a probability measure $\mu$, typically the distribution of certain underlying state process, so it can be viewed as a PDE on the Wasserstein space. By nature this is an infinite dimensional problem. After \cite{Lions}, new understandings of the master equation have been discovered, see, e.g.,  Bensoussan-Frehse-Yam \cite{BFY2,BFY3}  and Carmona-Delarue \cite{CD} where master equations are derived formally following different approaches.

There have been serious efforts on classical solutions of master equations in the past years, especially for global (in time) solutions. Buckdahn-Li-Peng-Rainer \cite{BLPR} established the global wellposedness for a linear master equation ($H$ is linear in $\pa_x V$)  in the case $\b=0$ by means of probabilistic techniques.
 Chassagneux-Crisan-Delarue \cite{CCD} used FBSDEs of the McKean-Vlasov type to study the global wellposedness for master equations with a general Hamiltonian, also without common noise ($\beta=0$). The groundbreaking paper Cardaliaguet-Delarue-Lasry-Lions \cite{CDLL}, by using a PDE approach, obtained global wellposedness of the master equation with general Hamiltonian and common noise. Moreover, \cite{CDLL} used the classical solution of the master equation to justify the mean field limit, i.e. the convergence of the Nash systems for the $N$-player games to  the master equation for the mean field game, as well as  the propagation of chaos for the $N$-player games with closed loop equilibria. Cardaliaguet-Cirant-Porretta \cite{CCP} developed a splitting method to prove local wellposedness of the master equation with more general Halmiltonian in the form $H(x, \mu,\pa_x V)$ for standard mean field games and mean field games with a major player. Moreover, there are several works in the realm of potential mean field games. Gangbo-Swiech \cite{GS2} showed the first order master equation ($\b_1=\beta=0$), derived from a deterministic linear quadratic mean field control problem, admits a local (in time) classical solution.  This was recently extended to global wellposedness for general Hamiltonian by Gangbo-Meszaros \cite{GM}. Bensoussan-Yam \cite{BY} studied the same type of problem, but by using the ``lifting" idea introduced in \cite{Lions}; and together with Graber,  they recently extended the result to the case that involves individual noises ($\b_1 >0$, but $\b=0$) in \cite{BGY1,BGY2}, with both local and global wellposedness results.  We emphasize that all the above global wellposedness results are under certain monotonicity assumption on $F$ and $G$, with the exception \cite{BLPR} which is linear and thus does not involve controls. In particular, \cite{CCD, CDLL} used the Lasry-Lions monotonicity condition, while \cite{BGY1,BGY2,GM} used the displacement convexity condition which implies the so-called displacement monotonicity. The Lasry-Lions monotonicity condition is also assumed in Bayraktar-Cohen \cite{BC} and Bertucci-Lasry-Lions \cite{BLL}, which studied classical solutions for finite state mean field game master equations. We also observe that the weak monotonicity in Ahuja \cite{Ahuja} is exactly the displacement monotonicity.

Because of its infinite dimensionality nature, besides the monotonicity condition (in Lasry-Lions' sense or in displacement sense), all the above global wellposedness results require very strong regularity assumptions on data. Relaxing these assumptions to study wellposedness remains largely open. There are two directions to relax the assumptions: one is to remove the  monotonicity condition and the other is to weaken the regularity assumptions on data.  The goal of this paper is in the second direction. To our best knowledge, this paper is the first work which establishes the global wellposedness of the master equation without requiring smooth data. Before discussing our paper, let's review several important progresses made on mean field games without the monotonicity condition. It will be very interesting to combine the ideas of our paper and these works and we shall leave that for future research.

The monotonicity condition is to guarantee the uniqueness of the mean field equilibrium, and then the game value at this unique equilibrium is the (candidate) solution to the master equation. A mean field game is associated with a mean field game system, a forward backward system in (stochastic) PDE form or equivalently in McKean-Vlasov SDE form. Unfortunately, this forward backward system is typically degenerate, and together with other technical conditions, the monotonicity condition ensures this degenerate system has a unique solution. It is well understood in the PDE literature that the corresponding non-degenerate system would have a unique solution, without requiring certain monotonicity condition. This is true in the mean field case as well.  Foguen Tchuendom \cite{F} studied a special one dimensional linear quadratic mean field game with common noise, where the data depend on the law of the state process only through its mean. Since the mean is one dimensional, the common noise exactly makes the problem non-degenerate and the mean field equilibrium unique. In this special case, the variable of measure is reduced to the one dimensional variable for the mean and the master equation is reduced to a standard PDE, see Delarue-Foguen Tchuendom \cite{DF}. \cite{F} also showed that, when there is no common noise and thus the system is degenerate, the game can indeed have multiple mean field equilibria. For the general case, since the measure is essentially infinitely dimensional, Delarue  \cite{Delarue1} introduced an infinite dimensional common noise to make the reformulated problem  non-degenerate and thus restored the uniqueness of mean field equilibria. In this case the master equation becomes an infinitely dimensional system of infinitely dimensional PDE, and its mild solution is studied in \cite{Delarue1}. Moreover, recently Bayraktar-Cecchin-Cohen-Delarue \cite{BCCD}  applied this approach to a finite state mean field game.

When neither the monotonicity condition nor the non-degeneracy is satisfied, the mean field game could have multiple equilibria, as shown in \cite{F}. In this case, one approach is to fix a special type of equilibria and then study its existence and properties.  The works Delarue-Foguen Tchuendom \cite{DF}, Cecchin-Dai Pra-Fisher-Pelino \cite{CDFP} and Cecchin-Delarue \cite{CecchinDelarue} are in this direction. A larger literature is on the possible convergence of the equilibria for the $N$-player game, which is quite often unique because the corresponding Nash system is non-degenerate due to the presence of the individual noises, to the mean field equilibria (which may or may not be unique), see, e.g., Cardaliaguet-Delarue-Lasry-Lions \cite{CDLL}, Carmona-Delarue \cite{CD1,CD2}, Delarue-Lacker-Ramanan \cite{DLR1,DLR2}, Lacker \cite{Lacker0,Lacker1,Lacker2,Lacker3}, Nutz-San Martin-Tan \cite{NST}, to mention a few. Finally, we note that the ongoing work Iseri-Zhang \cite{IZ} takes a quite different approach by investigating the set of game values over all mean field equilibria and establishes two main properties of the set value for mean field games: (i)  the dynamic programming principle; and (ii) the convergence of the $N$-player game set value to the mean field game set value.

\subsection{The main results and contributions of this paper}
As mentioned, the main goal of this paper is to establish global wellposedness for the master equation \reff{masterint}  with non-smooth data, while keeping the Lasry-Lions monotonicity condition. In particular, we will not require any differentiability in $\mu$, but only certain Lipschitz continuity. We emphasize that, due to the infinite dimensionality of the Wasserstein space of measures, the Lipschitz continuity is much weaker than the continuous differentiability, and thus is much more likely to hold in applications. Consequently, under such mild regularity conditions, one cannot expect classical solutions to the master equation, see  Example \ref{eg-nonclassical} below. We shall propose three weaker notions of solutions, all of them are required only to be Lipschitz continuous in $\mu$, and establish their global wellposedness. To our best knowledge, this is the first (global) wellposedness result  in the literature for master equations with non-smooth data. Moreover, other than slight different requirements on the regularity in $x$, our three notions are all equivalent. We shall remark that the master equation \reff{masterint} is non-local (in space), because the term $\pa_x V(t, \tilde \xi, \mu)$ in $\cM V$ involves the values $\pa_x V(t, \tilde x, \mu)$ for all $\tilde x$ in the support of $\mu$. As a consequence, while we have global wellposedness (existence, uniqueness, and stability), even classical solutions to the master equation typically do not satisfy the comparison principle, see Example \ref{eg-comparison} below for a counterexample, consistent with the fact that comparison principle typically fails for the values of non-zero sum games (c.f. Feinstein-Rudloff-Zhang \cite{FRZ}). So the viscosity solution approaches in Gangbo-Swiech \cite{GS1}, Gangbo-Tudorascu \cite{GT}, Pham-Wei \cite{PW},  and Wu-Zhang \cite{WZ}  for HJB equations on Wasserstein space (and slightly more general parabolic master equations in \cite{WZ}), where the comparison principle is a main task, do not work here. We believe this is the main reason that a good notion of weak solutions for master equations was open in the literature.   

Our approach for the global wellposedness of master equations relies heavily on the a priori estimate for the uniform Lipschitz continuity in $\mu$ of the solution $V$. Note that $V$ serves as the decoupling field for the closely related  forward backward mean field game system.  In the literature of standard FBSDEs, it has been well understood that the global wellposedness of the FBSDE is essentially equivalent to the uniform Lipschitz continuity of the decoupling field, see Delarue \cite{Delarue}, Zhang \cite{Zhang-FBSDE}, Ma-Yin-Zhang \cite{MYZ}, and Ma-Wu-Zhang-Zhang \cite{MWZZ}. Indeed, this Lipschitz continuity allows us to extend a local solution, which is much easier to obtain, to a global one. This strategy remains effective for master equations, see e.g. \cite{CCD, CD2} in the realm of classical solutions. We shall establish this uniform estimate, as well as the stability result, under conditions much weaker than those in the literature. While following the same spirit as in  \cite{Cardaliaguet,CDLL} which use PDE arguments, we shall use probabilistic arguments by utilizing the forward backward McKean-Vlasov SDEs as in \cite{CCD}{\footnote {We note though this FBSDE is different from the one in \cite{CD2} derived from the stochastic maximum principle.}}. Unlike the existing works, our arguments do not require the differentiability of the data or $V$ in $\mu$, which is particularly convenient for our purpose. We note that the monotonicity condition is crucial for deriving the uniform estimate here, which in turn implies the existence of global solutions (under our new notions) to the master equation. However, we emphasize that the monotonicity condition is not needed for the uniqueness in the class of Lipschitz continuous solutions. In other words, any alternative conditions such as the displacement monotonicity in \cite{GM} which could lead to this uniform Lipschitz continuity will also ensure the global wellposedness of the master equation. We shall explore this further in our future research. 

Our conditions are not sufficient even for local classical solutions. To facilitate our analysis, we shall introduce a smooth mollifier for continuous functions on Wasserstein space. Note that the Wasserstein space is infinitely dimensional, this mollification is by no means easy. A work in this direction is  Lasry-Lions \cite{LL0},  which used explicit inf-sup-convolution to approximate uniformly continuous functions on Hilbert spaces with $C^{1,1}$ functions (that is, the gradient is Lipschitz continuous). This result was extended further by Cepedello Boiso \cite{Cepedello,Cepedello1} to any superreflexive Banach space. Note that the Wasserstein space of measures can be lifted to the Hilbert space of square integrable random variables as in \cite{Lions}, so one can apply this regularization to our data $F$ and $G$. However, even in finitely dimensional case, the inf-sup-convolution does not ensure regularity beyond $C^{1,1}$. This is not sufficient for our purpose, for example when we need a local classical solution for the master equation with mollified data $(F_n, G_n, H_n)$.  Therefore, we have to come up with a new smooth mollifier. Our idea is to first discretize the underlying measure and then to mollify the coefficients of the involved Dirac measures. Our mollifier is infinitely differentiable and approximates the original function uniformly. More importantly, for Lipschitz continuous functions under the $1$-Wasserstein distance $\mathcal{W}_1$, the mollified functions are uniformly Lipschitz continuous under $\cW_1$ with a common Lipschitz constant. This property is crucial for the uniform estimate of the Lipschitz continuity of $V$ in $\mu$ mentioned in the previous paragraph. We shall remark though that the above  property fails if we replace the metric $\cW_1$ with the $2$-Wasserstein distance $\mathcal{W}_2$.  Nevertheless, although slightly less natural than $\cW_2$, the uniform regularity under $\cW_1$ serves our purpose well. We would also like to point out that our mollifier does not inherit the monotonicity condition. In fact, we doubt any reasonable mollifier could inherit that.  

We now explain in more details the three notions of solutions we propose, which we call {\it good solution}, {\it weak solution}, and {\it weak-viscosity solution}, respectively. As mentioned, they are all required only to be Lipschitz continuous in $\mu$. Moreover,  a good solution is required to be Lipschitz continuous in $x$, a weak solution is continuously differentiable in $x$, and a weak-viscosity solution is such that $\pa_x V$ is also uniformly Lipschitz continuous in $(x, \mu)$. When they have the desired regularity in $x$, all three notions are equivalent. More importantly, under our mild technical conditions and the monotonicity condition, the master equation \reff{masterint} has a unique global solution in all three senses, and the stability result also holds. We remark that the monotonicity condition is not needed for the local wellposedness. 

The notion of {\it good solution} is based on the stability argument, and we borrow the name from Jensen-Kocan-Swiech \cite{JKS} which studies fully nonlinear elliptic PDEs.  Roughly speaking, we first mollify the data to obtain smooth $(F_n, G_n, H_n)$, then consider the classical solution $V_n$ for the master equation with smooth data $(F_n, G_n, H_n)$, and finally define the good solution as the (unique) limit of $V_n$ which converges due to the stability result.  However, since the mollified data do not inherit the monotonicity condition, we are not able to obtain a global classical solution $V_n$, but only a local one which does not require the monotonicity condition. So our good solution is also first defined locally and then extended to a global one, thanks to the uniform Lipschitz continuity we will achieve.

The notion of {\it weak solution} is in the spirit of the integration by parts formula, applied to the mean field game system (MFG system):
\bea
\label{FSPDEint}
&\left.\ba{c}
\dis d\rho(t,x) = \big[\frac{\beta_1^2+\b^2}{2}\tr\big( \pa_{xx} \rho(t,x)\big) - div(\rho(t,x) \pa_p H(x,\pa_x u(t,x)))\big]dt \\
\dis- \b\partial_x\rho(t,x)\cd d B_t^0,\qq \rho(0,x) = \rho_0(x);
\ea\right.\\
\label{BSPDEint}
&\left.\ba{c}
\dis d u(t, x)=  - \big[\tr\big(\frac{\b_1^2+\beta^2}{2} \pa_{xx} u(t,x) +\b\partial_x v(t,x)\big) +H(x,\pa_x u(t,x)) + F(x, \rho_t)\big]dt\\
\dis + v(t,x)\cd dB_t^0,\qq u(T,x) = G(x, \rho_T).
\ea\right.
\eea
Here the Brownian motion $B^0$ is the common noise, \reff{FSPDEint} is the stochastic Fokker-Planck equation with solution $\rho$ in the sense of distribution, and \reff{BSPDEint} is the stochastic HJB equation with $\dbF^{B^0}$-progressively measurable solution pair $(u, v)$. Given the (candidate) solution $V$ to the master equation \reff{masterint}, we may decouple the MFG system by replacing \reff{FSPDEint} with:
\bea
\label{FSPDEint2}
\left.\ba{c}
\dis d\rho(t,x) = \big[\frac{\beta_1^2+\b^2}{2}\tr\big( \pa_{xx} \rho(t,x)\big) - div(\rho(t,x) \pa_p H(x,\pa_x V(t,x, \rho_t)))\big]dt \\
\dis- \b\partial_x\rho(t,x)\cd d B_t^0,\qq \rho(0,x) = \rho_0(x).
\ea\right.
\eea
We see particularly the involvement of $\pa_x V$ in \reff{FSPDEint} and thus we shall require its existence for weak solutions. We may define weak solutions to forward SPDEs and backward SPDEs  by standard integration by parts formula, in particular, we refer to Ma-Yin-Zhang \cite{MYZ} and Qiu \cite{Qiu} for weak solutions to BSPDEs. Then we call $V$ a weak solution to the master equation \reff{masterint} if, for any weak solution $\rho$ to \reff{FSPDEint2}, $u(t, x):= V(t,x, \rho_t)$ is a weak solution to \reff{BSPDEint}. We shall point out that this notion is different from the weak solution for the MFG system \reff{FSPDEint}-\reff{BSPDEint}  in  Porretta \cite{Porretta} and Cardaliaguet-Graber-Porretta-Tonon \cite{CGPT}. These two works consider the local coupling case: $F = F(x, \rho(t,x)), G= G(x, \rho(T,x))$, without the common noise ($\b=0$), and thus the MFG system becomes a forward backward system of standard PDEs. A weak solution to the MFG system is a pair $(\rho, u)$ such that $\rho$ is a weak solution to PDE \reff{FSPDEint} (again with  $\b=0$) for given $u$, and  $u$  is a weak solution to PDE \reff{BSPDEint} for given $\rho$. However, this has fundamental difference with our notion of weak solution to the master equation. Besides the obvious difference on the coupling of $\rho$ in $F, G$ and many other technical differences, we note that in \reff{FSPDEint2}  the term $\pa_x V$ depends on the solution $\rho$, while in \reff{FSPDEint} the term $\pa_x u$ is fixed. So the uniqueness of $V$ has different nature from the uniqueness of $(\rho, u)$ in \reff{FSPDEint}-\reff{BSPDEint}.

The notion of {\it weak-viscosity solution} again considers the decoupled MFG system \reff{FSPDEint2}-\reff{BSPDEint}. We say $V$ is a weak-viscosity solution to the master equation \reff{masterint} if, for any weak solution $\rho$ to \reff{FSPDEint2}, $u(t, x):= V(t,x, \rho_t)$ is a viscosity solution to BSPDE \reff{BSPDEint}. We note that, when there is no common noise ($\b=0$) but with $\b_1>0$, \reff{BSPDEint} becomes a standard non-degenerate parabolic PDE, which has a unique classical solution under very mild conditions. The classical solution for the BSPDE \eqref{BSPDEint}, with $\beta>0$, is much harder to obtain. Besides the weak solution approach for BSPDE  \reff{BSPDEint}, we may also treat it in a pathwise manner by viewing it as a path dependent PDE (PPDE). We shall adopt the viscosity solution approach for PPDEs developed by Ekren-Keller-Touzi-Zhang \cite{EKTZ}, Ekren-Touzi-Zhang \cite{ETZ1, ETZ2}, and Ren-Touzi-Zhang \cite{RTZ}. This approach requires certain pathwise regularity  in terms of  the paths $\o$ of $B^0$. For this purpose we need to require the Lipschitz continuity of $\pa_x V$ so that the solution $\rho$ to the FSPDE \reff{FSPDEint2} would have the desired regularity in $\o$. We emphasize again that, while for fixed $\rho$ the viscosity solution $u$ to the BSPDE \reff{BSPDEint} would have the desired comparison principle, the weak-viscosity solution $V$ to the master equation \reff{masterint} typically does not satisfy the comparison principle. We remark that one can also consider pathwise viscosity solution to the FSPDE \reff{FSPDEint2}, initiated by Lions-Sounganidis \cite{LS1, LS2} and see Buckdahn-Keller-Ma-Zhang \cite{BKMZ} and the references therein. However, unlike that  $u$  stands for the utility of the individual player, the $\rho$ for the FSPDE stands for the environment or say the collective states of (infinitely many) other players, thus it is more appropriate to take the global (in space) approach by considering the weak solution for $\rho$.  We would like to mention that \cite[Section 4.4.3]{CD2} also proposed a notion of viscosity solution for the master equation \eqref{masterint} following the standard approach of Crandall-Ishii-Lions \cite{CIL}.  However, due to the nonlocal feature of \reff{masterint}, the uniqueness of their viscosity solution is not clear (not to mention the comparison principle which we know is not true in general). Moreover, a very recent paper Bertucci \cite{Bertucci} proposed a notion of monotone solution for finite state space master equations, which is in the spirit of viscosity solutions, and established wellposedness under certain monotonicity condition. 

As an important application of our theory, we prove the convergence of the Nash system and the propagation of chaos for the $N$-player game, thus extend the corresponding result in \cite{CDLL, CD2} to our setting. The approach in \cite{CDLL, CD2},  even for the case without the common noise,  relies heavily on the boundedness of the second order derivatives, especially $\pa_{\mu\mu} V$, which is exactly what we want to avoid. We shall follow our approach for the global wellposedness of the master equation, namely we first establish the local (in time) convergence, and then extend it to the whole time interval. To our best knowledge, this approach is new for such a convergence in the literature. Without surprise, the uniform Lipschitz continuity plays a key role for this extension. We remark that another crucial property for our approach to work is the flow property of the system. In the literature, people typically consider the $N$-player game starting with i.i.d. random variables. This independence will be destroyed immediately when time evolves due to the interaction among the particles, and thus one cannot apply a local convergence result for i.i.d. initials to the system on a later interval. We shall instead study the $N$-player game starting with deterministic initials, which can be viewed as a conditional version of the standard system with i.i.d. initials. The convergence of the latter system will be obtained easily after we establish the convergence of the former system. However, we should point out that, due to some technical reasons, in this section we assume the Hamiltonian $H$ is uniform Lipschitz continuous. This unfortunately excludes the case that $H$ is quadratic in $\pa_x V$, which is studied in \cite{CD2} by using the classical solution approach (\cite{CDLL} also assumes the uniform Lipschitz continuity of $H$). We shall explore the convergence of this interesting case in our future study.

Finally, as an independent result, we provide a pointwise representation formula for the Wasserstein derivatives $\pa_\mu V(t,x,\mu, \tilde x)$ and $\pa_{\mu\mu}V(t,x,\mu, \tilde x, \bar x)$ through strong solutions of certain McKean-Vlasov FBSDEs, provided these FBSDEs are wellposed. We believe our formulas are new and are interesting in their own rights. In particular, our arguments provide an alternative approach for the existence of classical solutions for the master equation and, although not carried out in details in this paper, our arguments allow us to see the "minimum" technical conditions we will need to ensure the existence and continuity of these derivatives and hence to ensure the existence of classical solutions.  We note that \cite[Corollary 3.9]{CDLL}  also provided a pointwise representation formula for the gradient ${\d V\over \d \mu}(t,x,\mu, \tilde x)$.  Since  $\pa_\mu V (t,x,\mu, \tilde x) = \pa_{\tilde x} {\d V\over \d \mu}(t,x,\mu, \tilde x)$, so \cite{CDLL} implies a representation formula for $\pa_\mu V (t,x,\mu, \tilde x)$ as well, by involving a FBSPDE system whose initial value is the derivative of the Dirac measure. However, the connection between these two formulas is not clear to us. 

\ms
\no{\bf Connection with the earlier version of the paper: arXiv:1903.09907v1} (referred to as "the early version"). The early version has been circulated in the community for about one and a half years. It deals with a much simpler setting and we would like to refer readers who are only interested in the main ideas of our approach to the early version. We have made significant expansion in this version (the length of the paper is more than doubled). For the convenience of the readers who read the early version before, we list here a few main changes we made in this version.

\begin{itemize}
\item{} We extend the state space from $\dbT$ to $\dbR^d$, consider a general Hamiltonian $H$, and add the common noise.

\item{} For the crucial Lipschitz continuity estimates, we change from PDE arguments to probabilistic arguments, which seem more convenient to us.

\item{} The good solution and weak solution were called vanishing weak solution and Sobolev solution, respectively, in the early version, and we have improved their definition.  In particular, we do not need to require the differentiability in $\mu$ for the weak solution. The weak-viscosity solution is new in this version.

\item{} We add the whole section on the convergence of the Nash system.

\item{} We simplify the representation formula for $\pa_\mu V$ and add the representation formula for $\pa_{\mu\mu} V$.

\item{} We add a few examples in Appendix to illustrate some subtle points.

\end{itemize}

The rest of the paper is organized as follows. In Section \ref{sect-game} we introduce the mean field game and N-player game and their associated master equation and Nash system, in an heuristic way, and exhibit all the main results in the paper. In Section \ref{sect-mollifier} we construct a smooth mollifier for functions of probability measures. Section \ref{sect-regularity} is devoted to the uniform regularity of the value function and the stability result. In Sections \ref{sect-good}, \ref{sect-weak} and \ref{sect-viscosity} we propose good, weak and weak-viscosity solutions for our master equation and establish their wellposedness and equivalency. In Section \ref{sect-convergence} we establish wellposedness of classical solutions for our Nash system, and show various convergence results from the N-player game to the mean field game. In Section \ref{sect-representation} we provide pointwise probabilistic representation formulas for $\pa_\mu V, \pa_{\mu\mu} V$.  Finally, in Section \ref{sect-proof} we finish some technical proofs which were postponed in the previous sections.

\subsection{Some notations used in the paper}

For any $p\geq 1$ and $M,R\geq 0$, we introduce some notations used throughout the paper:
\begin{itemize}
\item $\Th:= [0, T]\times \dbR^d \times \cP_2$;
\item $D_R:=\Big\{(x,z)\in \mathbb R^{d\times 2}\,\,:\,\,|z|\leq R\Big\}$;
\item $Q_M:= \{x\in \dbR^d: |x_l|\le M, l=1,\cds,d\}$.
\item $\mathcal{P}_p:=\Big\{\mu\in\mathcal{P}:~\|\mu\|_p:=\Big( \int_{\mathbb R^d}|x|^p\mu(dx)\Big)^{\frac{1}{p}}< \infty\Big\}$ and $\mathcal{P}_p^M:=\Big\{\mu\in\mathcal{P}:~\|\mu\|_p\leq M\Big\}$;
\item $C^0(\mathcal{P}_p):=\Big\{U:\mathcal{P}_p\to \mathbb R\,\,:\,\,U\text{ is continuous in $\mathcal{P}_p$ (under $\mathcal{W}_p$)}\Big\}$;
\item $C^1(\mathcal{P}_2):=\Big\{U\in C^0(\mathcal{P}_2)\,\,:\,\,\pa_\mu U\text{ exists and is continuous on $\cP_2\times \dbR^d$}\Big\}$;
\item $C^{1,2,2}(\Th):=\Big\{U\in C^0( \Th; \dbR)\,\,:\,\,\pa_t U,\,\, \pa_x U,\,\, \pa_{xx} U,\,\, \pa_\mu U(t,x,\mu, \tilde x),\,\, \pa_{x}  \pa_\mu U(t,x,\mu, \tilde x)$, $\pa_{\tilde x}  \pa_\mu U(t,x,\mu, \tilde x)\,\, \text{and}\,\, \pa_{\mu\mu} U(t,x,\mu, \bar x, \tilde x)\text{ exist and are continuous}\Big\}$;
\item  $C^0_{Lip}(\dbR^d \times \cP_2):=\Big\{U:\mathbb R^d\times\mathcal{P}_2\to\mathbb R\,\,:\,\,U\text{ is uniform Lipschitz continuous}, \text{under $\mathcal{W}_1$}$ $\text{for $\mu$}\Big\}$;
\item $C^0_{Lip}(\Th):=\Big\{U\in C^0(\Th)\,\,:\,\,U\text{ is uniformly Lipschitz continuous in $(x,\mu)$, under $\mathcal{W}_1$}$ $\text{for $\mu$, uniformly in $t\in[0,T]$}\Big\}$;
\item $C^{0,1-}(\Th):=\Big\{U\in C^0_{Lip}(\Th):\pa_x U\text{ exists and is continuous in $(x,\mu)$ for all $t\in$}$ $[0,T]\Big\}$; 
\item $C^{0,2-}(\Th):=\Big\{U \in C^{0,1-}(\Th)\,\,:\,\,\pa_x U\in C^{0}_{Lip}(\Th) \Big\}$;
\item $\dbL^p(\cG):=\!\!\Big\{\xi:\Omega\to \mathbb R^d: \xi \text{ is $\cG$-measurable and $\mathbb E\big[|\xi|^p\big]<\infty$}\Big\}$ for any $\sigma$-algebra $\cG$ of $\Omega$;
\item $\dbL^p(\cG,\mu):=\Big\{\xi\in\dbL^p(\cG)\,:\, \mathcal{L}_{\xi}=\mu\Big\}$ for any $\mu\in\mathcal{P}$;
\item  $\wh L(x, z) := L(x, \pa_p H(x,z)) = z\cd \pa_p H(x,z) - H(x,z)$.
\end{itemize}

\section{Preliminaries and the main results}
\label{sect-game}
\setcounter{equation}{0}

We start with the basic setting in Wasserstein space. Let $[0, T]$ be a finite time horizon, and $\cP$  the set of all probability measures on $\mathbb R^d$. In particular,  $\d_x\in \cP$ denotes the Dirac-measure at  $x\in \mathbb R^d$. For any $p\geq 1$, $M\geq 0$, and any measure $\mu\in \cP$, denote
\bea
\label{cP}
\|\mu\|^p_p:= \int_{\mathbb R^d} |x|^p \mu(dx),~ \mathcal{P}_p:=\{\mu\in\mathcal{P}:~\|\mu\|_p < \infty\},~
\mathcal{P}_p^M:=\{\mu\in\mathcal{P}:~\|\mu\|_p  \le M\}.
\eea
Introduce the $p$-Wasserstein distance on $\mathcal{P}_p$: for any $\mu,\nu\in \cP_p$,
\bea
\label{Wp}
\cW_p(\mu, \nu) := \inf\Big\{\big(\dbE[|\xi-\eta|^p]\big)^{1\over p}:~\mbox{for all r.v. $\xi$, $\eta$ such that $\cL_\xi = \mu$, $\cL_\eta = \nu$}\Big\}.
\eea
At above $\xi, \eta$ are $\mathbb R^d$-valued random variables on arbitrary probability space and $\cL_\cd$ is the law of the random variable.  In particular, when $p=1$ we have the dual representation:
\bea
\label{W1}
\cW_1(\mu,\nu)=\sup \Big\{\int_{\mathbb R^d}\f(x)[\mu(dx)-\nu(dx)]: \mbox{$\f\in C^1(\mathbb R^d; \dbR)$ s.t. $\f(0)=0$, $|\pa_x\f|\le 1$}\Big\}.
\eea

Consider a function $U: \cP_2 \to \dbR$. By \cite{Lions, Cardaliaguet, WZ0}, the derivative of $U$ takes the form $\pa_\mu U:  \cP_2\times \mathbb R^d \to \dbR^d$ satisfying: for all $\dbR^d$-valued square integrable random variables $\xi, \eta$, 
\bea
\label{pamu}
U(\cL_{\xi +  \eta}) - U(\mu) = \dbE\big[\pa_\mu U(\mu, \xi)\cd \eta \big] + o(\|\eta\|_2),\q \mbox{where}\q \mu:= \cL_\xi.
\eea
Let $C^0(\cP_2)$ denote the set of continuous functions $U: \cP_2 \to \dbR$, and $C^1(\cP_2)$ the subset of $U\in C^0(\cP_2)$ such that  $\pa_\mu U$ exists and is continuous on $\cP_2\times \dbR^d$. Given $U\in C^1(\cP_2)$, we may define  $\pa_x\pa_\mu U$, $\pa_{\mu\mu}  U$, and higher order derivatives in the same manner. Moreover, denote
\bea
\label{Th}
\Th:= [0, T]\times \dbR^d \times \cP_2.
\eea
Let $C^{1,2,2}(\Th)$ denote the set of  $U\in C^0( \Th; \dbR)$ such that $\pa_t U$, $\pa_x U$, $\pa_{xx} U$, $\pa_\mu U(t,x,\mu, \tilde x)$, $\pa_{x}  \pa_\mu U(t,x,\mu, \tilde x)$, $\pa_{\tilde x}  \pa_\mu U(t,x,\mu, \tilde x)$ and $ \pa_{\mu\mu} U(t,x,\mu, \bar x, \tilde x)$ exist and are continuous. 
\begin{rem}
\label{rem-compact}
{\rm
For fixed $\mu\in \cP_2$,  by \reff{pamu} $\pa_\mu U(\mu, \cd)$ is unique $\mu$-a.s. However, for $U\in C^1(\cP_2)$,  $\pa_\mu U(\mu, x)$ is unique for all $(\mu, x)$. We note that our notion of $C^1(\cP_2)$ requires continuity in pointwise sense, which is stronger than the continuity in $\dbL^2$-sense required in some existing works, see e.g.  \cite{CDLL}.
\qed}
\end{rem}

From now on, we fix a  filtered probability space $(\O,\cF, \dbF, \dbP)$, on which are defined independent $d$-dimensional Brownian motions $B$ and $B^0$. We assume $\cF_0$ is rich enough to support any $\mu\in \cP_2$. Denote $\cF^0_t := \cF^{B^0}_t$ and $\cF_t := \cF_0 \vee \cF^B_t \vee\cF^0_t$. For any $p\ge 1$, $\cG\subset \cF$, and $\mu\in \cP_p$, denote by $\dbL^p(\cG)$ the set of $\dbR^d$-valued, $\cG$-measurable, and $p$-integrable random variables $\xi$; and $\dbL^p(\cG;\mu)$ the set of those $\xi\in \dbL^p(\cG)$ with $\cL_\xi=\mu$. Throughout the paper, given $\xi \in \dbL^p(\cF_t)$, we use $\tilde \xi, \bar \xi$ etc to denote conditionally independent copies of $\xi$ (by possibly extending to product sample space), conditional on $\cF^0_t$, and $\tilde \dbE, \bar \dbE$ are the conditional expectations which integrate only on $\tilde \xi, \bar\xi$, respectively, conditional on $\cF^0_t$. Moreover, we fix a constant $\b \ge 0$. 

One crucial property of $U\in C^{1,2,2}(\Th)$ is the It\^{o}  formula. For $i=1,2$, let  $d X^i_t= b^i_t dt + \si^i_t dB_t + \si^{i,0}_t d B^0_t$, where $b^i: [0, T]\times \O\to \dbR^d$ and $\si^i, \si^{i,0}: [0, T]\times \O\to \dbR^{d\times d}$ are $\dbF$-progressively measurable and bounded (for simplicity), then (cf., e.g.,  \cite[Theorem 4.17]{CD2}):  denoting $\rho_t:=\mathcal{L}_{X^2_t|\cF^{0}_t}$ as the conditional law,
\bea
\label{Ito}
&&d U(t, X^1_t, \rho_t) =  \Big[\pa_t U + \pa_x U\cdot b^1_t + \frac{1}{2} \tr\big(\pa_{xx} U [\si^1 (\si^1)^\top + \si^{1,0}(\si^{1,0})^\top]\big)\Big](t, X^1_t, \rho_t) dt \nonumber\\
&&+\tr\Big( \tilde \dbE_{\cF^0_t}\Big[\pa_\mu U(t,X^1_t,\rho_t,\tilde X^2_t) (\tilde b^{2}_t)^\top + \pa_x\pa_\mu U(t,X^1_t,\rho_t,\tilde X^2_t)\si^{1,0}_t (\tilde \si^{2,0})^\top\nonumber\\
&&+ \frac{1}{2} \pa_{\tilde x}\pa_\mu U(t, X^1_t, \rho_t, \tilde X^2_t)[\tilde \si^2 (\tilde \si^2)^\top + \tilde \si^{2,0}(\tilde \si^{2,0})^\top]\Big]\Big)dt\\
&& +\frac{1}{2}\tr\Big((\tilde\dbE\times\bar\dbE)_{\cF^0_t}\big[\pa_{\mu\mu}U(t,X^1_t,\rho_t,\tilde X^2_t,\bar X^2_t) \tilde\si^{2,0}(\bar \si^{2,0})^\top\big]\Big]\Big)dt+\pa_xU(t,X^1_t,\rho_t)\cd\si^1_tdB_t\nonumber\\
&&+\Big[(\si^{1,0}_t)^\top\pa_xU(t,X^1_t,\rho_t)+\tilde{\mathbb E}\big[(\tilde \si^{2,0}_t)^\top\pa_\mu U(t,X^1_t,\rho_t,\tilde X^2_t) \big]\Big]\cd dB_t^0.\nonumber
\eea
Throughout this paper, the elements of $\dbR^d$ are viewed as column vectors; $\pa_x U, \pa_{\mu} U$ are also column vectors; $\pa_x \pa_\mu U := \pa_x \big[(\pa_\mu U)^\top\big]\in \dbR^{d\times d}$, where $^\top$ denotes the transpose, and similarly for the other second order derivatives; The notation $\cd$ denotes the inner product of column vectors. Moreover, the term $\pa_x U \cd \si^1_t dB_t$ means $\pa_x U \cd (\si^1_t dB_t)$, but we omit the parentheses for notational simplicity.

\subsection{The master equation}
We first introduce the mean field game, whose value function will be characterized by the master equation. Given $t\in [0, T]$, denote $B^t_s:= B_s-B_t$, $B^{0,t}_s:=B_s^{0}-B_t^0$, $s\in [t, T]$, and let $\cA_t$ be the set of bounded and progressively measurable and adapted controls $\a: [t, T]\times C([t, T]; \dbR^{2d}) \to \dbR^d$.  For any $\xi\in \dbL^2(\cF_t)$ and $\a\in  \cA_t$, consider the following SDE:
\bea
\label{Xth}
\dis X^{t,\xi, \a}_s = \xi + \int_t^s  \a_r( X^{t,\xi, \a}_\cd, B^{0, t}_\cd) dr + B^t_s+\b B^{0,t}_s,\q s\in [t, T],
\eea
We note that, by the adaptedness, the control $\a$ actually takes the form  $\a_r( X^{t,\xi, \a}_{[t,r]}, B^{0, t}_{[t, r]})$. By Girsanov Theorem, the above SDE has a unique weak solution. Consider the conditionally expected utility for the mean field game:
\bea
\label{Jxi}
\left.\ba{c}
 J(t,x,\xi; \a,  \a'):= \dbE^{\dbP}_{\cF^0_t}\Big[G(X^{t, x, \a'}_T,  \cL_{X^{t,\xi,\a}_T|\cF^0_T}) \\
\dis + \int_t^T \big[F(X^{t,x, \a'}_s, \cL_{X^{t,\xi,\a}_s|\cF^0_s}) - L(X_s^{t,x,\alpha'},\a'_s (X^{t,x,\a'}, B^{0, t})) \big] ds\Big],
\ea\right.
\eea
where $L:\dbR^d\times\dbR^d\to \dbR$ and $F, G: \dbR^d\times \cP_2\to \dbR$ are measurable in all variables. Here $\xi$ denotes the initial state of the ``other" players, $\a$ is the common control of the other players, and $(x, \a')$ correspond to the initial state and control of the individual player. 

When $\xi\in \dbL^2(\cF_0\vee \cF^B_t)$ is independent of $\cF^0_t$,  it is clear that $J(t,x,\xi; \a,  \a')$ is deterministic and is law invariant, that is, if $\xi'\in \dbL^2(\cF_0\vee \cF^B_t)$ with $\cL_{\xi'}=\cL_\xi$, then $J(t,x,\xi'; \a,  \a')=J(t,x,\xi; \a,  \a')$ for any $x, \a,\a'$. Therefore, by abusing the notation $J$ we may introduce:
\bea
\label{Jmu}
J(t,x,\mu; \a,  \a') := J(t,x,\xi; \a,  \a'),\q \xi\in \dbL^2(\cF_0\vee \cF^B_t, \mu).
\eea
Now for any $ (t,x,\mu)\in \Th$ and $\a\in \cA_t$,  we consider the following optimization problem:
\bea
\label{Va0}
\dis V( t,x,\mu;\a) := \sup_{\a'\in \cA_t}  J(t,x,\mu; \a, \a').
\eea

\begin{defn}
\label{defn-NE}
We say $\a^*\in \cA_t$ is a mean field equilibrium (MFE) of \reff{Va0} at $(t,\mu)$ if 
\beaa
V( t,x,\mu; \a^*) = J(t,x,\mu; \a^*, \a^*) ~\mbox{for $\mu$-a.e. $x\in \dbR^d$}.
\eeaa
\end{defn}
We remark that an MFE relies on $(t,\mu)$, but is universal for all $x$. When there is a unique MFE  for each $(t,\mu)$, denoted as $\a^*(t,\mu)$, then clearly the game problem leads to a value function:
\bea
\label{V}
V(t,x,\mu) := V( t,x,\mu; \a^*(t,\mu)).
\eea

Introduce the Hamiltonian $H$ corresponding to the Lagrangian $L$:
\bea
\label{H}
H(x, z) := \sup_{a\in \dbR^d} [  a\cd z - L(x, a) ],\q x, z\in \dbR^d.
\eea 
In light of the It\^{o} formula \reff{Ito}, the value function $V$ in \reff{V} is associated with the following master equation:
\bea
\label{master}
\left.\ba{c}
\dis \cL V(t,x,\mu) :=\pa_t V + \frac{\h\b^2}{2} \tr(\pa_{xx} V) + H(x,\partial_x V) + F(x,\mu) + \cM V =0,\\
\dis V(T,x,\mu) = G(x,\mu),\q\mbox{where}\\
\dis \cM V(t,x,\mu)  := \tr\Big(\tilde  \dbE\Big[\frac{\h\b^2}{2} \pa_{\tilde x} \pa_\mu V(t,x, \mu, \tilde \xi) + \pa_\mu V(t, x, \mu, \tilde \xi)(\pa_pH)^\top(\tilde \xi, \pa_x V(t, \tilde \xi, \mu))\\
\dis  +\b^2\pa_x\pa_\mu V(t,x,\mu,\tilde \xi)+\frac{\b^2}{2}\bar{\mathbb E}\big[\pa_{\mu\mu}V(t,x,\mu,\bar\xi,\tilde\xi)\big]\Big]\Big),\q\mbox{and}\q \h \b^2 := 1+\b^2.
\ea\right.
\eea
Here the  term $\pa_p H$ is the derivative with respect to $z$, so it is also natural to denote it as $\pa_z H$, but nevertheless we use $\pa_p H$ as in standard PDE literature. 

On the opposite direction, assume  the data $F$, $G$, $L$, $H$ satisfy appropriate technical conditions and the master equation \reff{master} has a classical solution $V\in C^{1,2,2}(\Th)$. Then, for any $(t,\mu)\in [0, T]\times \cP_2$, the following $\a^*$ is an MFE at $(t, \mu)$ (see e.g. \cite{CD1}):
\bea
\label{alpha*}
\left.\ba{c}
\dis \a^*_s(\bx, B^{0,t}) :=  \pa_pH\big(\bx_s,\pa_x V(s,\bx_s,\cL_{X^*_s|\cF^0_s})\big),\q\bx\in C([t, T]; \dbR^d);\\
\dis \mbox{where}\q X^*_s = \xi + \int_t^s \pa_pH\big(X^*_r,\pa_x V(r, X^*_r,\cL_{X^*_r|\cF^0_r})\big) dr + B^t_s + \b B^{0,t}_s,\q s\in [t, T].
 \ea\right.
 \eea
However, we note that it requires very strong technical conditions on data in order to obtain a classical solution of the master equation \reff{master}. See Example \ref{eg-nonclassical} for a counterexample. Our goal of this paper is to investigate new weak notions of solutions and establish their wellposedness under mild regularity conditions.

\begin{rem}
\label{rem-global}
{\rm We emphasize that the term $\tilde  \dbE\Big[ \pa_\mu V(t, x, \mu, \tilde \xi)(\pa_pH)^\top(\tilde \xi, \pa_x V(t, \tilde \xi, \mu))\Big]$ in $\cM V$ of \reff{master} involves $\pa_x V(t, \tilde x, \mu)$ for $\mu$-a.e. $\tilde x$. That is,  the master equation \reff{master} is non-local in $x$. In particular, we cannot expect a comparison principle for its solution, even if there exists a unique classical solution. See Example \ref{eg-comparison} for a counterexample.
\qed}
\end{rem}

\begin{rem}
\label{rem-control}
{\rm The choice of admissible controls $\cA_t$ is actually very subtle, and the MFEs under different choices are in general not equivalent, 
see \cite{IZ} for more discussions.  However, we would like to point out  that in applications admissible controls should depend on the observed information. Note that players typically observe the state process $X$, and since $B^0$ is interpreted as the common noise, thus it is also reasonable to assume its observability (compared to the individual noises of the other players which are much harder to observe). So in this mean field setting one natural choice could be $\a = \a(X, B^0, \cL_{X|\cF^0})$. However, since $\cL_{X|\cF^0}$ is $\dbF^{0}$-measurable, then the above $\a$ is actually $\dbF^{X, B^0}$-measurable and for simplicity in this paper we take the form $\a=\a(X_t, B^0)$ as in \reff{Xth}. We emphasize that, however, for $N$-player games these two are not equivalent and it will be more natural to choose the counterpart of $\a = \a(X_t, B^0, \cL_{X_t|\cF^0})$, as we will do in the next subsection.
\qed}
\end{rem}

The master equation \reff{master} is associated with the following system of Forward Backward Stochastic PDEs (FBSPDEs ): given $t_0\in [0, T]$ and considering the equations on $[t_0, T]\times \dbR^d$,
\bea
\label{SPDE}
&&\dis \!\!\!\!\!\!\!\! d\rho(t,x) = \big[\frac{\h\beta^2}{2}\tr\big( \pa_{xx} \rho(t,x)\big) - div(\rho(t,x) \pa_p H(x,\pa_x u(t,x)))\big]dt-\b\partial_x\rho(t,x)\cd d B_t^0;\\
&&\dis \!\!\!\!\!\!\!\! d u(t, x)=  v(t,x)\!\cd\! dB_t^0\! -\! \big[\tr\big(\frac{\h\beta^2}{2} \pa_{xx} u(t,x) \!+\! \b\partial_x v^\top(t,x)\big) \!+\! H(x,\pa_x u(t,x)) \!+\! F(x, \rho(t,\cdot))\big]dt;\nonumber\\
&&\dis \!\!\!\!\!\!\!\! \rho(t_0, \cd) = \rho_0,\q u(T,x) = G(x, \rho(T,\cdot)),\nonumber
\eea
where $\rho_0:\O\to \cP_2$ is $\cF^0_{t_0}$-measurable. 
Here the first equation is a standard (forward) SPDE with solution $\rho$,   the second equation is a backward SPDE with solution pair $(u, v)$ taking values in $\dbR\times \dbR^d$, and  $\rho, u, v$ are all $\dbF^{0}$-progressively measurable, but we sometimes omit the variable $\o$. Moreover, $\rho_t =\rho_t(\o)= \rho(t,\cd, \o)$ is a (random) probability measure and when needed can be viewed as a weak solution to the SPDE:
\bea
\label{SPDEweak}
\left.\ba{c}
\dis d \int_{\mathbb R^d} \varphi(t,x)\rho(t,dx)= \b\int_{\mathbb R^d}\pa_x\varphi(t,x)\rho(t,dx)\cd dB_t^0\\
\dis +\int_{\mathbb R^d}\big[\partial_t\varphi(t,x)+\tr\big(\frac{\h\beta^2}{2}\pa_{xx}\varphi(t,x)\big)+\pa_x\varphi(t,x) \cd\pa_pH(x,\pa_x u(t,x))\big]\rho(t,dx)~dt,
\ea\right.
\eea
for any $\varphi\in C_c^{1,2}([0,T]\times\mathbb R^d)$.  Similarly, we may define the weak solution to the BSPDE:
\bea
\label{BSPDEweak}
\left.\ba{c}
\dis d \int_{\mathbb R^d} u(t,x) \varphi(t,x)dx= \int_{\mathbb R^d}\Big[u(t,x)\partial_t\varphi(t,x) +[\frac{\h\beta^2}{2} \pa_x u(t,x) + \b v(t,x)]\cd \pa_{x}\varphi(t,x)\\
\dis  -  [H(x,\pa_x u(t,x)) + F(x, \rho(t,\cdot))] \varphi(t,x) \Big]dx~dt + \int_{\mathbb R^d} v(t,x) \varphi(t,x)dx\cd dB_t^0.
\ea\right.
\eea
Then, provided the master equation \reff{master} has a classical solution $V$, we have the following relation for any fixed $(t_0, \rho_0)$:
\bea
\label{Vuv}
u(t,x,\o) = V(t, x, \rho_t(\o)). 
\eea

Alternatively, given $t_0$ and $\xi\in \dbL^2(\cF_{t_0})$, we may consider the following forward backward McKean-Vlasov SDEs on $[t_0, T]$: noting that $d B^{t_0}_t = dB_t$ and $d B^{0, t_0}_t= dB^0_t$, 
\bea
\label{FBSDE1}
\left.\ba{c}
\dis X^{\xi}_t=\xi + \int_{t_0}^t\pa_pH(X_s^\xi,Z_s^{\xi})ds+B^{t_0}_t+\b B_t^{0,t_0}; \\
\dis Y_t^{\xi}=G(X_T^{\xi},\rho_T)+\int_t^T[F(X_s^{\xi},\rho_s) -\wh L(X_s^{\xi},Z_s^{\xi})]ds- \int_t^TZ_s^{\xi}\cd dB_s-\int_t^TZ_s^{0,\xi}\cd dB_s^{0};\ms\\
\dis \mbox{where}\q \wh L(x, z) := L(x, \pa_p H(x,z)) = z\cd \pa_p H(x,z) - H(x,z),\q \rho_t := \rho^\xi_t:= \cL_{X_t^\xi|\cF^0_t}.
\ea\right.
\eea
Given the above $\rho$, we consider further the following standard decoupled FBSDE:
\bea
\label{FBSDE2}
&\dis X_t^{x}=x+ B^{t_0}_t+\b B_t^{0,t_0}; \\
&\dis Y_t^{x,\xi}=G(X_T^{x},\rho_T)+\!\! \int_t^T\!\!  [F(X_s^{x},\rho_s) + H(X_s^{x},Z_s^{x,\xi})]ds-\!\int_t^T\!\! Z_s^{x,\xi}\cd dB_s-\!\int_t^T\!\! Z_s^{0,x,\xi}\cd dB_s^{0}.\nonumber
\eea
The connection between \reff{SPDE} (hence \reff{master}) and \reff{FBSDE1}-\reff{FBSDE2} is that the $\rho_t$ in the two equations coincide and 
\bea
\label{uvYZ}
\left.\ba{c}
\dis Y^{\xi}_t = u(t, X^{\xi}_t, B^{0}), \q Z^{\xi}_t =  \pa_x u(t, X^{ \xi}_t, B^{0}),\q Z^{0,\xi}_t = [v + \b  \pa_x u](t, X^{ \xi}_t, B^{0}); \\
\dis Y^{x, \xi}_t = u(t, X^{ x}_t, B^{0}), \q Z^{x, \xi}_t =  \pa_x u(t, X^{ x}_t, B^{0}), \q Z^{0,x, \xi}_t = [v + \b  \pa_x u](t, X^{x}_t, B^{0}).
\ea\right.
\eea

Occasionally we may rewrite $X^x=X^{x,\xi}$ for notational consistency with $Y^{x,\xi}$ etc. When there is a need to emphasize the dependence on $t_0$, we will denote the solutions to \reff{FBSDE1}-\reff{FBSDE2} as $\Phi^{t_0,\xi}, \Phi^{t_0, x,\xi}$, $\Phi = X, Y, Z, Z^0$. 
We note that the decoupled FBSDE \reff{FBSDE2} can be replaced with the following coupled FBSDE which seems natural but is harder to analyze (please notice the notation $\Phi^{\xi,x}$ below is different from $\Phi^{x,\xi}$ in \eqref{FBSDE2}):
\bea
\label{FBSDE3}
\left.\ba{c}
 \dis X_t^{\xi,x}=x+\int_{t_0}^t\pa_pH(X_s^{\xi,x},Z_s^{\xi,x})ds+ B^{t_0}_t+\b B_t^{0,t_0}; \\
\dis Y_t^{\xi,x}=G(X_T^{\xi,x},\rho_T)+\!\int_t^T\!\![F(X_s^{\xi,x},\rho_s)- \wh L(X_s^{\xi,x},Z_s^{\xi,x})]ds\\
\dis-\!\int_t^T\!\! Z_s^{\xi,x}\cd dB_s-\!\int_t^T\!\! Z_s^{0,\xi,x}\cd dB_s^{0}.\qq\q\,\,\,\,
\ea\right.
\eea
However, we emphasize that we cannot replace the coupled McKean-Vlasov FBSDE \reff{FBSDE1} with a decoupled one like \reff{FBSDE2}, due to the involvement of the conditional law $\rho$. In fact, this is the main difficulty for studying the master equation. 

\subsection{The Nash system}
One of the most important applications of the mean field game and the master equation is to characterize the asymptotic behavior of the $N$-player game for a large interacting particle system. For $t_0\in [0, T]$, ${\vec x} =(x_1,\cds, x_N)\in \dbR^{N\times d}$, and ${\vec \a} = (\a^1,\cds, \a^N): [t_0, T]\times \dbR^d\times \cP_2\to \dbR^{N\times d}$, consider the following game problem for controlled interacting system over $[t_0, T]$:
\bea
\label{NashXa}
\left.\ba{c}
\dis X_{t}^{\vec x, \vec \a, i}=x_i + \int_{t_0}^t\a^i\big(s, X_s^{\vec x, \vec \a, i}, \mu_s^{\vec x, \vec \a, i}\big) ds+ B^{i,t_0}_t+ \beta B_t^{0,t_0}, ~\mu_s^{\vec x, \vec \a, i}:= {1\over N-1} \sum_{j\neq i} \d_{X_s^{\vec x, \vec \a, j}}\\
\dis J^N_i(t_0, \vec x, \vec \a) := \dbE\Big[G(X_T^{\vec x, \vec \a, i}, \mu_T^{\vec x, \vec \a, i}) + \int_{t_0}^T \big[F(X_s^{\vec x, \vec \a, i}, \mu_s^{\vec x, \vec \a, i})\\
\dis - L\big(X_s^{\vec x, \vec \a, i}, \a^i(s, X_s^{\vec x, \vec \a, i}, \mu_s^{\vec x, \vec \a, i})\big)\big] ds \Big]
\ea\right.
\eea
where $B^0, B^1, \cds, B^N$ are independent $d$-dimensional Brownian motions. As explained in Remark \ref{rem-control}, here the controls depend on $\mu$ as well. However, for simplicity  we are using state dependent controls only. Under the conditions of this paper, this restriction does not change the game problem, and we refer to \cite{IZ} for discussions on the subtly of path dependent controls in general case. The equilibrium ${\vec \a^*}$ is defined in the standard way:
\bea
\label{NEa*}
J^N_i(t_0, \vec x, \vec \a^*) = \sup_{\a^i} J^N_i(t_0, \vec x, \vec \a^{*,-i}, \a^i),\q i=1,\cds, N, 
\eea
where $\vec \a^{-i} := (\a^1,\cds, \a^{i-1},\a^{i+1}, \cds, \a^N)$. Under appropriate conditions, $\vec \a^*= \vec \a^*(t_0, \vec x)$ is unique for all $(t_0, \vec x)$. Then we may define the value function of the $N$-player game:
\bea
\label{Nashvalue}
v^{N,i}(t, \vec x) := J^N_i(t, \vec x, \vec \a^*(t,\vec x)),\q i=1,\cds, N.
\eea
We emphasize that, unlike \reff{Jxi}, the $J^N_i$ in \reff{NEa*}  is deterministic and hence so is $v^{N,i}$.

The above value functions $\{v^{N, i}\}_{1\le i\le N}$ satisfy  the following Nash system $[0,T]\times \dbR^{N\times d}$:
\bea
\label{Nash}
  \left.\begin{array}{c} 
\dis\mathcal{L}^{N,i}v^{N,i}(t, \vec x) =0,\q v^{N,i}(T,\vec x)=G(x_i,m_{\vec x}^{N,i}),\q\mbox{where}\\
\dis \mathcal{L}^{N,i}v^{N,i}(t, \vec x) := \pa_t v^{N,i} +\frac{1}{2}\sum_{j=1}^N \tr(\pa_{x_jx_j}v^{N,i}) +\frac{\b^2}{2}\sum_{j,k=1}^N\tr(\pa_{x_jx_k}v^{N,i})\\
\dis +H(x_i,\pa_{x_i}v^{N,i}) +F(x_i,m_{\vec x}^{N,i})+\sum_{j\not=i}\pa_pH(x_j,\pa_{x_j}v^{N,j})\cd\pa_{x_j}v^{N,i},\\
\dis m_{\vec x}^{N,i}:=\frac{1}{N-1}\sum_{j\not=i}\delta_{x_j}.
\end{array}
  \right.
\eea
Note that the system is symmetric with respect to $i$ and ${\vec x}_{-i}:=(x_1,...,x_{i-1},x_{i+1},...,x_N)$. Then, when the system is wellposed, the solution $v^{N,i}$ should also be symmetric on $i$ and ${\vec x}_{-i}$, that is,   there exists a function $U^N: \dbR^d\times \cP_2\to \dbR$, independent of $i$, such that 
\bea
\label{uvN}
v^{N,i}(t, \vec x) = U^N(t, x_i, m^{N,i}_{\vec x}),\q i=1,\cds, N.
\eea
 When $U^N$ is smooth, one can easily check that: for $j, k \neq i$ and $j\neq k$,
\bea
\label{pamuuN}
\left.\ba{c}
\dis \pa_{x_i} v^{N,i}(t, \vec x)= \pa_x U^N(t, x_i,m^{N,i}_{\vec x}),\,\,   \pa_{x_j} v^{N,i}(t, \vec x)=\frac{\pa_{\mu} U^N(t, x_i,m^{N,i}_{\vec x}, x_j) }{N-1},\\
\dis  \pa_{x_i x_i} v^{N,i}(t, \vec x)=\pa_{xx} U^N(t, x_i,m^{N,i}_{\vec x}),\,\, \pa_{x_ix_j} v^{N,i}(t, \vec x)=\frac{\pa_{x\mu} U^N(t, x_i,m^{N,i}_{\vec x}, x_j)}{N-1},\\
\dis  \pa_{x_jx_j} v^{N,i}(t, \vec x)=\frac{\pa_{\mu\mu}U^N(t,x_i,m_{\vec x}^{N,i},x_j,x_j)}{(N-1)^2}+\frac{\pa_{\tilde x\mu} U^N(t, x_i,m^{N,i}_{\vec x}, x_j)}{N-1},\\
\dis \pa_{x_j x_k} v^{N,i}(t, \vec x)=\frac{\pa_{\mu\mu} U^N(t, x_i,m^{N,i}_{\vec x}, x_j, x_k)}{(N-1)^{2}},\,\, \pa_t v^{N,i}(t, \vec x)=\pa_t U^N(t, x_i,m^{N,i}_{\vec x}).
\ea\right.
\eea
Then we may rewrite the Nash system \reff{Nash} as a discrete master equation:
\bea
\label{Nash2}
\mathcal{L}^{N}U^N(t,x_i, m^{N,i}_{\vec x}) =0,\q U^N(T, x_i, m^{N,i}_{\vec x}) = G(x_i, m^{N,i}_{\vec x}),
\eea
where, for $\tilde \xi, \bar \xi$ being independent with distribution $m^{N,i}_{\vec x}$,
\beaa
&\dis \mathcal{L}^{N}U^N(t,x_i, m^{N,i}_{\vec x}):= \pa_t U^N + \frac{\h\b^2}{2}\tr(\pa_{xx} U^N) + H(x_i,\partial_x U^N) + F(x_i,m^{N,i}_{\vec x})\\
&\dis  + \tr\Big(\tilde  \dbE\Big[\frac{\h\b^2}{2} \pa_{\tilde x} \pa_\mu U^N(t,x_i, m^{N,i}_{\vec x}, \tilde \xi) + \pa_\mu U^N(t, x_i, m^{N,i}_{\vec x}, \tilde \xi)(\pa_pH)^\top(\tilde \xi, \pa_x U^N(t, \tilde \xi, m^{N,i}_{\vec x}))\\
&\dis  +\b^2\pa_{x\mu} U^N(t,x_i,m^{N,i}_{\vec x},\tilde \xi)+\frac{\b^2}{2}\bar{\mathbb E}\big[\pa_{\mu\mu}U^N(t,x_i,\mu,\bar\xi,\tilde\xi)\big]\Big]\Big)\\
&\dis +\frac{\sum_{j\neq i}\tr\big(\pa_{\mu\mu}U^N(t,x_i,m_{\vec x}^{N,i},x_j,x_j)\big)}{2(N-1)^2} + {\pa_{\mu}U^N(t,x_i,m_{\vec x}^{N,i},x_j)\over N-1} \cdot\\
&\dis  \sum_{j\neq i}\big[\pa_p H(x_j,\pa_x U^N(t,x_j,m_{\vec x}^{N,j}))-\pa_pH(x_j,\pa_{x}U^N(t,x_j,m_{\vec x}^{N,i}))\big].
\eeaa

Similarly, we may express the Nash system \reff{Nash} in terms of FBSDEs.  Fix $(t_0, \vec x)\in [0, T]\times \dbR^{N\times d}$, and consider  the following  two systems  of FBSDEs on $[t_0,T]$: 
\bea
\label{NashFBSDE1}
&&\dis  \left\{\begin{array}{lll} 
  \dis X_{t}^{i,\vec x}=x_i +B^{i,t_0}_t+ \beta B_t^{0,t_0},\q X^{\rightarrow,\vec x}:= (X^{1,\vec x},\cds, X^{N,\vec x});\\
\dis Y_{t}^{i,\vec x}=G(X_{T}^{i,\vec x}, m_{X_T^{\rightarrow,\vec x}}^{N,i})+\int_t^T\Big[F(X_{s}^{i,\vec x}, m_{X_s^{\rightarrow,\vec x}}^{N,i}) +H(X_{s}^{i,\vec x}, Z_{i,s}^{i,\vec x})\\
\dis\qq+\sum_{j\neq i}Z_{j,s}^{i,\vec x}\cd\pa_{p}H(X_{s}^{j,\vec x},Z_{j, s}^{j, \vec x})\Big]ds -\sum_{j=1}^N\int_t^TZ_{j,s}^{i,\vec x}\cd dB_s^{j}-\int_t^TZ_{s}^{0,i,\vec x}\cd dB_s^{0};
\ea\right.\ms\\
&&\dis\label{NashFBSDE2}
  \left\{\begin{array}{ll} 
  \dis  X_{t}^{\vec x, i}=x_i + \int_{t_0}^t\pa_pH( X_{s}^{\vec x, i}, Z_{i,s}^{\vec x, i})ds+ B^{i,t_0}_t+ \beta B_t^{0,t_0},~ X^{\vec x,\rightarrow}:= ( X^{\vec x,1},\cds,  X^{\vec x, N}); \\
\dis  Y_{t}^{\vec x, i}=G( X_{T}^{\vec x, i}, m_{ X_T^{\vec x,\rightarrow}}^{N,i})+\int_t^T\big[F( X_{s}^{\vec x, i}, m_{ X_s^{\vec x,\rightarrow}}^{N,i}) -\h L( X_{s}^{\vec x, i}, Z_{i,s}^{\vec x,i})\big]ds\\
\dis\qq -\sum_{j=1}^N\int_t^T Z_{j,s}^{\vec x,i}\cd dB_s^{j}-\int_t^T Z_{s}^{0,\vec x,i}\cd dB_s^{0}.
\ea\right.
\eea
Note again that, similar to \eqref{FBSDE2} and \eqref{FBSDE3}, we used $\Phi^{i,\vec{x}}$ and $\Phi^{\vec{x},i}$ to denote the two systems above. They are connected with  \reff{Nash} as follows: for $i, j= 1,...,N$ and $t\in [t_0, T]$,
\bea 
\label{YvN}
\left.\ba{c}
\dis Y_{t}^{i,\vec x} = v^{N,i}(t, X_t^{\rightarrow,\vec x}), \q Z_{j,t}^{i,\vec x} =  \pa_{x_j} v^{N,i}(t, X_t^{\rightarrow,\vec x}),\q Z_{t}^{0,{i,\vec x}}=\beta \sum_{j=1}^N \pa_{x_j}v^{N,i}(t, X_t^{\rightarrow,\vec x});\\
\dis  Y_{t}^{\vec x, i} = v^{N,i}(t,  X_t^{\vec x,\rightarrow}), \q  Z_{j,t}^{\vec x, i} =  \pa_{x_j} v^{N,i}(t,  X_t^{\vec x,\rightarrow}),\q Z_{t}^{0,\vec x, i}=\beta \sum_{j=1}^N \pa_{x_j}v^{N,i}(t, X_t^{\vec x,\rightarrow}).
\ea\right.
\eea
We emphasize that we have to use the coupled system \reff{FBSDE1} to derive the law $\cL_{X^\xi_t}$. Due to the presence of the individual noises $B^i$, we can apply the high dimensional Girsanov theorem  and induce the empirical measures $m_{X_T^{\rightarrow,\vec x}}^{N,i}$ through the same decoupled system \reff{NashFBSDE1}. This simplifies the wellposedness of \reff{Nash} significantly.  We note that, as in the standard theory we may view \reff{NashFBSDE1} as a weak solution to the coupled system \reff{NashFBSDE2}, and  they both correspond to the same PDE system \reff{Nash}. In particular we have $Y^{i, \vec x}_{t_0} =  Y^{\vec x, i}_{t_0}$. However, compared to \reff{FBSDE2}, the system \reff{NashFBSDE1} involves an extra term $\sum_{j\neq i}Z_{j,s}^{i,\vec x}\cd \pa_{p}H(X_{s}^{j,\vec x},Z_{j, s}^{j, \vec x})$.

\subsection{Technical conditions and the main results}

In this subsection, we first collect some technical conditions which will be used in the paper. 
\begin{assum}
\label{assum-Lip1}
 (i) $F, G: \dbR^d \times \cP_1 \to \dbR$ are uniformly Lipschitz continuous in both $x$ and $\mu$  with a Lipschitz constant $L_1$, where the Lipschitz continuity in $\mu$ is under $\cW_1$. 
 
 (ii) $\pa_x F, \pa_x G$ exist and are continuous in $(x,\mu)$, again under $\cW_1$ for $\mu$.
\end{assum}

\begin{assum}
\label{assum-Lip2}
 $\pa_x F, \pa_x G$ are also uniformly Lipschitz continuous in both $x$ and $\mu$, under $\cW_1$ for $\mu$, with a Lipschitz constant $L_2$.
\end{assum}

\begin{rem}
\label{rem-Lip}
{\rm (i) Assumption \ref{assum-Lip1} (i) is  standard, except that it would look more natural to assume the Lipschitz continuity under $\cW_2$, in light of \reff{pamu}. We use $\cW_1$ here mainly because of an issue explained in Remark \ref{rem-mollifier} (i) below. We emphasize that, since $\cW_1 \le \cW_2$,  Assumption \ref{assum-Lip1} implies $F$ and $G$ are uniformly Lipschitz continuous under $\cW_2$.

(ii) Many results in the paper concerning FBSDE \reff{FBSDE1} require only the Lipschitz continuity in Assumption \ref{assum-Lip1} (i), not differentiability in Assumption \ref{assum-Lip1} (ii). The differentiability is mainly for the convenience of studying FBSPDE \reff{SPDE}. However, for the ease of presentation, and since anyway our main result will require the stronger Assumption \ref{assum-Lip2}, we assume Assumption \ref{assum-Lip1} (ii) throughout the paper. We note that it may be possible to study the weak solution of  FBSPDE \reff{SPDE} without requiring the differentiability in $x$.

(iii) Assumption \ref{assum-Lip2} is somewhat stronger than what we expected and it will be ideal to weaken it. However, we should point out that it is still much weaker than the technical conditions required in the literature for the existence of classical solutions to the master equation, see e.g. \cite{CDLL,CD1,CD2,CCD}.
\qed}
\end{rem}

We next impose conditions on the Hamiltonian $H$. For any $R>0$, denote 
\bea
\label{DR}
D_R := \{(x, z)\in \dbR^{d\times 2}: |z|\le R\}.
\eea

\begin{assum}
\label{assum-H1}
$H\in C^1$, and for any $R>0$, there exist $L^H_1(R), L^H_2(R)$ such that
\bea
\label{Hassum1}
   |\pa_p H(x, z)|\le L^H_1(R),~ \forall (x, z)\in D_R\q\mbox{and}\q   c^H_1:= \limsup_{|z|\to\infty} {\|\pa_x H(\cd,z)\|_\infty \over |z|} < {1\over T}.
\eea
Moreover, $\pa_x H, \pa_p H$ are Lipschitz continuous in $D_R$ with Lipschitz constant $L^H_2(R)$.
\end{assum}

\begin{assum}
\label{assum-H2}
There exists $c^H_2(R)>0$ such that, for any $x\in \dbR^d, |z_1|, |z_2|\le R$,
\bea
\label{Hconvex}
H(x, z_2) - H(x, z_1) - \pa_p H(x, z_1)\cd (z_2-z_1) \ge {c^H_2(R)\over 2} |z_2-z_1|^2.
\eea
\end{assum}
We note that, when $H\in C^2$, then \reff{Hconvex} means $\pa_{pp}^2H(x,z) \ge c^H_2(R)I_d$ for any $(x,z)\in D_R$.
We also emphasize that at above we only require the local regularity of $H$ with respect to $z$, which in particular holds for the linear quadratic case where $H(z) = {1\over 2} |z|^2$. 

We will also need the crucial monotonicity condition which  is standard in the literature.
\begin{assum}
\label{assum-mon}
For any $\mu_1, \mu_2\in \cP$ and for $\F = F, G$, we have
\bea
\label{mon}
\int_{\dbR^d} \big[\F(x, \mu_1) - \F(x, \mu_2)\big]\big[\mu_1(dx) - \mu_2(dx)\big] \le 0.
\eea
\end{assum}
One typical example satisfying the monotonicity condition is: 
\bea
\label{mon-eg1}
\Phi(x, \mu) = |x- m_\mu|^2,\q\mbox{where}\q m_\mu := \int_{\dbR^d} x \mu(dx).
\eea
In this case one can verify straightforwardly that
\bea
\label{mon-eg2}
\int_{\dbR^d} \big[\F(x, \mu_1) - \F(x, \mu_2)\big]\big[\mu_1(dx) - \mu_2(dx)\big] = -2 [m_{\mu_1} - m_{\mu_2}]^2\le 0.
\eea

\begin{rem}
\label{rem-mon}
{\rm (i) As we will see in the paper, this condition is crucial for the uniform Lipschitz continuity of $V$ with respect to $\mu$, and thus is crucial for the existence of global (in time) solutions of the coupled systems \reff{SPDE} and \reff{FBSDE1} as well as the master equation \reff{master}. However, we emphasize that the uniqueness of the solutions to the master equation, under all the notions we will propose, does not rely on this condition. 

(ii) We remark though that the uniqueness of solutions to the master equation does not imply the uniqueness of mean field equilibria, because it is possible that different mean field equilibria induce the same value function. A trivial example is that $F = G = L \equiv 0$, then $V\equiv 0$ but any control $\a$ is an equilibrium.

(iii) It will be very interesting and challenging to investigate the global wellposedness of master equations without the monotonicity condition. In particular, inspired by \cite{GM}, we expect that the wellposedness results in this paper will remain true under the displacement monotonicity. We shall leave these for future research.  
 \qed}
 \end{rem}

Our main goal of this paper is to establish the global wellposedness of the master equation \reff{master} under Assumptions \ref{assum-Lip1}, \ref{assum-Lip2}, \ref{assum-H1}, \ref{assum-H2}, and \ref{assum-mon}. We first note that under these assumptions the master equation \reff{master} may not have a classical solution in general, see Example \ref{eg-nonclassical} below for a counterexample. We shall propose three weaker notions of solutions. The first one is called {\it good solution}, which is in the spirit of the stability result and the name is motivated by  \cite{JKS}.  The second one is called {\it weak solution}, which is in the spirit of the integration by parts formula for the FBSPDEs \reff{SPDE}. The last one is called {\it weak-viscosity solution}, also in terms of the FBSPDEs \reff{SPDE}. We remark that the comparison principle for the viscosity solution is only for the $u$ in \reff{SPDE} (for fixed $\rho$). Typically the master equation does not satisfy the comparison principle, see Example \ref{eg-comparison} below for a counterexample. We also emphasize that we will not require any differentiability in $\mu$, neither for the data $F, G, L, H$ nor for the solution $V$. The main results of this paper are summarized below:
 
\begin{itemize}
\item{} We establish the global wellposedness of the master equation \reff{master} under all three notions of solutions. Moreover, with slightly different requirements on the regularity in $x$, the three notions are actually equivalent. See Theorems \ref{thm-good}, \ref{thm-weak}, and \ref{thm-viscosity}.  

\item{} The key for our global wellposedness results is the uniform Lipschitz continuity of $V$ in $\mu$, under the monotonicity condition \reff{mon}. Moreover, $V(t,\cd)$ keeps the monotonicity condition. See Theorem \ref{thm-well-posedness}. The related wellposedness of the FBSDEs \reff{FBSDE1}-\reff{FBSDE2} and the stability of $V$ are also established in Section \ref{sect-regularity}.

\item{} In order to study our good solution, in Section \ref{sect-mollifier} we construct a smooth mollifier for functions on Wasserstein space. The main feature of our mollifier is that it keeps the uniform Lipschitz continuity under  $\mathcal{W}_1$ (but not under $\mathcal{W}_2$). See Theorem \ref{thm-mollifier}.

\item{} We prove the convergence of the Nash system \reff{Nash} (or equivalently the discrete master equation \reff{Nash2}) to the master equation \reff{master}, see Theorem  \ref{thm-Nash1} (i). Moreover, we prove the convergence of the equilibrium empirical measure $\rho^N_t := {1\over N}\sum_{i=1}^N \d_{ X_{t}^{\vec x, i}}$ for the system \reff{NashFBSDE2} to the $\rho_t$ in \reff{SPDE} (or equivalently \reff{FBSDE1}), as well as a propagation of chaos property for the associated optimal trajectories, See Theorems \ref{thm-Nash1} (ii) and \ref{thm-chaos}.

\item{} As an independent result, we provide pointwise representation formulas for the derivatives of $V$, when the data are differentiable in $\mu$. These formulas are new, to our best knowledge, and can be viewed as alternative proofs for the existence of classical solutions, provided the data are smooth enough. See Section \ref{sect-representation}.
\end{itemize}

 \section{A smooth mollifier on Wasserstein space}
\label{sect-mollifier}
\setcounter{equation}{0}

To help for our notion of good solution for the master equation with less smooth data, in this section we construct a smooth mollifier for the data,  which is new in the literature, to our best knowledge.  The main difficulty lies in the fact that the Wasserstein space of measures is infinitely dimensional.
 
Fix $U\in C^0(\cP_1)$.  We construct the mollifier in  two steps. 

{\bf Step 1. Discretization of $\mu\in \cP_1$.} 

Since $\mu$ is infinitely dimensional, in order to mollify $U$ we first approximate $\mu$ with finitely dimensional measures. For this purpose,  we fix $n\ge 3$. Denote
\bea
\label{Deltai}
\D_{\vec i}:= [{i_1\over n},{i_1+1\over n}) \times\cds\times[{i_d\over n},{i_d+1\over n}),\q {\vec i} = (i_1,\cds, i_d)^\top\in \dbZ^d.
\eea
A natural discretization of $\mu$ is $\sum_{\vec i\in \dbZ^d} \mu(\D_{\vec i})\d_{\vec i\over n}$. However, note that $x\in \dbR^d \mapsto \1_{\D_{\vec i}}(x)$ is discontinuous, consequently $\mu\in \cP_1\mapsto \mu(\D_{\vec i})$ is discontinuous. So we shall replace $ \mu(\D_{\vec i})$ with $\psi_{\vec i}(\mu)$ for some $\psi_{\vec i}\in C^\infty(\cP_2)\cap C^0(\cP_1)$, see \reff{mun} below.

We first introduce a function $I = I_n \in C^\infty([0,1])$ such that: for $x\in [0,1]$,
\bea
\label{I}
\left.\ba{c}
I(x) =1 ~\mbox{for}~ x\le {1\over n^3},\q I(x) = 1-x ~\mbox{for}~ {1\over n} \le x\le 1-{1\over n},\q I(x) =0 ~\mbox{for}~ x\ge 1-{1\over n^3};\\
0\le I \le 1, \q - [1+{1\over n}] \le I'\le 0,\q \mbox{and}\q I(x) + I(1-x) = 1. 
\ea\right.
\eea
Define 
 \bea
\label{fi0}
\phi_{i}(x) :=  I\big(|nx-i|\big) \1_{[{i-1\over n}, {i+1\over n}]}(x), \q i\in \dbZ,~ x\in \dbR.
\eea
Then one can verify straightforwardly that $\phi_i \in C_{c}^\infty(\dbR)$, $0\le \phi_{ i}\le 1$, and for $x\in [{i\over n},{i+1\over n}]$:
\bea
\label{phi0}
\phi_{i}(x)+\phi_{i+1}(x)=1,\q \phi_j(x) = 0,\q j\neq i, i+1.
\eea
Moreover, we may extend the function $\phi$ to $\dbR^d$: by abusing the notation $x$, 
\bea
\label{fi}
\phi_{\vec i}(x):=\prod_{l=1}^d\phi_{i_l}(x_l),\q {\vec i} \in \dbZ^d,~ x=(x_1,\cds, x_d)^\top\in \dbR^d.
\eea
Then $\phi_{\vec i} \in C_{c}^\infty(\mathbb R^d)$, $0\le \phi_{\vec i}\le 1$, and for all $x\in \D_{\vec i}$,
\bea
\label{phi}
&\dis\sum_{\vec j\in  J_{\vec i}}\phi_{\vec j}(x)  = \prod_{l=1}^d(\phi_{i_l}(x)+\phi_{i_l+1}(x))=1,~\mbox{and}~ \phi_{\vec j}(x)=0,~ {\vec j} \notin J_{\vec i},\\
&\dis\mbox{where}\q J_{\vec i} := \{ {\vec j}: j_l = i_l, i_{l+1}, l=1,\cds, d\}.\nonumber
\eea

Since $\dbR^d$ is unbounded, we next introduce a truncation function $\ch=\ch_n$. Denote 
\bea
\label{Qr}
Q_M:= \{x\in \dbR^d: |x_l|\le M, l=1,\cds,d\}.
\eea
 Let $\ch \in C_c^\infty(\dbR^d)$ satisfy $0\leq \ch\leq 1$ in $\dbR^d$, $\ch\equiv 1$ in $Q_n$, $\ch\equiv 0$ in $Q_{3n\over 2}^c$,  and $|\pa_x\ch|\leq\frac{3}{n}$ in $\mathbb R^d$. For each $\mu\in \cP_1$, define
\bea
\label{mun}
&\dis\mu_n := \sum_{\vec i\in \mathbb Z^d} \psi_{\vec{i}}(\mu)  \d_{{\vec {i}\over n}},\\
&\dis\mbox{where}~
\psi_{\vec i}(\mu):= \int_{\mathbb R^d}\phi_{\vec i}(x)\ch(x)\mu(dx)+ \1_{\{\vec i=\vec 0\}} \int_{\mathbb R^d}(1-\ch(x))\mu(dx).\nonumber
\eea
It is clear that
\beaa
\mu_n(\dbR^d) = \sum_{\vec i\in\mathbb Z^d} \psi_{\vec i}(\mu)\d_{x_{\vec i}}(\dbR^d) =\int_{\mathbb R^d}  \sum_{\vec i\in\mathbb Z^d} \phi_{\vec i}(x)\ch(x) \mu(dx)+\int_{\mathbb R^d} (1-\ch(x))\mu(dx) = 1.
\eeaa
This implies that $\mu_n \in \cP$.  Moreover, note that
\[
\partial_\mu\psi_{\vec i}(\mu,x)= \pa_x(\phi_{\vec i}(x)\ch(x)) -  \1_{\{\vec i=\vec 0\}}\pa_x \ch(x).
\]
Then one can easily show that  $\psi_{\vec i}\in C^\infty(\cP_2)$.

{\bf Step 2. Mollification of $U$.} 

Note that $\psi_{\vec i}(\mu)=0$ whenever ${\vec i\over n}\notin Q_{2n}$, or say $\vec i \notin Q_{2n^2}$. Denote 
\bea
\label{tildeUn}
\left.\ba{c}
\tilde U_n(\mu) := U(\mu_n) = U\Big(\sum_{\vec i\in \dbZ^d_n} \psi_{\vec i}(\mu)  \d_{\vec i\over n}\Big),\q \mbox{where}\q \dbZ^d_n:= \dbZ^d \cap Q_{2n^2}.
\ea\right.
\eea
Since $\mu_n$ is a discrete measure, one can mollify $\tilde U_n$ through the coefficients $\psi_{\vec i}(\mu)$. However, note that $\{\psi_{\vec i}(\mu)\}_{\vec i\in\mathbb Z^d}$ is a (discrete) probability and $U$ is defined only on probability measures, we need some special treatment for the mollification. To be precise, note that $|\dbZ^d_n| = N_n:=(4n^{2}+1)^d$. Denote 
\bea
\label{Deltan}
\left.\ba{c}
\D_n := \{y=(y_{\vec i})_{\vec i\in \dbZ^d_n\setminus \{0\}} : |y_{\vec i}|\le N_n^{-3}\} \subset \dbR^{N_n-1},\q\mbox{and}\q  y_{\vec 0}:= -\sum_{\vec i\in\dbZ^d_n\setminus \{0\}} y_{\vec i}.
\ea\right.
\eea
 Then $\sum_{\vec i\in\dbZ^d_n}y_{\vec i}=0$ and $|y_{\vec 0}|\le {1\over N_n^2}$. Define  
\bea
\label{muny}
\mu_n(y) := \sum_{\vec i\in \dbZ^d_n} \wh\psi_{\vec i}(\mu, y) \d_{\vec i\over n},\q\mbox{where}\q \wh\psi_{\vec i}(\mu, y) :=  {N_n\over N_n+1}[\psi_{\vec i}(\mu) + {1\over N_n^2}+   y_{\vec i}].
\eea  
The role of $\wh\psi_{\vec i}(\mu,y)$ can be viewed as a perturbation of $\psi_{\vec{i}}(\mu)$ and one can easily check that
\beaa
\wh\psi_{\vec i}(\mu, y) \ge 0,\qq \sum_{\vec i\in \dbZ^d_n} \wh\psi_{\vec i}(\mu, y) = {N_n\over N_n+1}\Big[{1\over N_n}+1+0\Big] =1.
\eeaa
That is, $\mu_n(y) \in \cP$ for all $y\in \D_n$.  Finally, let $\zeta_n$ be a smooth density function with support $\D_n$, we then define the mollifier of $U$ as follows:
\bea
\label{Un}
U_n(\mu) := \int_{\D_n} \zeta_n(y)  U\big(\mu_n(y)\big) dy.
\eea

For any $\cM\subset \mathcal{P}_1$, denote $\|U\|_{L^\infty(\cM)} := \sup_{\mu\in \cM} |U(\mu)|$. Moreover, we use $\subset \subset$ to denote compact subsets. Then we have the following convergence result.
\begin{thm}
\label{thm-mollifier}
Let $U \in C^0(\cP_1)$ and $U_n$ be defined by \reff{Un}. Then 

(i) $U_n \in C^\infty(\cP_2)\cap C^0(\cP_1)$ and $\lim_{n\to\infty} \|U_n - U\|_{L^\infty(\cM)} =0$, for any $\cM\subset\subset \mathcal{P}_1$. 

(ii) If $U$ is Lipschitz continuous in $\mu$ under $\cW_1$ with Lipschitz constant $L$, then $U_n$ is uniformly Lipschitz continuous in $\mu$ under $\cW_1$ with Lipschitz constant $CL$, where $C$ may depend on $d$, but not on $n$.

(iii) Assume $U\in C^1(\cP_2)$, and $\pa_\mu U$ is uniformly continuous in $(\cM \cap \cP_2) \times K$ under $\cW_1$ for the component $\mu$, where $\cM\subset\subset \cP_1$ and $K\subset\subset\mathbb R^d$, then
\bea
\label{pauUconv}
\lim_{n\to\infty} \sup_{\mu\in \cM\cap \cP_2} \int_{K} |\pa_\mu U_n(\mu, x) - \pa_\mu U(\mu, x)| dx=0.
\eea
\end{thm}
\proof  (i) Recall \reff{muny} and denote $z_{\vec i}:=\psi_{\vec i}(\mu) +  {1\over N_n^2}+  y_{\vec i}$. We see that $z=(z_{\vec i})_{\vec i\in \dbZ^d_n\setminus \{0\}}$ takes values in $\wh \D_n(\mu) := \{z: z_{\vec i} \in [\psi_{\vec i}(\mu) + {1\over N_n^2} - {1\over N_n^3}, \psi_{\vec i}(\mu) + {1\over N_n^2} + {1\over N_n^3}]\} \subset \dbR^{N_n-1}$, and the inverse function of the mapping from $y\in \D_n \to z\in \wh\D_n(\mu)$ is: 
\[
\k(\mu,z):=(\k_{\vec i}(\mu,z))_{\vec i\in\dbZ^d_n\setminus \{0\}\} }\q \text{where}\q \k_{\vec i}(\mu, z) := z_{\vec i} - \psi_{\vec i}(\mu) - {1\over N_n^2}.
\]
Denote also $z_{\vec 0}:=\frac{N_n+1}{N_n}-\sum_{\vec i\in \dbZ^d_n\setminus \{0\}\}}z_{\vec i}$.  Then  ${N_n\over N_n+1}\sum_{\vec i\in \dbZ^d_n} z_{\vec i} \d_{x_{\vec i}}\in \cP$ for $z\in \wh \D_n(\mu)$, and
\bea
\label{Un2}
U_n(\mu) = \int_{\wh\D_n(\mu)} \zeta_n(\k(\mu, z)) U({N_n\over N_n+1}\sum_{\vec i\in \dbZ^d_n} z_{\vec i} \d_{x_{\vec i}}) dz.
\eea
Since $\psi_{\vec i}\in C^\infty(\mathcal{P}_2)$ and $\zeta_n$ is also smooth, we can easily see that $U_n \in C^\infty(\cP_2)$. 

Next, for any $\mu, \nu\in \cP_1$, recall \reff{W1} and \reff{mun} one can easily see that $|\psi_{\vec i}(\mu) -\psi_{\vec i}(\nu)|\le C_n \cW_1(\mu, \nu)$.  Then by \reff{Un2} we have $|U_n(\mu)-U_n(\nu)|\le C_n \cW_1(\mu, \nu)$. That is, for each $n$, $U_n$ is Lipschitz continuous under $\cW_1$, hence $U_n \in C^0(\cP_1)$. 

Finally,  recall \reff{W1} again and let $\f\in C^1(\dbR^d; \dbR)$ satisfy $\f(0)=0$, $|\pa_x\f|\le 1$. Note that
\beaa
&\dis |\f({\vec i\over n})|\le Cn,\q \vec i \in \dbZ^d_n;\\
&\dis \Big| \sum_{\vec i\in \dbZ^d_n}  \f({\vec i\over n}) \phi_{\vec i}(x)- \f(x)\Big| = \Big| \sum_{\vec i\in \dbZ^d_n}  [\f({\vec i\over n}) -\f(x)]\phi_{\vec i}(x)\Big| \le  \sum_{\vec i\in \dbZ^d_n} {C\over n} \phi_{\vec i}(x) = {C\over n},\q x\in \dbR^d,
\eeaa
where $C>0$  may depend on $d$. Then,  for any $y\in \D_n$, by \reff{muny} and \reff{mun} we have
\beaa
&&\Big|\int_{\mathbb R^d} \f(x) [\mu_n(y) (dx) - \mu(dx)]\Big| = \Big| \sum_{\vec i\in \dbZ^d_n} \wh\psi_{\vec i}(\mu, y) \f({\vec i\over n}) - \int_{\mathbb R^d} \f(x) \mu(dx)\Big| \\
&&\le   \Big| \sum_{\vec i\in \dbZ^d_n} \psi_{\vec i}(\mu) \f({\vec i\over n}) - \int_{\mathbb R^d} \f(x) \mu(dx)\Big| +{1\over N_n+1} \sum_{\vec i\in \dbZ^d_n} \psi_{\vec i}(\mu) |\f({\vec i\over n})| \\
&&\qq+{N_n\over N_n+1} \sum_{\vec i\in \dbZ^d_n}  [{1\over N_n^2} + |y_{\vec i}|]  |\f({\vec i\over n})|\\
&&\le \Big| \sum_{\vec i\in \dbZ^d_n}  \f({\vec i\over n}) \int_{\mathbb R^d} \phi_{\vec i}(x)\ch(x) \mu(dx)  - \int_{\mathbb R^d} \ch(x)\f(x) \mu(dx)\Big| \\
&&\qq+ \Big|\int_{\mathbb R^d}(1-\ch(x))\varphi(x)\mu(dx)\Big|+\frac{Cn}{N_n}\sum_{\vec i\in \dbZ^d_n}\int_{\mathbb R^d}\phi_{\vec i}(x)\ch(x)\mu(dx)+\frac{C n}{N_n}\\
&&\le \int_{\mathbb R^d}   \sum_{\vec i\in \dbZ^d_n}  |\f({\vec i\over n})- \f(x)|\phi_{\vec i}(x) \ch(x) \mu(dx)  + \int_{Q_n^c}|x|\mu(dx) + {Cn\over N_n}\\
&&\le \int_{\mathbb R^d}   {C\over n} \ch(x) \mu(dx)  + \int_{Q_n^c}|x|\mu(dx) + {Cn\over N_n}\le {C\over n}+\int_{Q_n^c}|x|\mu(dx)+\frac{Cn}{N_n}.
\eeaa
Note that the compactness of $\cM$ implies $\dis\lim_{M\to\infty}\sup_{\mu\in \cM}\int_{Q_M^c}|x|\mu(dx)=0$. Then 
\bea
\label{munconv}
\lim_{n\to\infty} \sup_{y\in \D_n, \mu\in \cM}\cW_1(\mu_n(y), \mu) = 0.
\eea
Since $U$ is continuous in $\cP_1$, by standard compactness arguments we have 
\beaa
\lim_{n\to\infty}  \sup_{y\in \D_n, \mu\in \cM}|U(\mu_n(y))-U(\mu)| = 0.
\eeaa
This, together with \reff{Un}, implies $\lim_{n\to\infty} \sup_{\mu\in \cM} |U_n(\mu)-U(\mu)|=0$.  

(ii) Let $\mu, \nu \in \cP_1$, $y\in \D_n$. For $\f\in C^1(\mathbb R^d; \dbR)$ satisfying $\f(0)=0$, $|\pa_x\f|\le 1$, we have
\bea
\label{munest}
&\dis\Big|\int_{\mathbb R^d} \f(x) [\mu_n(y) (dx) - \nu_n(y)(dx)]\Big| = \Big| \sum_{\vec i\in\dbZ^d_n} [\wh\psi_{\vec i}(\mu, y) - \wh\psi_{\vec i}(\nu, y)] \f({\vec{i}\over n})\Big| \nonumber\\
&\dis= \Big|{N_n\over N_n+1} \sum_{\vec i\in\dbZ^d_n} [\psi_{\vec i}(\mu) - \psi_{\vec i}(\nu)] \f({\vec{i}\over n})\Big| = \Big|\int_{\mathbb R^d} \tilde \f(x) [\mu (dx) - \nu(dx)]\Big|,
\eea
where $\tilde \f(x) := {N_n\over N_n+1}\ch(x) \sum_{\vec j\in\dbZ^d_n} \phi_{\vec j}(x) \f({\vec j\over n})$. For any $x\in \supp(\ch)$, there exists $\vec i\in \dbZ^d_n$ such that $x\in \D_{\vec i}$.  Then, by \reff{phi},
\beaa
\pa_x\tilde \f(x) &=& {N_n\over N_n+1}\pa_x \ch(x) \sum_{\vec j\in J_{\vec i}} \phi_{\vec j}(x) \f({\vec j\over n})+{N_n\over N_n+1}\ch(x) \sum_{\vec j\in J_{\vec i}} \pa_x\phi_{\vec j}(x) \f({\vec j\over n})\\
&=&{N_n\over N_n+1}\pa_x \ch(x) \sum_{\vec j\in J_{\vec i}} \phi_{\vec j}(x) \f({\vec j\over n})+{N_n\over N_n+1}\ch(x) \sum_{\vec j\in J_{\vec i}} \pa_x\phi_{\vec j}(x)[ \f({\vec j\over n}) -\f(x)].
\eeaa
Recall $|\pa_x \ch|\le {C\over n}$, $ |\f({\vec j\over n})|\le 2n$ for $\vec j\in \dbZ^d_n$, and $|\pa_x\phi_{\vec j}(x)|\le Cn$, $|{\vec j\over n}-x|\le {C\over n}$ for $j\in J_{\vec i}$. Then
\beaa
|\pa_x\tilde \f(x)| 
&\le&{C\over n} \sum_{\vec j\in J_{\vec i}} \phi_{\vec j}(x) (2n)+ \sum_{\vec j\in J_{\vec i}} (Cn){C\over n} \le C.
\eeaa
Thus by \reff{W1} we have $\cW_1(\mu_n(y),\nu_n(y)) \le C \cW_1(\mu, \nu)$ for all $y\in \D_n$. Therefore,
\beaa
|U_n(\mu)-U_n(\nu)| &\le& \int_{\D_n} \zeta_n(y)  \big|U\big(\mu_n(y)\big) - U\big(\nu_n(y)\big)\big| dy \le \int_{\D_n} \zeta_n(y)  L\cW_1(\mu_n(y),\nu_n(y)) dy\\
&\le&  \int_{\D_n} \zeta_n(y)  CL\cW_1(\mu,\nu) dy = CL\cW_1(\mu,\nu),
\eeaa
where $L$ is the Lipschitz constant of $U$. Thus we obtain the desired uniform Lipschitz continuity of $U_n$.

(iii) This result is interesting in its own right, but will not be used in the rest of the paper. Since the proof is quite lengthy, in order not to distract our main focus on master equations, we postpone it  to Appendix.
\qed

\begin{rem}
\label{rem-mollifier}
{\rm (i) If $U$ is Lipschitz continuous under $\cW_2$ with a Lipschitz constant $L$,  in general $U_n$ may not be uniformly Lipschitz continuous under $\cW_2$ with a common Lipschitz constant $CL$, see Example \ref{eg-Lipschitz} below.  Nevertheless, since $\cW_1 \le \cW_2$, if $U$ is Lipschitz continuous under $\cW_1$ with a Lipschitz constant $L$ as in Theorem \ref{thm-mollifier} (ii), then $U_n$ is also uniformly Lipschitz continuous under $\cW_2$ with the same Lipschitz constant $CL$. 

(ii) In Theorem \ref{thm-mollifier} (iii), our $U_n$ does not satisfy (recalling Remark \ref{rem-compact})
\beaa
\lim_{n\to\infty}\sup_{(\mu,x)\in (\cM\cap \cP_2)\times K}  |\pa_\mu U_n(\mu, x) - \pa_\mu U(\mu, x)|=0.
\eeaa 
See Example \ref{eg-mollifier} below. It will be interesting to know if there exists an alternative mollifier such that the above uniform convergence holds for $U\in C^1(\cP_2)$. 

(iii) If $\mu$ has a continuous density, \reff{pauUconv} clearly implies that 
\bea
\label{pauUconv2}
\lim_{n\to\infty}  \int_{K} |\pa_\mu U_n(\mu, x) - \pa_\mu U(\mu, x)| \mu(dx)=0.
\eea
However, this may not be true when $\mu$ is discrete, see also Example \ref{eg-mollifier} below.
\qed}
\end{rem}

\bs

For later purpose, we need mollify functions $U: \dbR^d \times \cP_1\to \dbR$. Let $\z_0$ be another density function with support $\{x\in \dbR^d: |x|\le 1\}$. Define
\bea
\label{Uen}
U_{n}(x,\mu) := \int_{\dbR^d}  U^\mu_n(x-{y\over n}, \mu) \z_0(y) dy,
\eea
where $U^\mu_n(x,\mu)$ is the mollification in $\mu$ constructed in this section, for any fixed $x$. Then we may easily extend Theorem \ref{thm-mollifier} to this case, and we omit the proof.

\begin{thm}
\label{thm-mollifierx}
Let $U \in C^0(\dbR^d\times \cP_1)$ and $U_{n}$ be defined by \reff{Uen}. Then 

(i) $U_{n} \in C^\infty(\dbR^d\times \cP_2)\cap C^0(\dbR^d\times \cP_1)$ and $\lim_{n\to\infty} \|U_{n} - U\|_{L^\infty(K\times \cM)} =0$, for any $K\subset\subset  \dbR^d, \cM\subset\subset \mathcal{P}_1$. 

(ii) If $U$ is Lipschitz continuous in $(x, \mu)$ (under $\cW_1$ for $\mu$) with Lipschitz constant $L$, then $U_{n}$ is uniformly Lipschitz continuous in $(x, \mu)$ (under $\cW_1$ for $\mu$) with Lipschitz constant $CL$, where $C$ may depend on $d$, but not on $n$.

(iii) Assume $U\in C^0(\dbR^d\times \cP_2)$ such that $\pa_\mu U(x,\mu, \tilde x)$ exists and is uniformly continuous in $K_1\times (\cM \cap \cP_2) \times K_2$, again under $\cW_1$ for $\mu$, where $\cM\subset\subset \cP_1$ and $K_1, K_2\subset\subset\mathbb R^d$, then
\bea
\label{pauUconvx}
\lim_{n\to\infty} \sup_{x\in K_1, \mu\in \cM\cap \cP_2} \int_{K_2} |\pa_\mu U_n(x, \mu, \tilde x) - \pa_\mu U(x, \mu, \tilde x)| d\tilde x=0.
\eea
\end{thm}
 
 \begin{rem}
 \label{rem-mollifier-mon}
 {\rm For $U$  satisfying the monotonicity condition \reff{mon}, it is unlikely that the smooth mollifier  $U_{n}$  will also satisfy \reff{mon}, see Example \ref{eg-mollifier-mon} below, and we doubt any good mollifier will maintain the monotonicity property. It will be very interesting if we can find alternative sufficient conditions for the global wellposedness of the master equation, as mentioned in Remark \ref{rem-mon} (iii), which can be inherited by our smooth mollifier.
 \qed}
 \end{rem}
 
 \section{Some crucial estimates for the function $V$}
\label{sect-regularity}
\setcounter{equation}{0}

We start with investigating the non-mean field (standard) equations for a given  $\rho$. The following results are also standard, and for completeness we provide a proof in Section \ref{sect-proof}.  

\begin{prop}
\label{prop-classical-x}
Assume  Assumptions \ref{assum-Lip1} and \ref{assum-H1} hold. Let $\rho: [0, T]\times \O \to \cP_2$ be $\dbF^{0}$-progressively measurable (not necessarily a solution to \reff{SPDE}) with $\sup_{0\le t\le T} \dbE[ \|\rho_t\|_2^2]<\infty$. 

(i) For any $x\in \dbR^d$ and for the $X^x$ in \reff{FBSDE2}, the following BSDE has a unique solution:
\bea
\label{FBSDE-x}
Y_t^{x}=G(X_T^{x},\rho_T)+\int_t^T[F(X_s^{x},\rho_s) + H(X_s^{x},Z_s^{x})]ds-\int_t^T Z_s^{x}\cd dB_s-\int_t^TZ_s^{0,x}\cd dB_s^{0}.
\eea

(ii) The BSPDE in \reff{SPDE} (with the given $\rho$) has a weak solution $(u, v)$ in the sense of \reff{BSPDEweak} with $u$ differentiable in $x$, and it holds that 
\bea
\label{FBSDE-xu}
Y^x_t = u(t, X^x_t), \q Z^x_t = \pa_x u(t, X^x_t)= \td Y^x_t, \q Z^{0,x}_t = [v + \b \pa_x u](t, X^x_t),
\eea
where, with $\td Z^x, \td Z^{0,x}$ taking values in $\dbR^{d\times d}$,
\bea
\label{FBSDE-pax}
  \left.\ba{lll} 
\dis \nabla Y_t^{x}=\pa_x G( X_T^{x},\rho_T)+\int_t^T[\pa_x F( X_s^{x},\rho_s)+\pa_x H( X_s^{x}, Z_s^{x})+\nabla Z_s^{x}\pa_p H( X_s^{x}, Z_s^{x})]ds\\
\dis \qquad\qquad\qquad\qquad\quad\, -\int_t^T\nabla Z_s^{x} dB_s^{t_0}-\int_t^T\nabla Z_s^{0,x} dB_s^{0,t_0},\q t_0\le t\le T.
\ea\right.
\eea
Moreover, the following estimate hold:
\bea
\label{pauest1}
|\pa_x u(t,x)| \le C_1, 
\eea
where $C_1$ depends on $d, T$, the $L_1$ in Assumption \ref{assum-Lip1}, and the $c^H_1, L^H_1$  in \reff{Hassum1}.

(iii) Assume further that $\pa_x F, \pa_x G$ are uniformly Lipschitz continuous in $x$ with a Lipschitz constant $L_2$, then $\pa_x u$ is uniformly Lipschitz continuous in $x$, with a Lipschitz constant $C_2$ depending additionally on $L_2$ and the $L^H_2$  in Assumption \ref{assum-H1}.
\end{prop}

We next turn to the coupled systems \reff{SPDE} and \reff{FBSDE1}, where $\rho$ is part of the solution. As standard in the literature, see e.g. \cite{CD2},  these systems are wellposed locally in time, namely when the time duration $T$ is small. 
\begin{prop}
\label{prop-local}
Let  Assumptions \ref{assum-Lip1} and \ref{assum-H1}  hold. Then there exists a constant $\d_1>0$, which depends only on $d$, the $L_1$ in Assumption \ref{assum-Lip1}, and the $c^H_1, L^H_1$  in \reff{Hassum1} such that the following hold whenever $T\le \d_1$. 

(i) The FBSDE \reff{FBSDE1} has a unique strong solution and the FBSPDE \reff{SPDE} has a unique weak solution. In particular, \reff{uvYZ} and \reff{pauest1} hold true.

(ii) Assume $F, G$ and $H$ are sufficiently smooth, then the master equation \reff{master} has a unique classical solution $V$, and \reff{Vuv} holds true. 
\end{prop}

For completeness we shall sketch a proof in Section \ref{sect-proof} below. We emphasize that the $\d_1$ does not depend on the second derivatives of the data. In fact, $\d_1$ depends on the Lipschitz constant of $F, G$ with respect to $\mu$ under $\cW_2$. However, due to the reason explained in Remarks \ref{rem-mollifier} (i), here we use $\cW_1$. Moreover, in Section \ref{sect-representation} below we shall provide a pointwise representation formula for the derivatives of $V$, provided their existence. These formulas are new and, although not used in this paper,  interesting in their own rights. 

\ms
We now focus on an a priori stability estimate for the FBSDEs \eqref{FBSDE1}-\reff{FBSDE2}, which relies heavily on the monotonicity condition \reff{mon}. The corresponding estimates for the FBSPDE \reff{SPDE} has been shown by a PDE argument, see \cite{Cardaliaguet,CDLL}. We shall instead use pure probabilistic approach, where the related FBSDEs have strong solutions. While essentially in the same spirit as the PDE method, our approach  is more convenient to work with data less regular than those required for classical solution theory, and it seems new in the mean field literature, to our best knowledge. 

We first note that the monotonicity condition \reff{mon} is equivalent to: 
\bea
\label{mon2}
\dbE\Big[\Phi(\xi_1, \cL_{\xi_1}) + \Phi(\xi_2, \cL_{\xi_2}) - \Phi(\xi_1, \cL_{\xi_2}) - \Phi(\xi_2, \cL_{\xi_1})\Big] \le 0,\q\forall \xi_1, \xi_2\in \dbL^2(\cF).
\eea

\begin{thm}
\label{thm-apriori}
For $i=1,2$, assume $F_i, G_i, H_i$ satisfy Assumption \ref{assum-Lip1}, \ref{assum-Lip2},  \ref{assum-H1}, \ref{assum-H2}; and FBSDEs \reff{FBSDE1}-\reff{FBSDE2} with data $(F_i, G_i, H_i)$ and initial conditions $(x_i, \xi_i)\in \dbR^d\times \dbL^2(\cF_{t_0})$ has a strong solution $(\Phi^{\xi_i}, \Phi^{0,\xi_i})$, $\Phi= X, Y, Z, Z^{0}$. 
If $(F_1, G_1) ~($or $(F_2, G_2))$  satisfies \reff{mon2}, then there exist constants $C, R>0$, depending only on $T$, the dimensions, and the parameters in the Assumptions, such that: denoting $\rho^i_t:= \cL_{X^{\xi_i}_t|\cF^{0}_t}$,
\bea
\label{apriori}
\left.\ba{c}
\dis |\D Y^{x, \xi}_{t_0} | \le C\Big[|\D x| +\dbE_{\cF^0_{t_0}}[|\D \xi|] + |\D I_{t_0}|\Big],~\mbox{a.s.}\q\mbox{where} \\
\dis \D x:= x_1-x_2;\q \D \xi:= \xi_1-\xi_2;\q \D \Phi:= \Phi_1-\Phi_2,\q \Phi = G, F, H;\\
\dis \D \Phi^\xi := \Phi^{\xi_1}-\Phi^{\xi_2},\q \D \Phi^{x,\xi} := \Phi^{x_1,\xi_1}-\Phi^{x_2,\xi_2},\q \Phi=X, Y, Z;\\
\dis |\D I_t|^2:= \dbE_{\cF^0_t}\Big[ \sup_{x\in \dbR^d} \big[|\D G(x, \rho^2_T)|^2 +|\pa_x\D G(x, \rho^2_T)|^2\big] \\
\dis+ \int_t^T \sup_{x\in \dbR^d} \big[|\D F(x, \rho^2_s)|^2 +|\pa_x\D F(x, \rho^2_s)|^2\big]ds\Big] \\
\dis+\sup_{(x,z)\in D_R}\Big[ |\D H(x,z)|^2 + |\pa_x \D H(x,z)|^2 + |\pa_p \D H(x,z)|^2\Big].
\ea\right.
\eea
\end{thm}

\begin{proof} For notational simplicity, we assume $t_0=0$.
Given $\rho^i$, by Proposition \ref{prop-classical-x} there exists corresponding $u_i$ such that \reff{uvYZ} holds, $|\pa_x u_i|\le R$ for some constant $R$, and $\pa_x u_i$ is Lipschitz continuous in $x$ with Lipschitz constant $R$. In particular, this implies that $|Z^{\xi_i}|, |Z^{x_i, \xi_i}|\le R$. We proceed in four steps.

{\it Step 1.} For $i=1, 2$, let $(\cY^i, \cZ^i)$ solve the following BSDE: for $j =3- i$ (namely $j\neq i$), 
\bea
\label{tildeYi}
\left.\ba{c}
\dis \cY^i_t = G_j(X^{\xi_i}_T, \rho_T^j)  + \int_t^T \big[F_j(X^{\xi_i}_s,  \rho_s^j) +H_j(X_s^{\xi_i},\cZ_s^i) - \cZ^i_s\cd \pa_p H_i(X^{\xi_i}_s, Z^{\xi_i}_s) \big] ds \\
\dis - \int_t^T \cZ^i_s\cd dB_s-\int_t^T \cZ^{0,i}_s\cd dB_s^{0}.
\ea\right.
\eea
We note that $\cY^i_t$ corresponds to $u_j(t, X^{\xi_i}_t)$, and \reff{tildeYi} follows directly from  the It\^o-Wentzell formula when $u_j$ is smooth.  However, here we do not need such smoothness of $u_j$. 
Denote 
\beaa
&\tilde\D Y^i := Y^{\xi_i}-\cY^i,\q \tilde\D Z^i:= Z^{\xi_i}- \cZ^i,\q \tilde\D Z^{0,i}:= Z^{0,{\xi_i}}- \cZ^{0,i},\\
&\tilde\D Y_t := \tilde\D Y^1_t + \tilde\D Y^2_t,\q \tilde\D Z_t :=\tilde\D Z^1_t + \tilde\D Z^2_t,\q \tilde\D Z_t^{0} :=\tilde\D Z^{0,1}_t + \tilde\D Z^{0,2}_t.
\eeaa
Recall $\h L(x,z) = z\cd \pa_p H(x,z) - H(x,z)$.  By \reff{FBSDE1} and \reff{tildeYi}  we have
\bea
\label{DeltaY}
&\dis \tilde\D Y_t = \sum_{i=1}^2 [G_i(X^{\xi_i}_T, \rho_T^i) -G_j(X^{\xi_i}_T, \rho_T^j)]  + \int_t^T \sum_{i=1}^2\big[F_i(X^{\xi_i}_s,  \rho_s^i)-F_j(X^{\xi_i}_s,  \rho_s^j)  \\
&\dis +H_i(X_s^{\xi_i}, Z_s^{\xi_i})- H_j(X_s^{\xi_i},\cZ_s^i) - \tilde\D Z^{i}_s\cd \pa_p H_i(X^{\xi_i}_s, Z^{\xi_i}_s) \big] ds - \!\!\!\int_t^T\!\!\! \tilde\D Z_s\cd dB_s-\!\!\!\int_t^T\!\!\! \tilde\D Z^{0}_s\cd dB_s^{0}.\nonumber
\eea
Since $(F_1, G_1)$ satisfies the monotonicity condition \reff{mon2}  and $H_i$ satisfies the convexity condition \reff{Hconvex}, we have
\bea
\label{mon-convex}
\left.\ba{c}
\dis\dbE_{\cF^{0}_t}\Big[\Phi(X^{\xi_1}_s, \rho^1_s) + \Phi(X^{\xi_2}_s, \rho^2_s) -\Phi(X^{\xi_1}_s, \rho^2_s)-\Phi(X^{\xi_2}_s, \rho^1_s)\Big] \le 0, ~\Phi = G, F;\\
\dis H_i(X_s^{\xi_i},Z_s^{\xi_i})-H_i(X_s^{\xi_i},\cZ_s^i)-\tilde\D Z_s^i\cd\pa_pH_i(X_s^{\xi_i},Z_s^{\xi_i}) \le -{c^H_2(R)\over 2} |\tilde\D Z^i_s|^2,~ i=1,2.
\ea\right.
\eea
Set $t=0$ in  \reff{DeltaY}, take expectation $\dbE=\dbE_{\cF^0_0}$, and plug the above into it, we have
\bea
\label{DeltaY1}
\left.\ba{c}
\dis \dbE[\tilde\D Y_0] \le   \dbE\Big[\D G(X^{\xi_1}_T, \rho_T^2) - \D G(X^{\xi_2}_T, \rho_T^2) +\int_0^T \big[[\D F(X^{\xi_1}_s,  \rho_s^2) -\D F(X^{\xi_2}_s,  \rho_s^2)] \\
\dis + [\D H(X_s^{\xi_1},\cZ_s^1)-\D H(X_s^{\xi_2},\cZ_s^2)]  -{c^H_2(R)\over 2}\sum_{i=1}^2 |\tilde\D Z^i_s|^2  \big] ds\Big].
\ea\right.
\eea

Let $\e>0$ be a constant which will be specified later. Note that $u_i(0,\cd)$ is $\cF^0_0$-measurable and hence deterministic, then
\bea
\label{DGest}
\left.\ba{lll}
\dis \dbE[|\tilde\D Y_0|\big] = \dbE\Big[\big|\D u(0, \xi_1) - \D u(0, \xi_2)\big|\Big]\ms \\
\dis\qq \le \|\pa_x \D u(0,\cd)\|_\infty \dbE[|\D \xi|] \le \e \|\pa_x \D u(0,\cd)\|_\infty^2 + C_\e \big(\dbE[|\D\xi|]\big)^2; \ms\\
\dis\dbE\Big[ \Big|\D G(X^{\xi_1}_T, \rho_T^2) - \D G(X^{\xi_2}_T, \rho_T^2)\Big|\Big] \le \dbE\Big[\|\pa_x \D G(\cd, \rho_T^2)\|_\infty |\D X^\xi_T|\Big] \ms\\
\dis\qq =  \dbE\Big[\|\pa_x \D G(\cd, \rho_T^2)\|_\infty\dbE_{\cF^0_T}[ |\D X^\xi_T|]\Big]  \le \e \dbE\Big[\big(\dbE_{\cF^0_T}[ |\D X^\xi_T|]\big)^2\Big] + C_\e |\D I_0|^2;\ms \\
\dis \dbE\Big[\!\int_0^T\!\!\!\big|\D F(X^{\xi_1}_s,  \rho_s^2) -\D F(X^{\xi_2}_s,  \rho_s^2)\big|ds\Big]  \le \e\sup_{ s\in [0,T]} \dbE\Big[\big(\dbE_{\cF^0_s}[ |\D X^\xi_s|]\big)^2\Big]\!+ C_\e |\D I_0|^2; \ms\\
\dis |\cZ^1_s - \cZ^2_s| = \big|\pa_x u_2(s,X^{\xi_1}_s) - \pa_x u_2(s,X^{\xi_2}_s) + \tilde \D Z^2_s \big| \le C\big[|\D X^\xi_s|+ |\tilde \D Z^2_s|\big];\ms\\
\dis \dbE\Big[\Big|\D H(X_s^{\xi_1},\cZ_s^1)-\D H(X_s^{\xi_2},\cZ_s^2)\Big|\Big]  \le  \|\pa_x\D H\|_R \dbE[|\D X_s|] + \|\pa_p H\|_R \dbE[|\cZ^1_s - \cZ^2_s|]\ms\\
\dis \qq\qq\qq\qq\qq\qq\qq \le \e \dbE\Big[\big(\dbE_{\cF^0_s}[ |\D X^\xi_s|]\big)^2+ |\tilde \D Z^2_s|^2\Big]  + C_\e |\D I_0|^2,
\ea\right.
\eea
where the estimate for $\cZ^i$ used the Lipschitz continuity of $\pa_x u_j$.
Then \reff{DeltaY1} leads to that
\beaa
\left.\ba{c}
\dis \sum_{i=1}^2  \dbE\Big[ \int_0^T  |\tilde\D Z^i_s|^2 ds\Big] \le  C\e\dbE\Big[ \int_0^T  |\tilde\D Z^2_s|^2 ds\Big]+ C\e\sup_{s\in [0,T]}\dbE\Big[ \big[\big(\dbE_{\cF^0_s}[ |\D X^\xi_s|]\big)^2 \Big]   \\
\dis + C\e \|\pa_x \D u(0,\cd)\|_\infty^2 + C_\e \big(\dbE[|\D\xi|]\big)^2+C_\e|\D I_0|^2. 
\ea\right.
\eeaa
By choosing $\e$ small enough, we obtain
\bea
\label{DeltaZest}
\left.\ba{c}
\dis \sum_{i=1}^2  \dbE\Big[ \int_0^T  |\tilde\D Z^i_s|^2 ds\Big] \le C\e\Big[\sup_{ s\in [0,T]}\dbE\Big[ \big[\big(\dbE_{\cF^0_s}[ |\D X^\xi_s|]\big)^2 \Big]  + \|\pa_x \D u(0,\cd)\|_\infty^2\Big]\\
\dis + C_\e \Big[\big(\dbE[|\D\xi|]\big)^2+|\D I_0|^2\Big]. 
\ea\right.
\eea

{\it Step 2.} We next estimate $\dbE_{\cF^0_t}[ |\D X^\xi_t|]$. Note that 
\beaa
\D X^\xi_t = \D \xi + \int_0^t [\pa_p H_1(X^{\xi_1}, Z^{\xi_1}_s) - \pa_p H_2(X^{\xi_2}, Z^{\xi_2}_s)] ds.
\eeaa
Then, 
\beaa
 \dbE_{\cF^0_t}[|\D X^\xi_t|] &\le& \dbE[|\D \xi|] + \int_0^t \dbE_{\cF^0_s}\Big[\Big|\pa_p H_1(X^{\xi_1}_s, Z^{\xi_1}_s) - \pa_p H_2(X^{\xi_2}, Z^{\xi_2}_s)\Big|\Big] ds\\
&\le& \dbE[|\D\xi|] + C \int_0^t \dbE_{\cF^0_s}\Big[\big|\pa_p \D H(X^{\xi_1}_s, Z^{\xi_1}_s)\big| + |\D X^\xi_s| + |\D Z^\xi_s| \Big] ds.
\eeaa
Similar to the estimate for $|\cZ^1_s - \cZ^2_s|$ in \reff{DGest}, we have $|\D Z^\xi_s| \le C|\D X^\xi_s|+  C|\tilde \D Z^2_s|$. Then
\beaa
 \dbE_{\cF^0_t}[|\D X^\xi_t|] \le \dbE[|\D\xi|] + C \int_0^t \dbE_{\cF^0_s}\big[ |\D X^\xi_s|  +   |\tilde \D Z^2_s|\big] ds + C|\D I_0|.
\eeaa
Apply the Gronwall inequality, we have
\bea
\label{Wrhoest}
\sup_{t\in [0,T]} \dbE_{\cF^0_t}[|\D X^\xi_t|] \le  C\int_0^T \dbE_{\cF^0_t}[|\tilde \D Z^2_s|]ds  + C\dbE[|\D\xi|]  + C|\D I_0|.
\eea

{\it Step 3.} We now estimate $\|\pa_x \D u(0,\cd)\|_\infty$. Fix an arbitrary $x\in \dbR^d$ (not necessary $x_i$). Following the same arguments for Proposition \ref{prop-classical-x}, see \reff{FBSDE-x-modified} and \reff{FBSDE-pax-modified} below, we have $\pa_x u_i(t, X^{x}_t) = \td Y^{i,x}_t= Z^{i, x}_t$, where, for the $X^x_t$ in \reff{FBSDE2} with $t_0=0$,
\bea
\label{paxurep}
\left.\ba{lll}
\dis \q Y_t^{i,x}=G_i(X_T^{x},\rho^i_T)- \int_t^TZ_s^{i,x}\cd dB_s-\int_t^T Z_s^{0,i, x}\cd dB_s^{0};\\
\dis\qq\q + \int_t^T[F_i(X_s^{x},\rho^i_s)+H_i(X_s^{x},Z_s^{i,x} )]ds\\
\dis \td Y_t^{i,x}=\pa_x G_i(X_T^{x},\rho^i_T)  -\int_t^T\nabla Z_s^{i,x} dB_s-\int_t^T\nabla Z_s^{0,i,x} dB_s^{0}\\
\dis\qq \q+\int_t^T[\pa_x F_i( X_s^{x},\rho^i_s)+\pa_x H_i(X_s^{x}, Z_s^{i,x} ) +\td  Z_s^{i,x} \pa_p H_i(X_s^{x}, Z_s^{i,x} )  ]ds.
\ea\right.
\eea
We first note that
\bea
\label{cWrhoest}
\cW_1(\rho^1_t, \rho^2_t)\le \dbE_{\cF^0_t}[|\D X^\xi_t|].
\eea
Applying standard BSDE estimates on the equation for  $ Y_t^{i,x}$, we have
\beaa
\dbE\Big[\int_0^T \big| Z_t^{1,x}- Z_t^{2,x}\big|^2dt\Big] 
\le C\sup_{t\in [0, T]} \dbE\Big[\big(\dbE_{\cF^0_t}[|\D X^\xi_t|]\big)^2\Big] +C|\D I_0|^2.
\eeaa
Moreover, since $\pa_{x} u_i$ is uniformly Lipschitz continuous in $x$, we see that $\td  Z_s^{i,x}$ is bounded.  Then, applying standard BSDE estimates on the equation for  $\td Y_t^{i,x}$, we have
\beaa
&&\big|\pa_x \D u(0, x)\big|^2 = \big|\td  Y_0^{1,x} - \td  Y_0^{2,x}\big|^2\\
&&\le C\sup_{t\in [0, T]} \dbE[\cW^2_1(\rho^1_t, \rho^2_t)]  +C|\D I_0|^2+ C\dbE\Big[\int_0^T | Z_s^{1,x}- Z_s^{2,x}|^2\big]ds\Big]\\
&&\le C\sup_{t\in [0, T]} \dbE\Big[\big(\dbE_{\cF^0_t}[|\D X^\xi_t|]\big)^2\Big] +C|\D I_0|^2.
\eeaa
Since $x$ is arbitrary, we obtain
\bea
\label{suppaxu}
\|\pa_x \D u(0, \cd)\|_\infty^2 \le C\sup_{t\in [0,T]} \dbE\Big[\big(\dbE_{\cF^0_t}[|\D X^\xi_t|]\big)^2\Big]  +C|\D I_0|^2.
\eea
Combine this with \reff{Wrhoest}, one can easily show that
\bea
\label{Wrhoest2}
\left.\ba{c}
\dis \|\pa_x \D u(0, \cd)\|_\infty^2+\sup_{ t\in [0,T]} \dbE\Big[\big(\dbE_{\cF^0_t}[|\D X^\xi_t|]\big)^2\Big]\\
\dis \le C \dbE\Big[\int_0^T |\tilde \D Z^2_s|^2ds\Big]+ C \big(\dbE[|\D\xi|]\big)^2 +   C|\D I_0|^2.
\ea\right.
\eea
Plug this into \reff{DeltaZest} and set $\e$ small enough, we have
\bea
\label{DeltaZest2}
\dis \sum_{i=1}^2  \dbE\Big[ \int_0^T  |\tilde\D Z^i_s|^2 ds\Big] \le  C\Big[  \big(\dbE[|\D\xi|]\big)^2+|\D I_0|^2\Big]. 
\eea
This, together with \reff{Wrhoest2}, implies further that
\bea
\label{suppaxu2}
 \|\pa_x \D u(0, \cd)\|_\infty^2+\sup_{t\in [0,T]} \dbE\Big[\big(\dbE_{\cF^0_t}[|\D X^\xi_t|]\big)^2\Big]\le  C\Big[ \big(\dbE[|\D\xi|]\big)^2 +   |\D I_0|^2\Big].
\eea

{\it Step 4.} Consider FBSDE \reff{FBSDE2} for $(X^{x_i}, Y^{x_i,\xi_i}, Z^{x_i,\xi_i}, Z^{0, x_i,\xi_i})$. 
Note that $\D X^x_t = \D x$ and $Z^{x_i,\xi_i}$ is bounded, by  \reff{cWrhoest} and \reff{suppaxu2}, it follows from standard BSDE arguments that
\beaa
|\D Y^{x,\xi}_0|^2 &\le& C\dbE\Big[|\D x|^2 + \cW^2_1(\rho^1_T, \rho^2_T) + \int_0^T \big[|\D x|^2 + \cW^2_1(\rho^1_s, \rho^2_s)\big]ds\Big]\\
&\le& C\Big[ |\D x|^2 + |\D I_0|^2 + \sup_{0\le t\le T} \dbE\big[ \big(\dbE_{\cF^0_t}[|\D X^\xi_t|]\big)^2\big]  \Big]\\
&\le& C\Big[ |\D x|^2 +\big(\dbE[|\D\xi|]\big)^2 +  |\D I_0|^2 \Big],
\eeaa 
This implies \reff{apriori} at $t_0=0$ immediately.
\end{proof}

We remark that the (candidate) solution $V$ of the master equation plays the role of the decoupling field for the FBSDE. 
As illustrated in \cite{Delarue, MWZZ, MYZ}, to extend from a local (in time) solution of an FBSDE to a global solution, the key is the uniform Lipschitz continuity of the decoupling field, which is exactly implied by \reff{apriori}.  We can thus establish the wellposedness of FBSDE \reff{FBSDE1} rigorously.  
\begin{thm}
\label{thm-well-posedness}
Assume that $F,G,H$ satisfy Assumption \ref{assum-Lip1}, \ref{assum-Lip2}, \ref{assum-H1}, \ref{assum-H2}, and \ref{assum-mon}.

(i) The FBSDEs  \eqref{FBSDE1}-\reff{FBSDE2} are  wellposed. Consequently,  for any $(t_0, x, \mu)\in [0, T]\times \dbR^d\times \cP_2$ and any $\xi\in \dbL^2(\cF_0\vee \cF^B_t, \mu)$, it induces a deterministic function:
\begin{equation}\label{VY}
V(t_0,x,\mu):=Y_{t_0}^{x,\xi}.
\end{equation}

(ii) Both $V$ and $\pa_x V$ are uniformly Lipschitz continuous in $(x, \mu)$, under $\cW_1$ for $\mu$, and H\"{o}lder-${1\over 2}$ continuous in $t$ in the following sense:
\bea
\label{Vtreg}
\left.\ba{c}
\dis \Big|V(t_0, x, \mu) - V(t_1, x, \mu)\Big| \le C \sqrt{t_1-t_0} +C \big[1+ |F(x, \mu)| + |H(x, 0)|\big]|t_1-t_0|;\\
\dis \Big|\pa_xV(t_0, x, \mu) - \pa_xV(t_1, x, \mu)\Big| \le C \sqrt{t_1-t_0}. 
\ea\right.
\eea

(iii) For any $t_0$, $V(t_0,\cdot,\cdot)$ satisfies the monotonicity condition \reff{mon} $(\mbox{or}~ \reff{mon2})$.
\end{thm}
\begin{proof} (i) The uniqueness is a direct consequence of Theorem \ref{thm-apriori}. To construct a solution for \reff{FBSDE1}, let $C_1$ denote the constant $C$ in \reff{apriori}, and let $\d_1$ be the constant in Proposition \ref{prop-local} but with the dependence on $L_1$ replaced with $L_1 \vee C_1$.  Fix a time partition $t_0=T_0<\cds<T_n = T$ such that $T_i - T_{i-1}\le \d_1$ for all $i$.

First, consider FBSDEs \reff{FBSDE1}-\reff{FBSDE2} on $[T_{n-1}, T_n]$ with initial condition $x\in\dbR^d$ and $\xi\in \dbL^2(\cF_{T_{n-1}})$, by Proposition \ref{prop-local} it has a solution, denoted as $(\Phi^{n-1, \xi}, \Phi^{n-1, x, \xi})$ for $\Phi=(X, Y, Z, Z^0)$.  Then, for any $\mu\in \cP_2$ and $\xi\in \dbL^2(\cF_0\vee \cF^B_{T_{n-1}}, \mu)$, one may define $V(T_{n-1}, x, \mu)$ by \reff{VY}. Given $\mu_1, \mu_2$, we may choose corresponding $\xi_1, \xi_2$ appropriately so that $\dbE[|\xi_1-\xi_2|] = \cW_1(\mu_1, \mu_2)$. Since we do not perturb $F, G, H$, namely $\D I = 0$, then \reff{apriori} implies
\bea
\label{VLip}
\Big|V(T_{n-1}, x_1, \mu_1) - V(T_{n-1}, x_2, \mu_2)\Big| \le C_1 \Big[|x_1-x_2| + \cW_1(\mu_1, \mu_2)\Big].
\eea
That is, $V(T_{n-1}, \cd,\cd)$ is uniformly Lipschitz continuous in $(x, \mu)$ with Lipschitz constant $C_1$, where the continuity in $\mu$ is under $\cW_1$.

Next, consider FBSDEs \reff{FBSDE1}-\reff{FBSDE2} on $[T_{n-2}, T_{n-1}]$ with initial condition $x\in\dbR^d$ and $\xi\in \dbL^2(\cF_{T_{n-2}})$, but the terminal condition is $V(T_{n-1}, \cd, \cd)$ instead of $G(\cd, \cd)$. By the Lipschitz continuity of $V(T_{n-1},\cd,\cd)$, it follows from Proposition \ref{prop-local} again that it has a solution, denoted as $(\Phi^{n-2, \xi}, \Phi^{n-2, x, \xi})$ for $\Phi=(X, Y, Z, Z^0)$.   Now define, for $\Phi = (X, Y, Z, Z^0)$,
\beaa
\left.\ba{c}
\dis \Phi^{\xi}_t := \Phi^{n-2,\xi}_t \1_{[T_{n-2}, T_{n-1}]}(t) + \Phi^{n-1, \Phi^{n-2,\xi}_{T_{n-1}}}_t\1_{(T_{n-1}, T_n]}(t),\ms\\
\dis \Phi^{x, \xi}_t := \Phi^{n-2,x, \xi}_t \1_{[T_{n-2}, T_{n-1}]}(t) + \Phi^{n-1, \Phi^{n-2,x, \xi}_{T_{n-1}}, \Phi^{n-2,\xi}_{T_{n-1}}}_t\1_{(T_{n-1}, T_n]}(t).
\ea\right. 
\eeaa 
One can easily verify that this provides a solution to FBSDEs \reff{FBSDE1}-\reff{FBSDE2} on $[T_{n-2}, T]$ with initial condition $(x,  \xi)$. By restricting to $\xi\in \dbL^2(\cF_0\vee\cF^B_{T_{n-2}}, \mu)$ we may define $V(T_{n-2}, x, \mu)$ by \reff{VY}, and by \reff{apriori}  again we see that $V(T_{n-2}, \cd,\cd)$ is uniformly Lipschitz continuous in $(x, \mu)$ with the same Lipschitz constant $C_1$.

Now repeat the arguments backwardly in time we can construct a solution to FBSDEs \reff{FBSDE1}-\reff{FBSDE2} on $[T_0, T]= [t_0, T]$. 

(ii) The uniform Lipschitz continuous of $V$ in $(x, \mu)$ has already been proved in (i). Next, note that $\pa_x V(0,x,\mu_i)=\pa_x u_i(0,x)$ for the $u_i$ in the proof of Theorem \ref{thm-apriori}. By Proposition \ref{prop-classical-x} (iii) we see that $\pa_x V(0,\cd,\cd)$ is uniformly Lipschitz continuous in $x$, and by \reff{suppaxu2} we have, recalling we may choose $\xi_i$ such that $\dbE[|\D \xi|] = \cW_1(\mu_1, \mu_2)$,
\beaa
\Big|\pa_x V(0,x,\mu_1) - \pa_x V(0,x, \mu_2)\Big| = |\pa_x \D u(0,x)|\le C\cW_1(\mu_1, \mu_2),
\eeaa
namely $\pa_x V(0, \cd)$ is also uniformly Lipschitz continuous in $\mu$ (under $\cW_1$). 

Moreover, for $t_0 < t_1 \le T$, $x\in \dbR^d$, $\mu\in \cP_2$, and $\xi\in \dbL^2(\cF_0\vee \cF^B_{t_0}, \mu)$, note that for the solution to \reff{FBSDE2} with initial time $t_0$, we have  $Y^{x,\xi}_{t_i} = V(t_i, X^{x, \xi}_{t_i}, \rho_{t_i})$, $i=0,1$. Then the backward  equation in \reff{FBSDE2} leads to
\beaa
 \dis V(t_0, x, \rho_{t_0})=V(t_1, X^{x, \xi}_{t_1}, \rho_{t_1}) -\int_t^T Z_s^{x,\xi}\cd dB_s- \int_t^T Z_s^{0,x,\xi}\cd dB_s^{0}\\
 \dis +\int_{t_0}^{t_1} \Big[F(X_s^{x,\xi},\rho_s) + H(X_s^{x,\xi}, Z_s^{x,\xi})-  Z_s^{x,\xi}\cd\pa_p H(X_s^{x,\xi},Z_s^{x,\xi})\Big]ds.
\eeaa
By the uniform Lipschitz continuity of $V$ with respect to $(x,\mu)$, we have
\beaa
&&\dis \Big|V(t_0, x, \mu) - V(t_1,x,\mu)\Big| = \Big|\dbE\Big[V(t_0, x, \rho_{t_0}) - V(t_1,x,\rho_{t_0})\Big]\Big|\\
&&\dis  \le \dbE\Big[\big|V(t_1, X^{x, \xi}_{t_1}, \rho_{t_1}) - V(t_1,x,\rho_{t_0})\big| \\
&&\dis \qq+ \int_{t_0}^{t_1} \big|F(X_s^{x,\xi},\rho_s) + H(X_s^{x,\xi}, Z_s^{x,\xi})-  Z_s^{x,\xi}\cd\pa_p H(X_s^{x,\xi},Z_s^{x,\xi})\big|ds\Big]\\
&&\dis  \le C\sup_{t_0\le s\le t_1} \dbE\Big[|X^{x,\xi}_s -x| + |X^\xi_s-\xi|\Big] + C \dbE\Big[ \int_{t_0}^{t_1} \big[|F(x, \mu)| + |H(x, 0)| +1\big]ds\Big]. \eeaa
One can easily see that 
\beaa
\sup_{t_0\le s\le t_1} \dbE\Big[|X^{x,\xi}_s -x| + |X^\xi_s-\xi|\Big] \le C\sqrt{t_1-t_0}.
\eeaa
Then we obtain immediately the estimate for $V$ in \reff{Vtreg}. 

Similarly, note that $\pa_x V(t, X^x_t, \rho_t) = \td Y^x_t$  for the $\td Y^x$ in \reff{FBSDE-pax}. By \reff{pauest1} $Z^x$ is bounded, then $\pa_x F$ and $\pa_x H(\cd, Z^x_t)$ are bounded. Now following similar arguments as above we can easily prove the second estimate in \reff{Vtreg}.

(iii) Note again that $\D G=\D F=\D H=0$. Then \reff{DeltaY1} implies that 
\beaa
0\ge \dbE[\tilde \D Y_0] =\dbE\Big[ V(0, \xi_1, \cL_{\xi_1}) - V(0, \xi_1, \cL_{\xi_2}) + V(0, \xi_2, \cL_{\xi_2}) - V(0, \xi_2, \cL_{\xi_1})\Big].
\eeaa
This exactly means that $V(0,\cd,\cd)$ satisfies the monotonicity condition \reff{mon2}. Similarly we can show $V(t,\cd,\cd)$ satisfies \reff{mon2} for all $t$.
\end{proof}

\begin{rem}
\label{rem-VcP2}
{\rm While the data $F, G$ are defined on $\cP_1$, due to Remark \ref{rem-mollifier} (i), the (candidate) solution $V$ is defined on $\cP_2$. However, since $\cP_2$ is dense in $\cP_1$ under $\cW_1$, the uniform Lipschitz continuity of $V$ in $\mu$ under $\cW_1$ enables us to extend $V$ to $\cP_1$ uniquely and the extended function is still uniformly Lipschitz continuous. So in this sense $V$ can also be viewed as a function on $[0, T]\times \dbR^d \times \cP_1$, and we shall do so whenever needed.
\qed}
\end{rem}
We note that the Assumption \ref{assum-Lip2} was used to obtain the Lipschitz continuity of $V$ with respect to $\mu$. If we fix $(x, \mu)$ and consider only the sensitivity with respect to the data $F, G, H$, this assumption is  actually not needed. We have the following stability result, provided that the FBSDEs have a solution.

\begin{thm}
\label{thm-stability}
For $i=1,2$, assume $F_i, G_i, H_i$ satisfy Assumption \ref{assum-Lip1}, \ref{assum-H1}, and \ref{assum-H2}; and FBSDEs \reff{FBSDE1}-\reff{FBSDE2} with data $(F_i, G_i, H_i)$ and the same initial conditions $(x, \xi)\in \dbR^d\times \dbL^2(\cF_{t_0})$ has a strong solution $(\Phi^i, \Phi^{i,x})$, $\Phi = X, Y, Z, Z^{0}$ (omitting the dependence on $\xi$ for notational simplicity). 
If $(F_1, G_1) ~($or $(F_2, G_2))$  satisfies the monotonicity condition \reff{mon2}, then there exist constants $C, R>0$, depending only on $T$, the dimensions, and the parameters in the Assumptions, such that: for the notations in Theorem \ref{thm-apriori},
\bea
\label{stability123}
\dis |\D Y^{x, \xi}_{t_0} | \le C [|\D I_{t_0}|^{1\over 4} + |\D I_{t_0}|],~\mbox{a.s.}.
\eea
\end{thm}
\begin{proof} We shall follow the proof of Theorem \ref{thm-apriori}, except that we cannot apply Proposition \ref{prop-classical-x} to claim the uniform Lipschitz continuity  of $\pa_{x} u_i$.  Again assume $t_0=0$ and we use the notation in \reff{apriori}, noticing though that $X^i$ here corresponds to $X^{\xi_i}$ there, and here $\D x = 0, \D \xi=0$.  We proceed in three steps.

{\it Step 1.} First note that
\beaa
\tilde \D Y_0 = \tilde \D Y^1_0 + \tilde \D Y^2_0 = [u_1(0, \xi) - u_2(0, \xi)] + [u_2(0, \xi) - u_1(0, \xi)]=0.
\eeaa
Then \reff{DeltaY1} implies
\bea
\label{DZest}
\left.\ba{c}
\dis \sum_{i=1}^2 \dbE\Big[\int_0^T |\tilde\D Z^i_s|^2ds\Big] \le  C |\D I_0|.
\ea\right.
\eea
Denote
\bea
\label{Girsanov}
\left.\ba{c}
\dis X_t := \xi + B_t + \b B^0_t,\q  \th^i_t := \pa_p H_i(X_t, \pa_x u_i(t, X_t)),\\
\dis M^i_t := \exp\Big(\int_0^t \th^i_s\cd dB_s - {1\over 2} \int_0^t |\th^i_s|^2ds\Big),\q {d\dbP_i\over d \dbP} := M^i_T.
\ea\right.
\eea
Under our conditions we have $|\th^i_t|\le C$. Since $B$ and $B^0$ are independent, it is clear that
\bea
\label{EMest}
\dbE\Big[|M^i_t|^p + |M^i_t|^{-p}\Big] \le C_p,\q\forall t, \forall p\ge 1.
\eea
Moreover, by Girsanov theorem the conditional $\dbP_i$-distribution of $X_t$, conditional on $\cF^0_t$, is equal to $\rho^i_t$. In particular, this implies that \reff{DZest} is equivalent to:
\bea
\label{DZest1}
 \sum_{i=1}^2 \dbE\Big[ M^i_T \int_0^T |\pa_x \D u(s, X_s)|^2ds\Big] \le  C |\D I_0|.
\eea

{\it Step 2.} We next estimate $\cW_1(\rho^1_t, \rho^2_t)$.  For any function $\f$ as in \reff{W1}, we have
\beaa
\dbE^{\dbP_1}_{\cF^0_t}[\f(X_t)]  - \dbE^{\dbP_2}_{\cF^0_t}[\f(X_t)] &=& \dbE_{\cF^0_t}\big[ M^1_t [\f(X_t)-\f(\xi)]\big]  - \dbE_{\cF^0_t}\big[M^2_t[\f(X_t)-\f(\xi)]\big] \\
&=& \dbE_{\cF^0_t}\Big[ [\f(X_t)-\f(\xi)] [M^1_t - M^2_t]\Big].
\eeaa
Note that $|e^{x_1}-e^{x_2}|\le [e^{x_1}+e^{x_2}] |x_1-x_2|$, then, denoting $\D \th := \th^1-\th^2$, 
\bea
\label{DMest}
\big|M^1_t - M^2_t\big| &\le& [M^1_t + M^2_t] \Big|\int_0^t \D \th_s\cd dB_s - {1\over 2} \int_0^t [|\th^1_s|^2-|\th^2_s|^2]ds\Big|\nonumber\\
&\le& [M^1_t + M^2_t] \Big[\big|\int_0^t \D \th_s \cd dB_s\big| +C \int_0^t |\D \th_s|ds\Big]
\eea
Note that
\bea
\label{Dthest}
|\D \th_t| \le |\pa_p \D H(X_t, \pa_x u_1(t, X_t))| + C|\pa_x \D u(t,X_t)| \le C|\pa_x \D u(t,X_t)|+|\D I_0|.
\eea
Then, 
\beaa
&&\Big|\dbE^{\dbP_1}_{\cF^0_t}[\f(X_t)]  - \dbE^{\dbP_2}_{\cF^0_t}[\f(X_t)] \Big| \\
&&\le \dbE_{\cF^0_t}\Big[  |X_t-\xi| [M^1_t + M^2_t]\big[\big|\int_0^t \D \th_s\cd dB_s\big| +C \int_0^t |\D \th_s|ds\big]\Big]\\
&&\le  C\dbE_{\cF^0_t}\Big[ [M^1_t+M^2_t] [|B_t|+ |B^0_t|]\big[\big|\int_0^t \D \th_s\cd dB_s\big| + \int_0^t |\D \th_s|ds\big]\Big]\\
&&\le C[1+|B^0_t|] \Big(\dbE_{\cF^0_t}\Big[ \int_0^t |\D u(s,X_s)|^2 ds\Big] + |\D I_0|^2\Big)^{1\over 2}.
\eeaa
Thus, by \reff{W1} and noting that $\pa_x u_i$ is bounded, we have
\beaa
\cW_1(\rho^1_t, \rho^2_t) &\le& C[1+|B^0_t|] \Big(\dbE_{\cF^0_t}\Big[ \int_0^t |\pa_x\D u(s, X_s)|^2 ds\Big]+ |\D I_0|^2\Big)^{1\over 2}\\
&\le& C[1+|B^0_t|] \Big(\dbE_{\cF^0_t}\Big[ \int_0^t |\pa_x\D u(s, X_s)| ds\Big]+ |\D I_0|^2\Big)^{1\over 2}.
\eeaa
Then it follows from \reff{DZest1} and \reff{EMest} that
\bea
\label{W1rhoest}
&&\dbE\Big[\cW^2_1(\rho^1_t, \rho^2_t)\Big] \le C\dbE\Big[[1+|B^0_t|^2]  \int_0^t |\pa_x\D u(s, X_s)| ds\Big]+ C |\D I_0|^2\nonumber\\
&&=C\dbE\Big[[1+|B^0_t|^2] (M^1_T)^{-{1\over 2}} (M^1_T)^{1\over 2}  \int_0^t |\pa_x\D u(s, X_s)| ds\Big]+ C |\D I_0|^2 \\
&&=C\Big(\dbE\Big[M^1_T \int_0^t |\pa_x\D u(s, X_s)|^2 ds\Big]\Big)^{1\over 2}+ C |\D I_0|^2  \le C\Big[|\D I_0|^{1\over 2} +  |\D I_0|^2\Big].\nonumber
\eea

{\it Step 3.} Finally, note that
\beaa
&|\Phi_1(X^x_t, \rho^1_t)-\Phi_2(X^x_t, \rho^2_t)| \le |\D \Phi(X^x_t, \rho^2_t)| + C \cW_1(\rho^1_t, \rho^2_t),\q \Phi = F, G;\\
&|H_1(X^x_t, Z^{1,x}_t )- H_2(X^x_t, Z^{2,x}_t)| \le  |\D H(X^x_t, Z^{2,x}_t)| + C|Z^{1,x}_t-Z^{2,x}_t|.
\eeaa
Applying standard BSDE estimates on the second equation of \reff{paxurep} we have
\bea
\label{DtildeZest}
|Y^{1,x}_0-Y^{2,x}_0|^2+\dbE\Big[\int_0^T |Z^{1,x}_t - Z^{2,x}_t|^2dt\Big] \le  C\sup_{t\in [0,T]} \dbE\Big[\cW^2_1(\rho^1_t, \rho^2_t)\Big] + C|\D I_0|^2. 
\eea
Note that $u_i(0,x) =Y^{i,x}_0$. Then by \reff{W1rhoest} we have
\beaa
&&|\D Y^{x,\xi}_0| = |\D u(0,x)| = |Y^{1,x}_0- Y^{2,x}_0| \\
&& \le C\sup_{ t\in [0,T]}\Big( \dbE\big[\cW^2_1(\rho^1_t, \rho^2_t)\big] \Big)^{1\over 2}+ C|\D I_0|\le C[|\D I_0|^{1\over 4} + |\D I_0|],
\eeaa
completing the proof.
\end{proof}

\begin{rem}
\label{rem-Pinsker}
{\rm When $\dbE[e^{c |\xi|^2}]<\infty$ for some $c>0$,  in the spirit of the Pinsker's inequality, by \cite[Proposition 6.3]{GL} we have
\beaa
\cW^2_1(\rho^1_t, \rho^2_t)\le C\dbE_{\cF^0_t}\Big[M^1_t\big[\int_0^t \D \th_s dB_s + {1\over 2} \int_0^t [|\th^1_s|^2-|\th^2_s|^2]ds\big]\Big]. 
\eeaa
Then one may  simplify the arguments in Step 2. For the general case, one can argue in this direction by first considering conditional law, conditional on $\cF_0\vee \cF^0_t$.  Our arguments here, however, are quite elementary. 
\qed}
\end{rem}

\section{The good solution of master equations}
\label{sect-good}
\setcounter{equation}{0}
In this section we propose the notion of good solution  for the master equation \reff{master},  with the name inherited from \cite{JKS}. The main idea is to utilize the mollification of the data and the stability result Theorem \ref{thm-apriori}. Roughly speaking,  let $(F_n, G_n, H_n)$ be a smooth mollifier of $(F, G, H)$, if the mollified master equation with data $(F_n, G_n, H_n)$ has a classical solution $V_n$, by Theorem \ref{thm-apriori} we see that $V_n$ has a unique limit $V$, and then we may define the limit function $V$ as the good solution of the original master equation \reff{master}.
However, note that $(F_n, G_n)$ may violate the monotonicity condition \reff{mon}, then Theorem \ref{thm-well-posedness} does not ensure the existence of classical solution $V_n$ for the mollified master equation. We shall instead apply Proposition \ref{prop-local} to obtain local (in time) classical solutions for the mollified master equation. This leads to the following notion of good solution.

First, let  $C^0_{Lip}(\dbR^d \times \cP_2)$ denote the set of uniformly Lipschitz continuous functions $\Phi: \dbR^d\times\cP_2\to \dbR$, under $\cW_1$ for $\mu$. Similarly, let $C^0_{Lip}(\Th)$ denote the set of $V\in C^0(\Th)$ such that $V$ is uniformly Lipschitz continuous in $(x,\mu)$, under $\cW_1$ for $\mu$, uniformly in $t\in [0, T]$. We note that, as explained in Remark \ref{rem-VcP2}, these functions can be extended from $\cP_2$ to $\cP_1$.

\begin{defn}
\label{defn-cA}
For any $0\le t_0< t_1\le T$ and any $\Phi\in C^0_{Lip}(\dbR^d \times \cP_2)$, let $\cA_{(F, H, \Phi)}([t_0, t_1])$ denote the set of  sequences $(F_n, H_n, \Phi_n)_{n\ge 1}$ satisfying

(i) For each $n$, $(F_n, H_n, \Phi_n)$ are sufficiently smooth;

(ii) $(F_n, H_n, \Phi_n)$ satisfy the regularity conditions in Assumptions \ref{assum-Lip1} and \ref{assum-H1}  uniformly in $n$, where the uniform regularity in Assumptions \ref{assum-Lip1} (ii) is in the following sense: for any $M>0$,  there exists a modulus of continuity function $\k_M$, which may depend on $M$, but is independent of  $n$, such that $\pa_x F_n, \pa_x \Phi_n$ are uniformly continuous on $Q_M\times \cP^M_2 (\subset\subset \dbR^d\times \cP_1)$ with the modulus of continuity function $\k_M$.

(iii) For any $M, R>0$, with the notation $\cP_2^M$ defined by \reff{cP} and $D_R$ defined by \reff{DR},
 \bea
 \label{Fnconv}
\left.\ba{c}
\dis \lim_{n\to\infty} \Big[ \sup_{(x,\mu) \in \dbR^d\times \cP_2^M} \big[|F_n-F| + |\Phi_n-\Phi| + |\pa_x F_n-\pa_x F|+|\pa_x \Phi_n-\pa_x \Phi|\big](x, \mu)\\
 \dis+ \sup_{(x,z)\in D_R}\big[ |H_n-H| + |\pa_x H_n-\pa_x H|+|\pa_p H_n-\pa_p H|\big](x,z)\Big]=0.
\ea\right.\eea
\end{defn}

\begin{defn}
\label{defn-good}
 We say $V \in C^{0}_{Lip}(\Th)$ with $V(T,\cd) = G$ is a good solution of master equation \reff{master} if, for any $0\le t_0< t_1\le T$ and any  $\{(F_n, H_n, \Phi_n)\}_{n\ge 1} \in \cA_{F, H, V(t_1,\cd,\cd)}([t_0, t_1])$,  there exists $\d \in (0, t_1-t_0]$ such that the master equation \reff{master} on $[t_1-\d, t_1]$ with data $(F_n, H_n)$ and terminal condition $\Phi_n$ has a classical solution $V_n$ and it holds that 
 \bea
 \label{Vnconv}
\lim_{n\to\infty}  |(V_n - V)(t, x, \mu)| =0\q \mbox{for all}\q  (t, x,\mu) \in [t_1-\d, t_1]\times \dbR^d\times \cP_2.
 \eea 
\end{defn}

\begin{thm}
\label{thm-good}
Let Assumptions \ref{assum-Lip1} and \ref{assum-H1} hold.

(i) The master equation \reff{master} has at most one good solution $V$. 

(ii) Assume further that either $T\le \d_1$ for the $\d_1$ in Proposition \ref{prop-local},  or Assumptions \ref{assum-Lip2},  \ref{assum-H2},  and \ref{assum-mon} hold. Then the function $V$ defined by \reff{VY}  is the unique good solution of the master equation \reff{master}. 
\end{thm}
proof (i) Assume by contradiction that  there are two  good solutions $V$ and $\hat V$. Denote 
\bea
\label{t1}
t_1 := \sup\big\{ t\in [0, T]: \mbox{there exists $(x, \mu)$ such that $\hat V(t, x, \mu) \neq V(t,x,\mu)$}\big\}>0.
\eea
 By the continuity of $V$ and $\hat V$ we have $V(t_1,\cd) = \hat V(t_1, \cd)=:\Phi$.  Let $(F_n,  \Phi_n)$ be the smooth mollifier of $(F,  \Phi)$ constructed in \reff{Uen}  and $H_n$ a standard mollifier of $H$. We claim that $(F_n, H_n, \Phi_n)_{n\ge 1}\in \cA_{(F, H, \Phi)}([0, t_1])$.  Since both $V$ and $\hat V$ are good solutions, then there exists $\d \in (0, t_1-t_0]$ such that, for any $n\ge 1$, the master equation on $[t_1-\d, t_1]$ corresponding to $(F_n, H_n, \Phi_n)$ has a classical solution $V_n$ and, for any $(t, x,\mu) \in [t_1-\d, t_1]\times \dbR^d\times \cP_2$, 
\beaa
\lim_{n\to\infty}  \Big[|(V_n -  V)(t, x, \mu)|+ |(V_n - \hat V)(t, x, \mu)| \Big]=0.
\eeaa
This implies that  
$\hat V = V$ on $[t_1-\d, t_1]\times \dbR^d\times \cP_2$,  contradicting  with the definition of $t_1$ in  \reff{t1}. Therefore, we must have  $\hat V = V$.

To see $(F_n, H_n, \Phi_n)_{n\ge 1}\in \cA_{(F, H, \Phi)}([0, t_1])$, by Theorem \ref{thm-mollifierx} (i)-(ii) and by the properties of standard mollifiers for $H$, we may easily verify all the properties except the uniform property of $\pa_x F_n, \pa_x \Phi_n$ required in Definition \ref{defn-cA} (ii). Without loss of generality we shall only verify it for $F_n$. For this purpose, we fix $M>0$ and consider $(x, \mu)\in Q_M\times \cP_2^M$. Recall \reff{Uen} we have
\beaa
\pa_x F_n(x,\mu) = \int_{\dbR^d} \pa_x F_{n}^{\mu}(x-{y\over n}, \mu) \z_0(dy).
\eeaa
It is clear that  $x-{y\over n} \in Q_{M+1}$ for all $y\in \supp(\z_0)$. Moreover, from the construction of $F_{n}^\mu$,  we can easily see that $\pa_x F_{n}^{\mu}  = (\pa_x F)_{n}^{\mu}$, the mollification of $\pa_x F$ with respect to $\mu$. Therefore, fix $x\in Q_{M+1}$ and denote $\check F(\mu) := \pa_x F(x, \mu)$, it suffices to verify the uniform continuity of $\check F_n$ on $\cP^M_2$, where $\check F_n$ is the smooth mollifier of $\check F$ constructed in \reff{Un}.  For any $\e>0$, since $\cP^M_2\subset\subset \cP_1$ and $\check F$ is continuous, there exists $\d>0$ such that $|\check F(\mu') - \check F(\mu)|\le \e$ for all $\mu\in \cP^M_2$ and $\mu'\in \cP_1$ satisfying $\cW_1(\mu', \mu)\le \d$. Again since $\cP^M_2\subset\subset \cP_1$, by \reff{munconv} there exists $n_0$ such that $\sup_{n> n_0} \sup_{y\in \D_n, \mu\in \cP^M_2}\cW_1(\mu_n(y), \mu) \le \d$ and hence $|\check F(\mu_n(y))- \check F(\mu)|\le {\e\over 3}$ for all $\mu\in \cP^M_2$, $y\in \D_n$, and $n>n_0$. Now for any $\mu, \nu\in \cP^M_2$ and  $n>n_0$, by \reff{Un} we have
\beaa
&&|\check F_n(\mu) - \check F_n(\nu)| \le \int_{\D_n} \z_n(y) |\check F(\mu_n(y)) - \check F(\nu_n(y))|dy \\
&&\le \int_{\D_n} \z_n(y) \big[|\check F(\mu_n(y)) -\check F(\mu)|+ |\check F(\nu_n(y)) - \check F(\nu)|\big]dy + |\check F(\mu) - \check F(\nu)|\\
&&\le {2\e\over 3} + |\check F(\mu) - \check F(\nu)| \le \e,\q\mbox{whenever}\q \cW_1(\mu, \nu)\le \d.
\eeaa
Note that $\check F_1, \cds, \check F_{n_0}$ are continuous under $\cW_1$ and hence uniformly continuous on $\cP^M_2$, by choosing a smaller $\d$ if necessary, we have $|\check F_n(\mu) - \check F_n(\nu)|\le \e$ for all $\mu, \nu\in \cP^M_2$ satisfying $\cW_1(\mu, \nu)\le \d$ and all $n\ge 1$. This is the desired uniform continuity of $\check F_n$.

(ii) The case $T\le \d_1$ is easier, and we will only focus on the global existence under the conditions of Theorem \ref{thm-well-posedness}. Let $V$ be defined by \reff{VY}. It is obvious that $V(T, \cd,\cd) = G$. By Theorem \ref{thm-well-posedness} we have $V\in C^0(\Th)$ and $V(t,\cd,\cd)\in C^0_{Lip}(\dbR^d\times \cP_2)$ for all $t\in [0, T]$. To verify the good solution property, we fix $0\le t_0<t_1\le T$ and a desired $\{(F_n, H_n, \Phi_n)\}_{n\ge 1} \in \cA_{F, H, V(t_1,\cd,\cd)}([t_0, t_1])$.  By Proposition \ref{prop-local}, there exists $\d \in (0, t_1-t_0]$, independent of $n$, such that the master equation on $[t_1-\d, t_1]$ corresponding to $(F_n, H_n, \Phi_n)$   has a classical solution $V_n$. Now fix $(t,x,\mu) \in [t_1-\d, t_1]\times \dbR^d\times \cP_2$ and $\xi\in \dbL^2(\cF_0\vee \cF^B_t, \mu)$.  Let $(X^{n}, Y^{n}, Z^{n}, Z^{0,n})$ be the solution to FBSDEs \reff{FBSDE1} on $[t, t_1]$ with data $(F_n, H_n)$, initial condition $\xi$, and terminal condition $\Phi_n$. Denote $\rho^n_s := \cL_{X^{n}_s|\cF^0_s}$ and $\D \Phi_n:= \Phi_n-\Phi$ and similarly for $F$ and $H$. Then by Theorem \ref{thm-stability} we have, for some $C$ and $R$ independent of $n$,
\beaa
&\dis \big|V_n(t,x,\mu) - V(t,x,\mu)\big| \le C\dbE\Big[\big[|\D I^H_n|+ |\D I^F_n| + |\D I^\Phi_n|\big]^{1\over 4} +|\D I^H_n|+ |\D I^F_n| + |\D I^\Phi_n|\Big],\ms\\
&\dis  \mbox{where}\q |\D I^H_n| := \sup_{(x,z)\in D_R} \Big[|\D H_n| + |\pa_x \D H_n|+|\pa_p \D H_n](x, z)\Big];\\
&\dis |\D I^F_n|^2:= \dbE_{\cF^0_t}\Big[  \int_t^{t_1} [ \|\D F_n(\cd, \rho^n_s)\|_\infty^2 + \|\pa_x \D F_n(\cd, \rho^n_s)\|_\infty^2] ds\Big];\\
&\dis |\D I^\Phi_n|^2:= \dbE_{\cF^0_t}\Big[ \|\D\Phi_n(\cd, \rho^n_{t_1})\|_\infty^2 + \|\pa_x \D\Phi_n(\cd, \rho^n_{t_1})\|_\infty^2\Big].
\eeaa
By \reff{Fnconv} we see that $\lim_{n\to\infty}| \D I^H_n|=0$.  Moreover, for any $M>\|\mu\|_2$,
\beaa
&\dis \dbE[|\D I^\Phi_n|] \le C\sup_{\nu \in \cP_2^M} \big[\|\D\Phi_n(\cd, \nu)\|_\infty + \|\pa_x \D\Phi_n(\cd, \nu)\|_\infty\big]  \\
&\dis+C\dbE\Big[  \Big( \dbE_{\cF^0_t}\big[ \|\D\Phi_n(\cd, \rho^n_{t_1}) - \D\Phi_n(\cd, \mu)\|_\infty^2 + \|\pa_x \D\Phi_n(\cd, \rho^n_{t_1})\|_\infty^2\big]\Big)^{1\over 2} \1_{\{\|\rho^n_{t_1}\|_2 >M\}}\Big].
\eeaa
Note that $\pa_x \D \Phi_n$ is uniformly bounded and $\D \Phi_n$ is uniformly Lipschitz continuous in $\mu$ under $\cW_1$,  Then, by \reff{Fnconv} again we have
\beaa
&&\dis\limsup_{n\to\infty} \dbE[|\D I^\Phi_n|] \le {C\over M}\limsup_{n\to\infty}\dbE\Big[\Big( \dbE_{\cF^0_t}[\cW^2_1(\rho^n_{t_1},\mu)] + 1\big]\Big)^{1\over 2} \|\rho^n_{t_1}\|_2 \Big]\\
&&\dis \le {C\over M}\limsup_{n\to\infty} \Big(\dbE\Big[\cW^2_2(\rho^n_{t_1},\mu)] + 1\Big]\Big)^{1\over 2} \Big(\dbE[\|\rho^n_{t_1}\|_2^2\big] \Big)^{1\over 2}\\
&&\dis \le{C\over M}\limsup_{n\to\infty} \Big(\dbE[|X^n_{t_1}-\xi|^2]  + 1\Big)^{1\over 2} \Big(\dbE[|X^n_{t_1}|^2\big] \Big)^{1\over 2}\le {C\over M}[1+\|\mu\|_2].
\eeaa
Since $M$ is arbitrary, we have $\lim_{n\to\infty} \dbE[|\D I^\Phi_n|]=0$. Similarly,  $\lim_{n\to\infty} \dbE[|\D I^F_n|]=0$. This proves \reff{Vnconv} and hence $V$ is a good solution. 
\qed

\begin{rem}
\label{rem-good}
{\rm We emphasize that the monotonicity condition \reff{mon} is used only for the existence of (global) good solutions, not for the uniqueness in the class of good solutions. The key condition in Theorem \ref{thm-good} is the uniform Lipschitz continuity of $V$, and the monotonicity condition is a sufficient condition to ensure the Lipschitz continuity of $V$ in $\mu$. In other words, alternative conditions which could provide a priori estimates for the uniform Lipschitz continuity of $V$ will also ensure the wellposedness of good solutions. Indeed, as we already saw, the local wellposedness does not require the monotonicity condition, see also the recent works \cite{BGY2,GM} for the global wellposedness of the master equation arising from a potential mean field game under an alternative displacement convex condition. It will be very interesting to have a systematic study on this Lipschitz continuity. We note that \cite{MWZZ} investigated this issue for standard (not mean field) FBSDEs.
\qed}
\end{rem}

\begin{rem}
\label{rem-gooduniform}
{\rm Under Assumptions \ref{assum-Lip1} and \ref{assum-H1}, one may choose the $\d$ in Definition \ref{defn-good} uniformly as the $\d_1$ (or $\d_1 \wedge [t_1-t_0]$, more precisely) in Proposition \ref{prop-local}, corresponding to a possibly larger $L_1$, larger than the Lipschitz constant of $V$. Indeed, by Proposition \ref{prop-local} and Theorem \ref{thm-apriori}, it follows from the arguments in Theorem \ref{thm-good} (i) that $V$ has to coincide with the value function defined by \reff{VY} for $t\in [T-\d_1, T]$. Similarly, for any $t_1$, by considering $V(t_1,\cd,\cd)$ as the terminal condition of the master equation on $[t_0, t_1]$, we can choose the same $\d_1$.
\qed}
\end{rem}

We conclude this section with the following stability result of good solutions.

\begin{thm}
\label{thm-stability2}
Assume $\{(F_n, G_n, H_n)\}_{n\ge 0}$ satisfy Assumptions \ref{assum-Lip1}, \ref{assum-Lip2}, \ref{assum-H1}, \ref{assum-H2}, and \ref{assum-mon}  uniformly, and let $V_n$ be the unique good solution to master equation \reff{master} with data $(F_n, G_n, H_n)$. If  $(F_n, G_n, H_n)$ converges to $(F, G, H)$ in the sense of \reff{Fnconv} (with $\Phi$ there replaced with $G$), then the master equation \reff{master} with data $(F, G, H)$ has a unique good solution $V$ and it holds that  $\lim_{n\to\infty} \|V_n - V\|_{L^\infty([0,T]\times\mathbb R^d\times \mathcal{P}_2^M)}=0$ for any $M>0$.
\end{thm}
\proof For any $M>0$ and any $(t,x,\mu)\in [0,T]\times\mathbb R^d\times \mathcal{P}_2^M$, following the proof of Theorem   \ref{thm-good} (ii) we can show that, by replacing $t_1, \Phi$ with $T,  G$, respectively,
\beaa
\big|V_n(t,x,\mu) - V(t,x,\mu)\big| \le C\dbE\Big[\big[|\D I^H_n|+ |\D I^F_n| + |\D I^G_n|\big]^{1\over 4} +|\D I^H_n|+ |\D I^F_n| + |\D I^G_n|\Big],
\eeaa
and, for $N\ge M$,
\beaa
\dbE[\D I^G_n]  &\le&  C\sup_{\nu \in \cP_2^{N}} \big[\|\D G_n(\cd, \nu)\|_\infty + \|\pa_x \D G_n(\cd, \nu)\|_\infty\big]  +{CM\over N}.
\eeaa
Similarly we have the estimate for $\dbE[|\D I^F_n|]$. Then
\beaa
&\dis \sup_{(t,x,\mu)\in [0,T]\times\mathbb R^d\times \mathcal{P}_2^M} \Big|V_n(t,x,\mu) - V(t,x,\mu)\Big| \le {CM^{1\over 4}\over N^{1\over 4}} + C\big[ |\D I^N_{n}|^{1\over 4} + |\D I^N_{n}|\big],\\
&\dis  \mbox{where}\q |\D I^N_{n}| := |\D I^H_n| +\sup_{\nu \in \cP_2^{N}}\Big[\|\D G_n(\cd, \nu)\|_\infty + \|\pa_x \D G_n(\cd, \nu)\|_\infty \q\\
&\dis + \|\D F_n(\cd, \nu)\|_\infty + \|\pa_x \D F_n(\cd, \nu)\|_\infty\Big].
\eeaa
 By \reff{Fnconv} we have $\dis\lim_{n\to\infty} |\D I^N_n| = 0$, implying the claimed convergence for any $M>0$.
\qed

\section{The weak solution  of master equations}
\label{sect-weak}
\setcounter{equation}{0}

In this section we propose another notion of solution, which we call weak solution, for the master equation \reff{master}. Roughly speaking, given a candidate solution $V$, we will use $V$ to decouple the FBSPDE \reff{SPDE} and consider the weak solutions to the two SPDEs separately, in the sense of \reff{SPDEweak} and \reff{BSPDEweak}, respectively. That is, we first consider the weak solution to the (forward) SPDE:
\bea
\label{FSPDE}
d\rho(t,x) = \big[\frac{\h\beta^2}{2}\tr\big( \pa_{xx} \rho(t,x)\big)\! - \! div(\rho(t,x) \pa_p H(x,\pa_x V(t,x, \rho_t)))\big]dt \!-\! \b\partial_x\rho(t,x)\cd d B_t^0;
\eea
and next, given $\rho$, consider the weak solution $(u, v)$ to BSPDE in \reff{SPDE}. Then we call $V$ a weak solution if $V(t, x, \rho_t)=u(t,x)$. 

We note that \reff{FSPDE} involves $\pa_x V$.  For this purpose, let $C^{0,1-}(\Th)$ denote the subset of $C^0_{Lip}(\Th)$ such that $\pa_x V$ exists and is continuous in $(x,\mu)$ for every $t\in [0, T]$. However, in light of Proposition \ref{prop-classical-x}, this is a very mild requirement.
\begin{defn}
\label{defn-weak}
 We say $V \in C^{0,1-}(\Th)$ with $V(T,\cd) = G$ is weak solution of master equation \reff{master} if, for any $0\le t_0< t_1 \le T$ and any initial condition $\rho_0$,  the SPDE \reff{FSPDE} on $[t_0, t_1]$  has a weak solution $\rho$; moreover, for any such weak solution $\rho$,  $u(t,x):= V(t,x,\rho_t)$ is a weak solution to the BSPDE in \reff{SPDE} on $[t_0, t_1]$. That is, there exists appropriate $v(t,x)$ such that $(u, v)$ satisfies \reff{BSPDEweak} on $[t_0, t_1]$.
\end{defn}

\begin{thm}
\label{thm-weak}
Let  Assumptions \ref{assum-Lip1} and  \ref{assum-H1} hold, and $V\in C^{0,1-}(\Th)$ with $V(T,\cd,\cd)= G$. Then $V$ is weak solution of master equation \reff{master} if and only if it is good solution.

Consequently,  if either $T\le \d_1$ for the $\d_1$ in Proposition \ref{prop-local}, or Assumptions \ref{assum-Lip2},  \ref{assum-H2},  and \ref{assum-mon} also hold true, then the function $V$ defined by \reff{VY} is also the unique weak solution of the master equation \reff{master}. 
\end{thm}

\proof By Theorem \ref{thm-good}, clearly it suffices to prove the equivalence of the two notions under Assumptions \ref{assum-Lip1} and  \ref{assum-H1}. We proceed in two steps.

{\it Step 1.} In this step we prove the result for $T\le \d_1$, for the $\d_1$ in Proposition \ref{prop-local}. 

{\it Step 1.1.} In this case the function $V$ constructed by \reff{VY} is the unique good solution. We show that it is also a weak solution. Without loss of generality we shall verify Definition \ref{defn-weak} only for $[t_0, t_1]=[0, T]$. First for any $\xi \in \dbL^2(\cF_0)$, by Proposition \ref{prop-local} (i) the FBSDE \reff{FBSDE1} is wellposed with $Z^\xi_t = \pa_x V(t, X^\xi_t, \cL_{X^\xi_t|\cF^0_t})$. Then we see that $\rho_t := \cL_{X^\xi_t|\cF^0_t}$ is a weak solution to  SPDE \reff{FSPDE}. 

 Let $(F_n, G_n, H_n)_{n\ge 1}\in \cA_{(F, G, H)}([0, T])$ be as in the proof of Theorem \ref{thm-good} (i), and $V_n$ be the classical solution to the master equation \reff{master} with data $(F_n, G_n, H_n)$, which exists due to Proposition \ref{prop-local} (ii). By Theorem \ref{thm-apriori}, see also Remark \ref{rem-gooduniform}, we have $\lim_{n\to\infty} [V_n-V](t,x,\mu)=0$ for any $(t,x,\mu)\in\Th$. Moreover, note that $\pa_x V(t_0, x, \mu) = \td Y^x_{t_0}$ for the $\td Y^x$ in \reff{FBSDE-pax}. By the uniform regularity of $(\pa_x F_n, \pa_x G_n, \pa_x H_n, \pa_p H_n)$ required in Definition \ref{defn-cA}, one can easily show that $\pa_x V_n$ are locally uniformly continuous, in the sense of Definition \ref{defn-cA} (ii) with the $\k_M$ independent of $n$. Then one can easily see that $\lim_{n\to\infty} [\pa_x V_n-\pa_x V](t,x,\mu)=0$ for any $(t,x,\mu)\in \Th$.

Now let $\rho$ be an arbitrary weak solution to SPDE \reff{FSPDE} on $[0, T]$. Then $\rho_t = \cL_{X_t |\cF^0_t}$ where, possibly in an enlarged probability space,
\bea
\label{weak-X}
 X_t = \xi + \int_0^t \pa_p H( X_s, \pa_x V(s,  X_s, \rho_s)) ds + B_t + \b B^0_t.
\eea
Let $\tilde  X, \bar  X$ be conditionally independent copies of $ X$, conditional on $\dbF^0$. Denote 
\bea
\label{uunvn}
u(t,x) := V(t, x, \rho_t), \q u_n(t,x) := V_n(t, x, \rho_t), \q v_n(t,x):=  \b\tilde \dbE_{\cF^0_t}\big[\pa_\mu V_n(t, x, \rho_t, \tilde X_t)\big].
\eea
Apply It\^o formula \reff{Ito}, we have: 
\beaa
&d u_n(t,x) =  v_n(t,x)\cd dB^0_t+ \pa_t V_n(t,x,\rho_t) dt +  \tr\Big(\tilde  \dbE_{\cF^0_t}\Big[\frac{\h\b^2}{2} \pa_{\tilde x} \pa_\mu V_n(t, x, \rho_t, \tilde  X_t) \\
&+ \pa_\mu V_n(t, x, \rho_t, \tilde  X_t)(\pa_pH)^\top(\tilde  X_t, \pa_x u(t, \tilde  X_t))  +\frac{\b^2}{2}\bar{\mathbb E}_{\cF^0_t}\big[\pa_{\mu\mu}V_n(t, x,\rho_t,\bar X_t,\tilde X_t)\big]\Big]\Big) dt. 
\eeaa
Since $V_n$ satisfies the master equation \reff{master} with data $(F_n, G_n, H_n)$, we have 
\bea
\label{dun}
&&\dis d u_n(t,x) =  -\Big[\tr\big(\frac{\h\b^2}{2} \pa_{xx} u_n(t,x) + \b \pa_x v_n(t,x)\big)+ H_n(x,\partial_x u_n(t, x)) + F_n(x, \rho_t) \Big]  dt\nonumber\\
&&\dis \qq\qq \qq\qq+ I_n(t,x) dt + v_n(t,x)\cd dB^0_t,\q\mbox{where}\\
&&\dis I_n(t,x):=  \tilde  \dbE_{\cF^0_t}\Big[ \pa_\mu V_n(t, x, \rho_t, \tilde  X_t)\cd \big[\pa_pH(\tilde  X_t, \pa_x u(t, \tilde  X_t)) -\pa_pH_n(\tilde  X_t, \pa_x u_n(t, \tilde  X_t))\big]  \Big]. \nonumber
\eea
Now for any $\f\in C^\infty_c([0, T]\times \dbR^d)$, denote
\beaa
Y^n_t:= \int_{\dbR^d} u_n \f(t,x) dx,\q Z^n_t:= \int_{\dbR^d} v_n\f(t,x) dx,\\
 \Phi^n_t :=  \int_{\dbR^d} \Big[ u_n \pa_t \f(t,x) + [\frac{\h\b^2}{2} \pa_x u_n + \b v_n] \cd\pa_x \f(t,x) \\
 -\big[H_n(x,\partial_x u_n(t, x)) + F_n(x, \rho_t)\big] \f(t,x)\Big] dx.
\eeaa
Then we have
\bea
\label{BSDEn}
 Y^n_t  = \int_{\dbR^d} G_n(x, \rho_T)\f(T,x) dx -\int_t^T \Phi^n_s ds + \int_t^T \int_{\dbR^d} I_n\f(s,x) dx ds   -\int_t^T Z^n_s\cd dB^0_s.
\eea
It is clear that, as $n\to \infty$, 
\bea
\label{Ynconv}
Y^n_t \to Y_t := u(t,x) := V(t,x, \rho_t).
\eea
Note that $V_n$ is uniformly Lipschitz continuous in $(x,\mu)$, then $\pa_x u_n$ and $v_n$ are uniformly bounded, which implies $\sup_{0\le t\le T} \dbE[|\F^n_t|^2] \le C$, for any $n$.  Then by standard BSDE estimates, it follows from \reff{BSDEn} and \reff{Ynconv} that 
\beaa
\lim_{n,m\to \infty} \dbE\Big[\int_0^T \Big|\int_{\dbR^d} [v_n - v_m]\f(t,x) dx\Big|^2 dt \Big] = \lim_{n, m\to \infty} \dbE\Big[\int_0^T |Z^n_t - Z^m_t|^2 dt \Big] = 0.
\eeaa
Since $\f$ is arbitrary, $v_n$ has a weak limit $v$ such that
\bea
\label{v}
\lim_{n,m\to \infty} \dbE\Big[\int_0^T \Big|\int_{\dbR^d} [v_n - v]\f(t,x) dx\Big|^2 dt \Big] = 0,\q \forall \f\in C^\infty_c([0, T]\times \dbR^d).
\eea
In particular, \reff{v} holds for $\pa_x \f$ as well. Then we can easily see that
\beaa
\lim_{n\to \infty} \Phi^n_t = \Phi_t :=  \int_{\dbR^d} \Big[ u \pa_t \f(t,x) + [\frac{\h\b^2}{2} \pa_x u + \b v]\cd \pa_x \f(t,x) \\
 -\big[H(x,\partial_x u(t, x)) + F(x, \rho_t)\big] \f(t,x)\Big] dx.
\eeaa
Moreover, by the boundedness of $\pa_\mu V_n$ again we have $\dis\lim_{n\to\infty} I^n(t,x) = 0$. Thus \reff{BSDEn} implies 
\beaa
\int_{\dbR^d} u\f(t,x) dx  =  \int_{\dbR^d} G(x, \rho_T)\f(T,x) dx -\int_t^T \Phi_s ds  -\int_t^T \int_{\dbR^d} v\f(s, x) dx\cd dB^0_s.
\eeaa
This is exactly \reff{BSPDEweak}, namely $u$ is a weak solution to the BSPDE in \reff{SPDE}, and hence $V$ is a weak solution to the master equation \reff{master}.

{\it Step 1.2.} We now assume $V$ is an arbitrary weak solution.  Let $X, \rho$ be as in \reff{weak-X}, and  $u(t,x):= V(t,x,\rho_t)$. Then there exists $v$ such that \reff{BSPDEweak} holds for all $\f\in C^\infty_c([0,T]\times \dbR^d)$. Let $\psi\in C^\infty_c(\dbR^d)$ be a density function, namely $\psi \ge 0$ and $\int_{\dbR^d} \psi(x) dx = 1$. Denote
\beaa
&\dis u_n(t,x) :=\int_{\dbR^d} u(t, y) n\psi\big(n(x-y)\big) dy,\q v_n(t,x) := \int_{\dbR^d} v(t, y) n\psi\big(n(x-y)\big) dy,\\
&\dis \Phi_n(t,x):= \int_{\mathbb R^d}\Big[  [\frac{\h\beta^2}{2} \pa_x u(t, y) + \b v(t, y)]\cd  n^2\pa_{x}\psi(n(x-y))  \\
&\dis + [H(y,\pa_x u(t,y)) + F(y, \rho_t)]  n \psi(n(x-y)) \Big]dy.
\eeaa
Note that, for any fixed $x$, $\f(t,y) := n \psi(n(x-y))$ is a desired test function. Then, by considering $y$ as the variable, \reff{BSPDEweak} implies  
\beaa
d u_n(t,x) =  v_n (t,x)\cd dB_t^0 - \Phi_n(t,x) dt.
\eeaa
Note that $u_n, \Phi_n, v_n$ are all smooth in $x$.  Denote
\beaa
Y^n_t := u_n(t, X_t),\q Z^n_t := \pa_x u_n(t, X_t),\q Z^{0,n}_t := v_n(t, X_t) + \b \pa_x u_n(t, X_t).
\eeaa
Apply the It\^o-Wentzell formula, we have, omitting the variables $(t, X_t)$ inside the functions, 
\beaa
&&\dis\!\!\!\!\!\!  d Y^n_t =  v_n \cd dB^0_t - \Phi_n dt+ \pa_x u_n\cd d X_t +\tr \big[{\h\b^2\over 2} \pa_{xx} u_n  + \b \pa_x v^\top_n\big] dt\\
&&\dis \!\!\!\!\!\! = \big[ -\Phi_n + \pa_x u_n\cd\pa_p H( X_t, \pa_x u) +\tr({\h\b^2\over 2} \pa_{xx} u_n  + \b \pa_x v^\top_n)\big]dt + Z^n_t \cd dB_t + Z^{0,n}_t\cd dB^0_t\\
&&\dis\!\!\!\!\!\!  =-\tilde \Phi_n(t, X_t)dt + Z^n_t\cd dB_t + Z^{0,n}_t\cd dB^0_t,\q  \mbox{where}\\
&&\dis \!\!\!\!\!\! \tilde \Phi_n(t,x):=  \int_{\mathbb R^d}[H(y,\pa_x u(t,y)) + F(y, \rho_t)]  n \psi(n(x-y))dy -  \pa_x u_n(t,x)\cd\pa_p H(x, \pa_x u(t,x)).
\eeaa
Now denote
\beaa
Y_t := u(t, X_t),\q Z_t:= \pa_x u(t, X_t).
\eeaa
One can easily see that $u_n\to u$, $\pa_x u_n \to \pa_x u$, and
\beaa
\tilde \Phi_n(t, x) \to   F(x, \rho_t) + H(x, \pa_x u(t,x)) -\pa_x u(t,x)\cd\pa_p H(x, \pa_x u(t,x)).
\eeaa
Then by standard BSDE arguments we can see that $Z^{0, n}$ converges to some $Z^0$ and 
\bea
\label{weak-YZ}
\left.\ba{c}
\dis Y_t = G(X_T, \rho_T) - \int_t^T Z_s\cd dB_s - \int_t^T Z^0_s\cd dB^0_s\\
\dis + \int_t^T \Big[F(X_s, \rho_t) + H(X_s, Z_s) -Z_s\cd\pa_p H(X_s,Z_s)\Big] ds.
\ea\right.
\eea
That is, $(X, Y, Z, Z^0)$ satisfies the FBSDE \reff{weak-X}-\reff{weak-YZ}. Since $T\le \d_1$, by the uniqueness of the solution to the FBSDE we know $Y_0 = u(0, \xi) = V(0,\xi, \rho_0)$ is unique. This  proves the uniqueness of $V(0,\cd,\cd)$. Similarly $V(t,\cd,\cd)$ is also unique, for any $t\in [0, T]$.

\ms
{\it Step 2.} We now consider arbitrary $T$. 

{\it Step 2.1.} Assume $V\in C^{0,1-}(\Th)$ is the (unique) good solution.  Again without loss of generality we shall verify Definition \ref{defn-weak} only on $[0, T]$.  Let $\d_1$ be as in Step 1, but with $L_1$ larger than the Lipschitz constant of $V$ with respect to $(x,\mu)$. 

We first show that, for any initial condition $\rho_0$,  SPDE \reff{FSPDE} has a weak solution on $[0, T]$. Indeed, let $0=t_0<\cds<t_n=T$ be a partition such that $t_i-t_{i-1}\le \d_1$ for all $i$. First by Step 1.1 the SPDE \reff{FSPDE} has a weak solution on $[t_0, t_1]$ with initial condition $\rho_0$. Next, consider the problem on $[t_1, t_2]$, by Step 1.1 again we can see that SPDE \reff{FSPDE} has a weak solution on $[t_1, t_2]$ with initial condition $\rho_{t_1}$. Repeat the arguments forwardly in time we may construct a weak solution to SPDE \reff{FSPDE} on $[0, T]$ with initial condition $\rho_0$.

Now let $X$ and $\rho$ be an arbitrary weak solution to \reff{weak-X} on $[0, T]$ and denote $u(t,x) := V(t,x,\rho_t)$.  By Step 1.1, $u$ is a weak solution to \reff{BSPDEweak} on $[t_{n-1}, t_n]$. Next, consider the problem on $[t_{n-2}, t_{n-1}]$ with terminal condition $V(t_{n-1},\cd,\cd)$, by Step 1.1 again $u$ is a weak solution to \reff{BSPDEweak} on $[t_{n-2}, t_{n-1}]$. Repeat the arguments backwardly in time we see that $u$ is a weak solution to \reff{BSPDEweak} on $[0, T]$.

{\it Step 2.2.} Let $V\in C^{0,1-}(\Th)$ be a weak solution. Again let $\d_1$ and the partition $0=t_0<\cds<t_n=T$ be as in Step 2.1. By Step 1.2, $V$ is the good solution on $[t_{n-1}, t_n]$. Repeat the arguments backwardly in time we see that $V$ is the good solution on $[0, T]$.
\qed

\begin{rem}
\label{rem}
{\rm (i) Under Assumptions \ref{assum-Lip1} and  \ref{assum-H1}, if a weak solution exists, by the constructions in the proof of Theorem \ref{thm-weak} we  see that  the McKean-Vlasov SDE \reff{weak-X} actually has a strong solution.  However, for an arbitrary  $V\in C^{0,1-}(\Th)$, it is not clear that the SPDE \reff{FSPDE} (or equivalently the McKean-Vlasov SDE \reff{weak-X}) has a weak solution. 

(ii) When there is no common noise,  $\rho_t$ is deterministic. In this case we can show that SPDE \reff{FSPDE} has a weak solution for any $V\in C^{0,1-}(\Th)$, see Proposition \ref{prop-weak} below. 
\qed}
\end{rem}

\section{The weak-viscosity solution  of master equations}
\label{sect-viscosity}
\setcounter{equation}{0}

It is well understood that one cannot expect comparison principle even for classical solutions of master equation \reff{master}, see Example \ref{eg-comparison} below, thus the notion of viscosity solution for master equation has been considered infeasible. However, the backward SPDE in \reff{SPDE} is parabolic and thus is legitimate to investigate its viscosity solutions. In light of this, we modify Definition \ref{defn-weak} and introduce the following weak-viscosity solution to the master equation.

Our idea is to require the function $u$ in Definition \ref{defn-weak} to be a viscosity solution, instead of a weak solution, to the BSPDE in \reff{SPDE}. Thus we shall first specify the notion of viscosity solution to BSPDEs. When there is no common noise, namely $\b=0$, the $\rho_t$ and $u(t,x)$ become deterministic, $v(t,x) = 0$,  and the BSPDE in \reff{SPDE} becomes a standard parabolic PDE. Then the viscosity solution is in the standard sense as in  \cite{CIL} (actually it will be a classical solution under our conditions). In the general case, however, the state space for the variables $(t,x, \o)$ is infinitely dimensional with adaptedness requirement in $\o$, the standard approach of \cite{CIL} does not work. Note that a BSPDE can be viewed as a Path-dependent PDE (PPDE for short), so we shall apply the viscosity solution approach for PPDEs proposed by  \cite{EKTZ} and the subsequent works, see  \cite{Zhang} and the references therein. This approach, however, requires certain regularity in $\o$, which corresponds to the paths of $B^0$. For this purpose, denote $\O^0:= C([0, T]; \dbR^d)$ and $B^0$ the canonical space, namely $B^0(\o) = \o$ for $\o\in \O^0$. The state space $[0, T]\times \dbR^d \times \O^0$ is equipped with the metric:
\bea
\label{d}
{\bf d}((t_1, x_1, \o^1), (t_2, x_2, \o^2)):= |t_1-t_2| + |x_1-x_2| + \sup_{0\le s\le T} |\o^1_{t_1\wedge s} - \o^2_{t_2\wedge s}|.
\eea
Since $\rho$ is $\dbF^{B^0}$-progressively measurable, we may view it as a function $\rho: [0, T]\times \O^0 \to \cP_2$. Moreover, let $C^{0,2-}(\Th)$ denote the subset of $V\in C^{0,1-}(\Th)$ such that $\pa_x V$ also belongs to $C^{0}_{Lip}(\Th)$. We remark that here we are requiring stronger regularities than good solutions and weak solutions in order to have the pathwise regularity in $\o\in \O^0$. When there is no common noise, we may define weak-viscosity solution also in the space $C^{0, 1-}(\Th)$.
\begin{lem}
\label{lem-rhocont}
Let Assumption \ref{assum-H1} hold and $V\in C^{0, 2-}(\Th)$. For any initial condition $\rho_0\in \cP_2$,  the SPDE \reff{FSPDE} on $[0, T]$ has a unique weak solution $\rho$ which is uniformly Lipschitz continuous in $\o\in \O^0$: 
\bea
\label{rhocont}
\cW_1(\rho_t(\o^1), \rho_t(\o^2)) \le \cW_2(\rho_t(\o^1), \rho_t(\o^2)) \le C\|\o^1-\o^2\|.
\eea
\end{lem}
\proof Let $\xi \in \dbL^2(\cF_0, \rho_0)$, and consider the following McKean-Vlasov SDE:
\bea
\label{SDEo}
X^\o_t = \xi + \int_{0}^t \pa_p H\big(X^\o_s + \b \o_s, \pa_x V(s, X^\o_s+ \b \o_s, \cL_{X^\o_s + \b \o_s})\big)ds + B_t,
\eea
where $X^\o$ corresponds to $X-\b \o$. By the Lipschitz continuity of $\pa_pH$ and $\pa_x V$, the SDEs \reff{weak-X} and \reff{SDEo} have unique strong solution $X$ and $X^\o$, and 
\beaa
 \dbE[|X^{\o^1}_t-X^{\o^2}_t|^2] \le C \int_{0}^t |\o^1_s - \o^2_s|^2ds,\q\forall \o^1, \o^2\in \O^0.
 \eeaa
Denote  $\rho(t,x):= \cL_{X_t|\cF^0_t}$. It is clear that $\rho$ is the unique weak solution to SPDE \reff{FSPDE} and $\rho(t,x,\o) = \cL_{X^\o_t + \b \o_t}$, for $\dbP_0$-a.e. $\o\in \O^0$. Fix this version for $\rho$. Then, for any $\o^1, \o^2\in \O^0$,
\beaa
&&\dis \cW_2^2(\rho_t(\o^1), \rho_t(\o^2)) \le \dbE\Big[|[X^{\o^1}_t+\b\o^1_t]-[X^{\o^2}_t+\b \o^2_t]|^2\Big] \\
&&\dis \le C \dbE\Big[|X^{\o^1}_t-X^{\o^2}_t|^2\Big]  + C|\o^1_t-\o^2_t|^2 \le C \int_{0}^t |\o^1_s - \o^2_s|^2ds+ C|\o^1_t-\o^2_t|^2.
\eeaa
This implies \reff{rhocont}  immediately.
\qed

We now write down the PPDE corresponding to the BSPDE in \reff{SPDE}: 
\bea
\label{PPDE}
\left.\ba{c}
\dis \pa_t u(t,x,\o) + \tr\big(\frac{\h\beta^2}{2} \pa_{xx} u +\b\partial_{x\o} u + {1\over 2} \pa_{\o\o} u\big) + H(x,\pa_x u) + F(x, \rho_t(\o))=0,\\
\dis u(T,x,\o)= G(x, \rho_T(\o)).
\ea\right.
\eea
Here the variable $(t,x,\o) \in [0, T]\times \dbR^d \times \O^0$, and $\pa_\o u$ is the path derivative introduced by \cite{Dupire}. Since technically we are not going to use it, we refer to \cite{Zhang} for details. By Lemma \ref{lem-rhocont} we see that the data $F(x, \rho(\o))$ and $G(x, \rho(\o))$ in \reff{PPDE} are continuous in $\o$. In this section we shall adopt the definition in  \cite{RTZ}, which is easier to present. We emphasize that we may replace this definition with any appropriate notion of viscosity solutions for BSPDEs, in particular, the pseudo-Markoivian viscosity solution proposed by  \cite{EZ} for fully nonlinear BSPDEs also works for our purpose.  

Let $L>0$, $(t,x,\o) \in  [0, T]\times \dbR^d\times\O^0$, and $\e>0$.  Denote
\beaa
& \ch^t_\e := \inf\{s> t: s-t + |B^t_s| + |B^{0,t}_s| \ge \e\} \wedge T,\\
&  (\o\otimes_t B^{0,t})_s:= \o_s\1_{[0, t]}(s) + [\o_t + B^{0,t}_s]\1_{(t, T]}(s).
\eeaa
Let $\cT^t_\e$ be the set of $\dbF^{B^t, B^{0,t}}$-stopping times $\t$ on $[t, T]$ such that $\t\le \ch^t_\e$, and $\cA^t_L$ the set of $\dbF^{B^t, B^{0,t}}$-progressively measurable $\dbR^d$-valued processes $\l$ on $[t, T]$ such that $|\l|\le L$.  Now for any $u\in C^0([0, T]\times \dbR^d\times\O^0; \dbR)$, introduce the semi-jets for viscosity solutions:
\bea
\label{jets}
\left.\ba{lll}
\! \underline \cJ_L u(t,x,\o) := \Big\{(a, z, z_0) \in \dbR \times \dbR^d\times \dbR^d: \exists \e>0~\mbox{s.t., for}~ \forall \t\in \cT^t_\e, \l\in \cA^t_L\\
 u(t,x,\o) \ge \dbE\Big[ M^\l_\t \big[u(\t, x+ B^t_\t + \b B^{0,t}_\t, \o\otimes_t B^{0,t}) - a(\t-t) - z\cd B^t_\t - z_0\cd B^{0, t}_\t\big]\Big]\Big\};\\
\! \overline \cJ_L u(t,x,\o) := \Big\{(a, z, z_0) \in \dbR \times \dbR^d\times \dbR^d: \exists \e>0~\mbox{s.t., for}~ \forall \t\in \cT^t_\e, \l\in \cA^t_L\\
 u(t,x,\o) \le \dbE\Big[ M^\l_\t \big[u(\t, x+ B^t_\t+ \b B^{0,t}_\t, \o\otimes_t B^{0,t}) - a(\t-t) - z\cd B^t_\t - z_0\cd B^{0, t}_\t\big]\Big]\Big\}.
\ea\right.
\eea

\begin{defn}
\label{defn-PPDEviscosity} Let $u\in C^0([0, T]\times \dbR^d\times\O^0)$.

(i) We say $u$ is an $L$-viscosity subsolution (resp. supersolution) of BSPDE in \reff{SPDE} (or PPDE \reff{PPDE}) at $(t,x,\o)$ if: for all $(a,z, z_0) \in \underline \cJ_L u(t,x,\o)$ (resp. $\overline \cJ_L u(t,x,\o)$),
\bea
\label{viscosity}
a + H(x, z) + F(x, \rho_t(\o))  \ge ~(\mbox{resp.} ~\le )~ 0.
\eea

(ii) We say $u$ is a viscosity solution  of BSPDE in \reff{SPDE} (or PPDE \reff{PPDE}) if, for some $L>0$,  $u$ is both an $L$-viscosity subsolution and an $L$-viscosity supersolution at all $(t,x,\o)$. 
\end{defn}

\begin{defn}
\label{defn-viscosity}
 We say $V \in C^{0,2-}(\Th)$ with $V(T,\cd) = G$ is a weak-viscosity solution of master equation \reff{master} if, for any initial condition $\rho_0\in \cP_2(\dbR^d)$ and the corresponding solution $\rho$ to the SPDE \reff{FSPDE} on $[0, T]$ as in Lemma \ref{lem-rhocont},   $u(t,x,\o):= V(t,x,\rho_t(\o))$ is a viscosity solution to the BSPDE in \reff{SPDE} in the sense of Definition \ref{defn-PPDEviscosity}. 
\end{defn}

\begin{rem}
\label{rem-viscosity}
{\rm For fixed $\rho$, by \cite{RTZ} we can establish the comparison principle for the viscosity solution $u$ to PPDE \reff{PPDE}. However,  we emphasize again that this does not imply the comparison principle for the weak-viscosity solution $V$ to the master equation \reff{master}. In fact, as we see in Example \ref{eg-comparison} below, even classical solutions of master equations may not satisfy the comparison principle. Nevertheless, we will have the desired wellposedness of weak-viscosity solutions, including existence, uniqueness and stability. 
\qed}
\end{rem}

\begin{thm}
\label{thm-viscosity}
Let $F,G,H$ satisfy Assumptions \ref{assum-Lip1} and  \ref{assum-H1}. Then $V\in C^{0,2-}(\Th)$ with $V(T,\cd,\cd)= G$ is a weak-viscosity solution of master equation \reff{master} if and only if it is a good solution (and hence a weak solution).

Consequently, if  either $T\le \d_1$ for the $\d_1$ in Proposition \ref{prop-local} and Assumption \ref{assum-Lip2} holds, or Assumptions \ref{assum-Lip2},  \ref{assum-H2},  and \ref{assum-mon} also hold true, then the function $V$ defined by \reff{VY}  is the unique {weak-viscosity} solution of the master equation \reff{master}. 
\end{thm}
\proof Under Assumptions \ref{assum-Lip1},  \ref{assum-Lip2}, and  \ref{assum-H1}, by Proposition \ref{prop-classical-x} (iii) we have the desired regularity of $V$ with respect to $x$. Then either by Theorem \ref{thm-good} or Theorem \ref{thm-weak}, it suffices to prove the equivalence of the three notions of solutions under Assumptions \ref{assum-Lip1} and  \ref{assum-H1}, provided $V\in C^{0,2-}(\Th)$.
Let $\d_1>0$ be the constant in Proposition \ref{prop-local}, we shall only prove the equivalence on $[0, T]$ by assuming further $T\le \d_1$.  The general case follows the same arguments as in Theorem \ref{thm-weak} Step 2. 

(i) First assume $V$ is a good solution. Let $(X, \rho)$ be as in  \reff{weak-X}, and $(F_n, G_n, H_n, V_n)$ are appropriate approximations.  Introduce $u_n$ as in \reff{uunvn}, then $u_n$ satisfies \reff{dun} in strong sense. By Lemma \ref{lem-rhocont}, one can easily show that $I_n$ is also uniformly continuous in $\o$. We claim that $u_n$ is a viscosity solution to the corresponding PPDE:
\bea
\label{PPDEn}
\left.\ba{c} \pa_t u_n(t,x,\o) + \tr\big(\frac{\h\beta^2}{2} \pa_{xx} u_n +\b\partial_{x\o} u_n + {1\over 2} \pa_{\o\o} u_n\big) + H_n(x,\pa_x u_n) + F_n(x, \rho_t(\o))\\
 + I_n(t, x, \o)=0;\qq
 u_n(T,x,\o) = G_n(x,\rho_T(\o)).
 \ea\right.
 \eea
 Indeed,  fix $(t,x,\o)$ and denote $X^{t,x}_s:=x+ B^t_s+ \b B^{0,t}_s$.  By It\^o-Wentzell formula we have
\beaa
&&\dis d u_n  (s, X^{t,x}_s, \o\otimes_t B^{0,t}) = \pa_x u_n\cd [dB_s+\b dB^0_s] + v_n\cd dB^0_s   \\
&&\dis \qq -\Big[\tr\big(\frac{\h\b^2}{2} \pa_{xx} u_n + \b \pa_x v_n^\top\big)+ H_n+ F_n + I_n \Big]  ds + \tr\big( {\h\b^2\over 2} \pa_{xx} u_n + \b\pa_x v_n^\top\big)ds\\
&&= \pa_x u_n\cd dB_s + [v_n + \b \pa_x u_x]\cd dB^0_s    - [H_n+ F_n + I_n] ds.
\eeaa
Now set $L:= L^H_1(C_1)$ for the $C_1$ in \reff{pauest1}. For any $(a, z, z_0) \in \underline \cJ_L u(t,x,\o)$ with corresponding $\e>0$, and any $\t\in \cT^t_\e$, $\l\in \cA^t_L$, we have
 \beaa
 0  &\ge& \dbE\Big[M^\l_\t \big[u_n(\t, X^{t,x}_s, \o\otimes_t B^{0,t})  - u_n(t,x,\o) - a(\t-t) - z\cd B^t_\t - z_0\cd B^{0, t}_\t\big]\Big]\\
 &=&-\dbE\Big[\int_t^\t M^\l_s \big[ H_n + F_n + I_n   + a + \l_s\cd [z-\pa_x u_n]\big](s, X^{t,x}_s, \o\otimes_t B^{0,t})ds \Big]
 \eeaa
 Choose $\l_s$ so that 
 \beaa
 H_n(X^{t,x}_s, \pa_x u_n(s, X^{t,x}_s, \o\otimes_t B^{0,t})) - H_n(X^{t,x}_s, z) = \l_s\cd \big[\pa_x u_n(s, X^{t,x}_s, \o\otimes_t B^{0,t})-z\big].
 \eeaa
 Then, for any $\t\in \cT^t_\e$,
  \beaa
 \dbE\Big[\int_t^\t M^\l_s \big[ H_n(X^{t,x}_s, z) + F_n + I_n   + a \big](s, X^{t,x}_s, \o\otimes_t B^{0,t})ds \Big]\ge 0.
 \eeaa
 Now by standard arguments we have $a + H_n(x, z) + F_n(x, \rho_t(\o)) + I_n(t,x,\o) \ge 0$. That is, $u_n$ is an $L$-viscosity subsolution of PPDE  \reff{PPDEn}. Similarly we can show that $u_n$ is an $L$-viscosity supersolution, hence a viscosity solution of PPDE  \reff{PPDEn}.

Now send $n\to 0$, noting in particular that $\lim_{n\to\infty} I_n(t,x)=0$ from the proof of Theorem \ref{thm-weak}, by stability of viscosity solutions we see that $u = \lim_{n\to\infty} u_n$ is a viscosity solution to  PPDE \reff{PPDE}. Thus $V$ is a weak-viscosity solution of the master equation \reff{master}.

(ii) On the other hand, let $V$ be an arbitrary weak-viscosity solution to the master equation \reff{master}, and $(X, \rho)$ be as in \reff{weak-X}. Then $u(t,x,\o):= V(t,x,\rho_t(\o))$ is a viscosity solution to  PPDE \reff{PPDE}. Given $\rho$, by \cite{RTZ} the viscosity solution to \reff{PPDE} is unique and we must have  $Y_t = u(t, X_t, B^0)= V(t, X_t, \rho_t)$, where $Y$ solves the following BSDE:
 \bea
\label{viscosity-YZ}
Y_t = G(X_T, \rho_T) + \int_t^T [F(X_s, \rho_s)  - \h L(X_s, Z_s)]ds - \int_t^T Z_s\cd d B_s - \int_t^T Z^{0}_s\cd dB^0_s. 
\eea
Now fix $t$ and let $\d>0$, we have
\bea
\label{EtZ}
\left.\ba{lll}
\dis \dbE_{\cF_t}\Big[\int_t^{t+\d} Z_s ds\Big] = \dbE_{\cF_t}\Big[B^t_{t+\d} \int_t^{t+\d} Z_s\cd dB_s \Big] \\
\dis =  \dbE_{\cF_t}\Big[B^t_{t+\d} \big[Y_{t+\d} - Y_t + \int_t^{t+\d} [F-\h L] ds - \int_t^{t+\d} Z^0_s \cd dB^0_s\big]  \Big]\\
\dis =  \dbE_{\cF_t}\Big[B^t_{t+\d} \big[V(t+\d, X_{t+\d}, \rho_{t+\d})- V(t, X_t, \rho_t) + \int_t^{t+\d} [F-\h L] ds\big]  \Big]\\
\dis = \dbE_{\cF_t}\Big[B^t_{t+\d}\big( \pa_x V(t+\d, X_{t}, \rho_{t+\d})\cd[B^t_{t+\d} + \b B^{0, t}_{t+\d}] \big)\Big]+  I^t_{t+\d}  \\
\dis= \d\dbE_{\cF_t}\Big[ \pa_x V(t+\d, X_{t}, \rho_{t+\d})  \Big] +  I^t_{t+\d}
\ea\right.
\eea
where, noting that $\dbE_{\cF_t}[B^t_{t+\d} \eta] =0$ for any $\cF_0\vee\cF^B_t\vee \cF^{B^0}_{t+\d}$-measurable random variable $\eta$, 
\beaa
&&\dis I^t_{t+\d} := \dbE_{\cF_t}\Big[B^t_{t+\d}\big[\int_t^{t+\d} [F-\h L] ds\\
&&\dis  +V(t+\d, X_{t+\d}, \rho_{t+\d}) - V(t+\d, X_{t}, \rho_{t+\d}) - \pa_x V(t+\d, X_{t}, \rho_{t+\d})\cd[B^t_{t+\d} + \b B^{0, t}_{t+\d}]\big] \Big].
\eeaa
Since $\pa_x V$ is bounded and uniformly Lipschitz continuous, by Taylor expansion we have
\beaa
&&\dis |I^t_{t+\d}|\le  C\dbE_{\cF_t}\Big[|B^t_{t+\d}|\big[ \int_t^{t+\d}\!\!\! |F-\h L|ds+|X_{t+\d} - X_t - B^t_{t+\d} - \b B^{0, t}_{t+\d}| + |X_{t+\d}-X_t|^2 \big]\Big]\\
&&\dis \le C\dbE_{\cF_t}\Big[|B^t_{t+\d}|\big[ \int_t^{t+\d}[|\pa_p H| +|F|+|\h L|]ds + |\int_t^{t+\d} \pa_p H ds + B^t_{t+\d} + \b B^{0, t}_{t+\d}|^2\Big]\\
&&\dis \le C\d^{3\over 2} + C\dbE_{\cF_t}\Big[|B^t_{t+\d}|\big[ \int_t^{t+\d}[|X_s| + |Z_s|]ds\Big] \\
&&\dis \le C[1+|X_t|] \d^{3\over 2} + C\d \Big(\dbE_{\cF_t}\big[\int_t^{t+\d} |Z_s|^2ds\big]\Big)^{1\over 2}.
\eeaa
Plug this into \reff{EtZ}, divide both sides by $\d$, and then send $\d\to 0$, we have 
\beaa
Z_t = \pa_x V(t, X_t, \rho_t), \q dt\times d\dbP-\mbox{a.s.}
\eeaa
Then \reff{weak-X} becomes
\beaa
X_t = \xi + \int_0^t \pa_p H(X_s, Z_s) ds + B_t + \b B^0_t.
\eeaa
This, together with \reff{viscosity-YZ}, forms a coupled McKean-Vlasov FBSDE. Since $T\le \d_1$, by the uniqueness of the FBSDE system we see that $V$ coincides with the good solution. 
\qed

\section{Convergence of the Nash system}
\label{sect-convergence}
\setcounter{equation}{0}

In this section we study the convergence of the Nash system \reff{Nash}, arising from the $N$-player game \reff{NashXa}-\reff{NEa*}. For technical reasons we need to strengthen Assumption \ref{assum-H1}.
\begin{assum}
\label{assum-H3}
$H\in C^1(\dbR^{2d})$ and  $\pa_x H, \pa_p H$ are bounded and Lipschitz continuous. 
\end{assum}
This condition is also assumed in \cite{CDLL}. However, \cite{CD2} allows to deal with the case that $H$ is quadratic in $z$, which is covered by Assumption \ref{assum-H1} but unfortunately is excluded here. We shall leave this interesting case for future research.

We first have the global wellposedness of the Nash system. The result is not surprising and we sketch  a proof  in Appendix.  We note that this result does not require the monotonicity condition \reff{mon}, due to the non-degeneracy as mentioned in Introduction, and the Lipschitz continuity of $F, G$ with respect to $\mu$ can be weakened to under $\cW_2$.
\begin{prop}
\label{prop-Nashwellposed}
Let Assumptions \ref{assum-Lip1}  (i)  and  \ref{assum-H3} hold. 
Then the Nash system has a  unique classical solution $v^N=(v^{N,i})_{1\le i\le N}\in C^{1,2}([0, T)\times \dbR^{N\times d}) \cap C^{0}([0, T]\times \dbR^{N\times d})$; the FBSDEs \reff{NashFBSDE1} and  \reff{NashFBSDE2} have unique strong solutions; and the relation \reff{YvN} holds.  

Moreover, there exists a constant $C_N$, which may depend on $N$, such that 
\bea
\label{Nashpaxv}
|\pa_{x_j}v^{N,i}| \le C_N, \q |\pa_{x_jx_k} v^{N,i}(t, \vec x)| \le {C_N\over \sqrt{T-t}}.
\eea
\end{prop}

We remark that, unlike \reff{pauest1}, in general we do not have a uniform bound for $\pa_{x_j} v^i_N$. This is not desirable for the convergence of the Nash system, see Theorem \ref{thm-Nash1} below.   The works \cite{CDLL, CD2} get around of this difficulty by using the boundedness of the second derivatives of $V$, including $\pa^2_{\mu\mu} V$, which however is not possible under our conditions. We shall instead use the  crucial uniform Lipschitz continuity of $V$ established in Section \ref{sect-regularity}.

\subsection{Convergence of the Nash system}

The main result of this subsection is as follows. Recall \reff{FBSDE1}, \reff{uvN}, \reff{NashFBSDE2}, and denote
\bea
\label{rhoN}
\rho^N_t := {1\over N} \sum_{i=1}^N \d_{ X^{\vec x, i}_t},\q  d_\e := \max(d, 2+\e),\q \|\vec x\|^2:= {1\over N} \sum_{i=1}^N |x_i|^2.
\eea
\begin{thm}
\label{thm-Nash1}
Let  Assumptions \ref{assum-Lip1}, \ref{assum-Lip2},  \ref{assum-H2}, \ref{assum-mon}, and \ref{assum-H3} hold.

(i) For any $(t_0,\vec x)\in[0,T]\times \mathbb R^{N\times d}$ and any $\e>0$, we have
\bea
\label{uNconv}
\big|U^N-V|(t_0,x_i,m_{\vec x}^{N,i}) \leq {C_{\e}\over N^{1\over d_\e}}\big[1 + |x_i|+\|\vec x\|\big].
\eea

(ii) For any $(t_0,\vec x)\in[0,T]\times \mathbb R^{N\times d}$, $\xi\in \dbL^2(\cF_{t_0})$,  $\e>0$, $1<p<2$, we have
\bea
\label{rhoNconv}
\Big(\sup_{t_0\le t\le T} \dbE\Big[\cW^p_1(\rho^N_t, \rho_t)\Big]\Big)^{1\over p} \le C_{p,\e}\Big( \dbE\big[\cW^p_1(\rho^N_{t_0}, \rho_{t_0})\big]\Big)^{1\over p} + {C_{p,\e}\over N^{1\over d_\e}}\big[1 +\|\vec x\|\big].
\eea
\end{thm}
As usual, here $C_\e$ depends only on $d$, $T$, $\e$, and the parameters in the assumptions, but not on $N$, and $C_{p,\e}$ may depend on $p$ as well. When $d>2$, since $d_\e$ in \reff{rhoN} does not involve $\e$, then $C_\e, C_{p,\e}$ do not depend on $\e$ either.
 
To obtain the convergence, we  need the following lemma. 
\begin{lem}
\label{lem-empirical}
Let $q>p\ge 1$, $\cX_1,\cds,\cX_N\in \dbL^q(\cF^{B, B^0}_T;\dbR^d)$ be independent random variables, and $E_1,\cds, E_N\in\cF_0$ form a partition of $\O$ with $\dbP(E_i) = {1\over N}$. Denote $\cX:= \sum_{i=1}^N \cX_i \1_{E_i}$. Then, for any $\e>0$, there exists $C_{p,q,\e}>0$, depending only on $d, p, q,$ and $\e$, such that
\bea
\label{empiricalConv}
\dbE\Big[\cW^p_1({1\over N} \sum_{i=1}^N \d_{\cX_i}, \cL_\cX)\Big]\le\dbE\Big[\cW^p_p({1\over N} \sum_{i=1}^N \d_{\cX_i}, \cL_\cX)\Big] \le {C_{p,q,\e}\over N^{p\over d_\e}} \|\cX\|_q^p.
\eea
\end{lem}
 The proof follows similar arguments as in  \cite[Theorem 1]{FG}, see also \cite[Section 5.1.2]{CD1}, and we postpone it to Appendix.

\begin{rem}
\label{rem-empirical}
{\rm (i) The rate in \reff{rhoNconv} is due to \reff{empiricalConv} and is the same as in \cite{CDLL, CD2}. However, \cite{CDLL, CD2} has a better rate ${1\over N}$ in \reff{uNconv}. The approach there for \reff{uNconv}  does not use Lemma \ref{lem-empirical}. Instead it relies on the classical solution of the master equation, in particular on the boundedness of  $\pa_{\mu\mu} V$, which we want to avoid. 

(ii) Note that in Theorem \ref{thm-Nash1} we need only $\dbE[\cW_1^2(\mu^N, \mu)]$, not $\dbE[\cW_2^2(\mu^N, \mu)]$. When $d=1$ and $\|\cX\|_q <\infty$ for some $q>2$, we have a better rate for $\dbE[\cW_1^2(\mu^N, \mu)]$:
\bea
\label{empiricald1}
\dbE[\cW_1^2(\mu^N, \mu)]  \le {C\over N}[1+\|\cX\|_q^q].
\eea
In fact, when $\cX_i$ are i.i.d. \cite{DGM} even has a central limit theorem for the convergence  of $\cW_1(\mu^N, \mu)$. Consequently \reff{uNconv} and \reff{rhoNconv} will have a better rate ${1\over \sqrt{N}}$ when $d=1$:
\bea
\label{uNconvd1}
\left.\ba{c}
\dis\big|U^N-V|(t_0,x_i,m_{\vec x}^{N,i}) \leq {C\over \sqrt{N}}\big[1 + |x_i|+ {1\over N}\sum_{j=1}^N |x_j|^q\big],\\
 \Big(\sup_{t_0\le t\le T} \dbE\big[\cW^2_1(\rho^N_t, \rho_t)\big]\Big)^{1\over 2} \le C\Big( \dbE\big[\cW^2_1(\rho^N_{t_0}, \rho_{t_0})\big]\Big)^{1\over 2} + {C\over \sqrt{N}}\big[1 +{1\over N}\sum_{j=1}^N |x_j|^q\big].
\ea\right.
\eea
We provide a simple proof for \reff{empiricald1} in Appendix.
\qed}
\end{rem}

We next establish a local version of the theorem. Fix $(t_0, \vec x)$, and $\xi\in \dbL^2(\cF_{t_0})$.
\begin{prop}
\label{prop-Nashlocal1}
Assume $F,G,H$ satisfy Assumption \ref{assum-Lip1}, \ref{assum-Lip2}, \ref{assum-H2},  and \ref{assum-H3}. Fix an arbitrary $1<p<2$. Let  $G': \dbR^d \times \cP_2\to \dbR$ be Lipschitz continuous and change the terminal condition of \reff{NashFBSDE2} to $G'$. Then there exist  $C, \delta>0$, depending on $d$, $p$, and the parameters in the assumptions, but independent of $N$ and the Lipschitz constant of $G'$,  such that: whenever $T-t_0\le \d$, for any  $\e>0$, for some constant $C_{\e}$ which may depend on $\e$ as well,
\bea
\label{muNconvlocal}
\left.\ba{c}
\dis \max_{1\le i\le n} | Y^{\vec x, i}_{t_0}-V(t_0, x_i, \rho_{t_0})|+\sup_{t_0\le t\le T} \Big(\dbE_{\cF^0_{t_0}}\big[\cW_1^p(\rho^N_t, \rho_t)\big]\Big)^{1\over p} \le CI_p + {C_{\e}\over N^{1\over d_\e}}\big[1 + \|\vec x\|\big], \\
\dis \mbox{where}~ I^p_p:=  \cW^p_1(\rho^N_{t_0}, \rho_{t_0}) +{1\over N}\sum_{i=1}^N\mathbb\dbE_{\cF^0_{t_0}}\Big[|G'-G|^p( X_T^{\vec x, i}, m_{ X_T^{\vec x,\rightarrow}}^{N,i})\Big].
\ea\right.
\eea
\end{prop}
We remark that, since both $X^{\vec x,\rightarrow}$ and $\rho$ are independent of $\cF_0\vee \cF^B_t$, then we may replace the $\dbE_{\cF^0_{t_0}}$ in both places of \reff{muNconvlocal} to $\dbE_{\cF_{t_0}}$.
 
 \proof  By considering conditional distribution on $\cF^0_{t_0}$, we may assume without loss of generality that $t_0=0$. In this case, notice that $\dbE_{\cF^0_0} = \dbE$. 
  
 Let $\xi_{\vec x}\in \dbL^2(\cF_0, \rho^N_0)$ and consider FBSDEs on $[0, T]$:
\bea
\label{NashFBSDExix}
&& \dis \left\{\begin{array}{lll} 
  \dis X_{t}^{\xi,\vec x}= \xi_{\vec x} + \int_0^t\pa_pH(X_{s}^{\xi,\vec x},Z_{s}^{\xi,\vec x})ds+ B_t+ \beta B_t^{0}; \\
\dis Y_{t}^{\xi, \vec x}= G(X_T^{\xi, \vec x},\rho_T)+\int_t^T\big[F(X_{s}^{\xi,\vec x},\rho_s) -\h L(X_{s}^{\xi,\vec x},Z_{s}^{\xi, \vec x})\big]ds\\
\dis\qq -\int_t^T Z_{s}^{\xi,\vec x}\cd dB_s-\int_t^TZ_{s}^{0,\xi,\vec x}\cd dB_s^{0}.
\ea\right.\ms\\
\label{NashFBSDExixi}
&&\dis  \left\{\ba{lll} 
  \dis X_{t}^{\xi,\vec x, i}=x_i + \int_0^t\pa_pH(X_{s}^{\xi,\vec x, i},Z_{i,s}^{\xi,\vec x, i})ds+ B^{i}_t+ \beta B_t^{0}; \\
\dis Y_{t}^{\xi, \vec x, i}=G(X_T^{\xi, \vec x, i},\rho_T)+\int_t^T\big[F(X_{s}^{\xi,\vec x, i},\rho_s) -\h L(X_{s}^{\xi,\vec x, i},Z_{i,s}^{\xi, \vec x,i})\big]ds\\
\dis\qq -\int_t^TZ_{i,s}^{\xi,\vec x,i}\cd dB_s^i-\int_t^TZ_{s}^{0,\xi,\vec x,i}\cd dB_s^{0}.
\ea\right.
\eea
By Proposition \ref{prop-classical-x} (iii), the BSPDE in \reff{SPDE} (with the given $\rho$) has a unique weak solution $u$ such that $|\pa_x u|\le C$ (independent of $N$) and $\pa_x u$ is uniformly Lipschitz continuous in $x$. Then it is clear that the above FBSDEs are wellposed with $Z^{\xi, \vec x}_s = \pa_x u(s, X^{\xi, \vec x}_s)$, $Z_{i,s}^{\xi,\vec x,i}= \pa_x u(s, X_{i,s}^{\xi,\vec x,i})$ uniformly bounded.  Compare \reff{FBSDE1} and \reff{NashFBSDExix}, and note that
\bea
\label{Xpaxu}
\left.\ba{c}
\dis X^\xi_t = \xi + \int_0^t \pa_p H(X^\xi_s, \pa_x u(s, X^\xi_s)) ds + B_t + \b B^0_t;\\
\dis X^{\xi,\vec x}_t = \xi_{\vec x} + \int_0^t \pa_p H(X^{\xi,\vec x}_s, \pa_x u(s, X^{\xi,\vec x}_s)) ds + B_t + \b B^0_t.
\ea\right.
\eea
By the Lipschitz continuity of $\pa_p H$ and $\pa_x u$ we have $|X^\xi_t - X^{\xi,\vec x}_t|\le C |\xi-\xi_{\vec x}|$.  Then
\beaa
\cW_1\big( \cL_{X^{\xi, \vec x}_t|\cF^0_t}, \rho_t\big) = \cW_1\big( \cL_{X^{\xi, \vec x}_t|\cF^0_t},  \cL_{X^{\xi}_t|\cF^0_t}\big)\le \dbE_{\cF^0_t}\big[|X^\xi_t - X^{\xi,\vec x}_t|\big] \le C\dbE_{\cF^0_t}\big[|\xi-\xi_{\vec x}|\big].
\eeaa
Choose $\xi, \xi_{\vec x}\in \dbL^2(\cF_0)$ appropriately such  that $\dbE[|\xi-\xi_{\vec x}|]  = \cW_1(\rho^N_0, \rho_0)$. Then, noting that $\xi,\xi_{\vec x}$ are independent of $\cF^0_t$, 
\bea
\label{uNest1}
\cW_1\big( \cL_{X^{\xi, \vec x}_t|\cF^0_t}, \rho_t\big) \le C\dbE[|\xi-\xi_{\vec x}|]= C\cW_1(\rho^N_0, \rho_0).
\eea
Moreover, note that the systems \reff{NashFBSDExixi} for $i=1,\cds, N$ are conditionally independent, conditional on $\dbF^0$. By applying Lemma \ref{lem-empirical} under $\dbE_{\cF^0_t}$, more precisely by using the regular conditional probability distribution of \cite{SV}, and by \reff{Xpaxu} we have
\bea
\label{uNest2}
\left.\ba{lll}
\dis \dbE\Big[\cW^p_1\big({1\over N}\sum_{i=1}^N \d_{X_{t}^{\xi,\vec x, i}}, \cL_{X^{\xi, \vec x}_t|\cF^0_t}\big)\Big] = \dbE\Big[\dbE_{\cF^0_t}\big[\cW^p_1\big({1\over N}\sum_{i=1}^N \d_{X_{t}^{\xi,\vec x, i}}, \cL_{X^{\xi, \vec x}_t|\cF^0_t}\big)\big]\Big] \\
\dis\le \dbE\Big[{C_{\e} \over N^{p\over d_\e}} \dbE_{\cF^0_t} \big[\|X^{\xi, \vec x}_t\|^p\big] \Big] \le {C_{\e} \over N^{p\over d_\e}} \dbE\Big[ \Big(\dbE_{\cF^0_t} \big[\|X^{\xi, \vec x}_t\|^2\big] \Big)^{p\over 2}\Big]\\
\dis \le {C_{\e} \over N^{p\over d_\e}} \dbE\Big[\Big(|B^0_t|^2 + \int_0^t |B^0_s|^2ds  +\|\vec x\|^2\Big)^{p\over 2}\Big] \le {C_{\e} \over N^{p\over d_\e}} \big[1  +\|\vec x\|^p\big] .
\ea\right.
\eea

We next show  that:  whenever $T\le \d$,
\bea
\label{uNest3}
\left.\ba{c}
\dis \dbE\big[\cW^p_1\big(\rho^N_t, {1\over N}\sum_{i=1}^N\d_{X^{\xi, \vec x, i}_t}\big)\big] \le  C\d^{p\over 2}\sup_{0\le t\le T} \dbE\big[\cW^p_1(\rho^N_t, \rho_{t})\big]  +CI_p^p+ {C\over N^p} \big[1+\|\vec x\|^p\big].
\ea\right.
\eea
 Recall   \reff{NashFBSDE2}, \reff{NashFBSDExixi}, and denote $\D \Phi^i:= \Phi^{\vec x, i} - \Phi^{\xi, \vec x, i}$, $\Phi = X, Y, Z_j, Z^0$, where $\Phi^{\xi, \vec x, i}_{j,s}:=0$ for $j\neq i$. Note that
 \beaa
 &&\dis G'( X_T^{\vec x, i}, m_{ X_T^{\vec x,\rightarrow}}^{N,i}) - G(X_T^{\xi,\vec x, i}, \rho_T)\\
 &&\dis\q =[G'( X_T^{\vec x, i}, m_{ X_T^{\vec x,\rightarrow}}^{N,i}) - G( X_T^{\vec x, i}, \rho_T)] + [G( X_T^{\vec x, i}, \rho_T) - G(X_T^{\xi,\vec x, i}, \rho_T)];\\
 &&\dis \h L( X_{s}^{\vec x, i}, Z_{i,s}^{\vec x,i})-\h L(X_{s}^{\xi,\vec x, i},Z_{i,s}^{\xi, \vec x,i}) = [H(X_{s}^{\xi,\vec x, i},Z_{i,s}^{\xi, \vec x,i}) - H( X_{s}^{\vec x, i}, Z_{i,s}^{\vec x,i}) ] \\
 &&\dis \q +\D Z^i_{i,s}\cd \pa_p H( X_{s}^{\vec x, i}, Z_{i,s}^{\vec x,i}) + Z_{i,s}^{\xi, \vec x,i}\cd [\pa_p H( X_{s}^{\vec x, i}, Z_{i,s}^{\vec x,i})-\pa_p H(X_{s}^{\xi,\vec x, i},Z_{i,s}^{\xi, \vec x,i})],
 \eeaa
and similarly for the $F$ term. Then
\bea
\label{uNDXi}
\left.\ba{c}
\dis \D X^i_t =\int_0^t [\l_s \D X^i_s + \l_s \D Z^i_{i,s}] ds;\\
\dis \D Y^i_t = G'( X_T^{\vec x, i}, m_{ X_T^{\vec x,\rightarrow}}^{N,i}) - G( X_T^{\vec x, i}, \rho_T) +\l_T\cd \D X^i_T + \int_t^T [\l_s \cd\D X^i_s + \l_s\cd \D Z^i_{i,s}] ds \\
\dis +  \int_t^T[F( X_{s}^{\vec x, i}, m_{ X_s^{\vec x,\rightarrow}}^{N,i})-F( X_{s}^{\vec x, i},\rho_s)]ds- \sum_{j=1}^T \int_t^T \D Z^i_{j,s}\cd dB^j_s - \int_t^T\D Z^{0,i}_s\cd dB^0_s ,
\ea\right.
\eea
where the generic process $\l$ is uniformly bounded, thanks to the Lipschitz continuities and the fact $|Z^{\xi, \vec x, i}_{i,s}|\le C$.  One can easily check that
\bea
\label{DXest}
&&\dis \sup_{0\le t\le T} |\D X^i_t|\le C\int_0^T |\D Z^i_{i,s}|ds;\nonumber\\
&&\dis |\D Y^i_0| = \Big|\dbE\Big[G'( X_T^{\vec x, i}, m_{ X_T^{\vec x,\rightarrow}}^{N,i}) - G( X_T^{\vec x, i}, \rho_T) +\l_T\cd \D X^i_T + \int_0^T [\l_s\cd \D X^i_s + \l_s\cd \D Z^i_{i,s}] ds \nonumber\\
&&\dis\qq\qq\qq +  \int_0^T[F( X_{s}^{\vec x, i}, m_{ X_s^{\vec x,\rightarrow}}^{N,i})-F( X_{s}^{\vec x, i},\rho_s)]ds\Big]\Big|\\
&&\dis\q \le C\dbE\Big[|G'-G|( X_T^{\vec x, i}, m_{ X_T^{\vec x,\rightarrow}}^{N,i}) +   \cW_1(m_{X_T^{\vec x,\rightarrow}}^{N,i}, \rho_T) + \int_0^T [|\D Z^i_{i,s}|  +\cW_1(m_{ X_s^{ x,\rightarrow}}^{N,i},\rho_s)]ds\Big].\nonumber
\eea
We emphasize that here we used only the Lipschitz continuity of $G$, not of $G'$. Then, by the Burkholder-Davis-Gundy inequality and the Doob's maximum inequality we have
\beaa
&&\dis\!\!\!\! \dbE\Big[\Big(\int_0^T |\D Z^i_{i,s}|^2 ds \Big)^{p\over 2}\Big] \le  \dbE\Big[\Big(\int_0^T [\sum_{j=1}^N|\D Z^i_{j,s}|^2+ |\D Z^{0,i}_{s}|^2 ]ds \Big)^{p\over 2}\Big] \\
&&\dis\!\!\!\! \le C \dbE\Big[\Big|\sum_{j=1}^T \int_0^T \D Z^i_{j,s}\cd dB^j_s + \int_0^T\D Z^{0,i}_s\cd dB^0_s \Big|^p\Big]\\
&&\dis\!\!\!\! = C\dbE\Big[\Big|G'( X_T^{\vec x, i}, m_{ X_T^{\vec x,\rightarrow}}^{N,i}) - G( X_T^{\vec x, i}, \rho_T) +\l_T\cd \D X^i_T + \int_t^T [\l_s\cd \D X^i_s + \l_s\cd \D Z^i_{i,s}] ds \\
&&\dis\!\!\!\!\qq +  \int_t^T[F( X_{s}^{\vec x, i}, m_{\h X_s^{\vec x,\rightarrow}}^{N,i})-F( X_{s}^{\vec x, i},\rho_s)]ds - \D Y^i_0 \Big|^p\Big]\\
&&\dis\!\!\!\! \le C_0\d^{p\over 2} \dbE\Big[ \Big(\int_0^T |\D Z^i_{i,s}|^2 ds \Big)^{p\over 2}\Big] + C\dbE\big[|G'-G|^p( X_T^{\vec x, i}, m_{ X_T^{\vec x,\rightarrow}}^{N,i})\big] + C\sup_{0\le t\le T} \dbE\big[\cW^p_1(m_{ X_t^{\vec x,\rightarrow}}^{N,i},\rho_t)\big].
\eeaa
By choosing $\d\le {1\over (2C_0)^{2\over p}}$ for the above $C_0$, we obtain
\beaa
\dbE\Big[\Big(\int_0^T |\D Z^i_{i,s}|^2 ds \Big)^{p\over 2}\Big] \le   C\dbE\big[|G'-G|^p( X_T^{\vec x, i},m_{ X_T^{\vec x,\rightarrow}}^{N,i})\big] + C\sup_{0\le t\le T} \dbE\big[\cW^p_1(m_{ X_t^{\vec x,\rightarrow}}^{N,i},\rho_t)\big].
\eeaa
 Note that, for any $q\ge 1$,
 \bea
 \label{EXlinear}
 \dbE[| X_t^{\vec x, i}|^q] \le C[1+|x_i|^q],
 \eea
thanks to the boundedness of $\pa_p H$. Then,  for $1<p<2$
\bea
\label{mrho}
\dbE\big[\cW^p_1(m_{ X_t^{\vec x,\rightarrow}}^{N,i}, \rho^N_t)\big] &\le&  C\dbE\Big[\Big({1\over N(N-1)}\sum_{j\neq i} | X_t^{\vec x, j}| + {1\over N} | X^{\vec x, i}_t| \Big)^p\Big]\nonumber\\
&\le& {C\over N^p}\Big[ 1+|x_i|^p + \|\vec x\|^p \Big].
\eea
Therefore,
\beaa
&\dis\dbE\Big[\Big(\int_0^T |\D Z^i_{i,s}|^2 ds \Big)^{p\over 2}\Big] \le   C\dbE\big[|G'-G|^p( X_T^{\vec x, i},m_{ X_T^{\vec x,\rightarrow}}^{N,i})\big]  \\
&\dis+ C\sup_{0\le t\le T} \dbE\Big[\cW_1^p( \rho^N_t, \rho_t)\Big] +{C\over N^p}\Big[ 1+|x_i|^p + \|\vec x\|^p\Big] . 
\eeaa
By the first line of \reff{DXest} again we have
\bea
\label{rhoNerr}
&&\dis\dbE\Big[\cW^p_1\big(\rho^N_t, {1\over N}\sum_{i=1}^N\d_{X^{\xi, \vec x, i}_t}\big)\Big] \le \dbE\Big[\big({1\over N}\sum_{i=1}^N |\D X^i_t|\big)^p\Big] \le {1\over N}\sum_{i=1}^N\dbE[ |\D X^i_t|^p]\nonumber\\
&&\dis \le  {C\over N}\sum_{i=1}^N\dbE\Big[\big(\int_0^T |\D Z^i_{i,s}|ds\big)^p\Big]\le  {C\d^{p\over 2}\over N}\sum_{i=1}^N\dbE\Big[\Big(\int_0^T |\D Z^i_{i,s}|^2 ds \Big)^{p\over 2}\Big] \nonumber \\
&&\dis \le C\d^{p\over 2}\sup_{0\le s\le T} \dbE\Big[\cW_1^2( \rho^N_s, \rho_s)\Big] +CI_p^p + {C\over N^p}\Big[ 1+ \|\vec x\|^p \Big].
\eea
This is exactly \reff{uNest3}.

Now, by \reff{uNest1}, \reff{uNest2}, and \reff{uNest3}, we have
\beaa
&&\dis \dbE\big[\cW^p_1(\rho^N_t, \rho_t)\big] \le C\dbE\Big[\cW^p_1\big(\rho^N_t, {1\over N}\sum_{i=1}^N\d_{X^{\xi, \vec x, i}_t}\big)  \\
&&\dis\qq+\cW_1^p\big({1\over N}\sum_{i=1}^N \d_{X_{t}^{\xi,\vec x, i}}, \cL_{X^{\xi, \vec x}_t|\cF^0_t}\big) + \cW_1^p\big( \cL_{X^{\xi, \vec x}_t|\cF^0_t}, \rho_t\big) \Big]\\
&&\le  C_0\d^{p\over 2}\sup_{0\le s\le T} \dbE\Big[\cW_1^p( \rho^N_s, \rho_s)\Big] +CI_p^p  +{C_{\e} \over N^{p\over d_\e}} \Big[1 +\|\vec x\|^p\Big].
\eeaa
Again set $\d\le {1\over (2C_0)^{2\over p}}$ for the above $C_0$,  we obtain  
\beaa
\sup_{0\le t\le T}  \dbE\big[\cW^p_1(\rho^N_t, \rho_t)\big] \le CI_p^p  +{C_{\e} \over N^{p\over d_\e}} \Big[1 +\|\vec x\|^p\Big].
\eeaa
Finally, plugging the above estimates into \reff{DXest} we can easily get 
\beaa
|\D Y^i_0|^p\le CI_p^p  +{C_{\e} \over N^{p\over d_\e}} \Big[1 +\|\vec x\|^p\Big].
\eeaa
Notice that $Y^{\xi, \vec x, i}_0 = V(0, x_i, \rho_0)$, we obtain \reff{muNconvlocal} at $t_0=0$ immediately. 
\qed

\ms

\no{\bf Proof of Theorem \ref{thm-Nash1}.}  First, by Proposition \ref{prop-classical-x} (iii) and Theorem \ref{thm-apriori} we see that $V$ is uniformly Lipschitz continuous in $(x, \mu)$ and $\pa_x V$ is uniformly Lipschitz continuous in $x$. Fix $1<p<2$ and let $\d>0$ be as in Proposition \ref{prop-Nashlocal1}, where the constants $C, C_{\e}$ depend on the regularity of $V$ instead of  $G$. Note that in (i) actually we can set $p={3\over 2}$, so the constant $C_\e$ in \reff{uNconv} actually does not depend on the $p$ in (ii). Set $t_0<\cds<t_n=T$ be such that $t_i - t_{i-1} \le \d$. Note that $\d$ is independent of $N$, then so is $n$. 

(i) Fix $(t_k, \vec x)$ and $p={3\over 2}$. Consider FBSDE \reff{FBSDE1} on $[t_k, t_{k+1}]$ with initial condition $\xi_{\vec x}\in \dbL^2(\cF_0\vee \cF^B_{t_k})$ and terminal condition $V(t_{k+1},\cd,\cd)$, and FBSDE system \reff{NashFBSDE2} on $[t_k, t_{k+1}]$ with initial condition $\vec x$ and terminal condition $U^N(t_{i+1},\cd,\cd)$. Note that  $V(t_{i+1},\cd,\cd)$ is uniformly Lipschitz continuous in $(x,\mu)$ and the initial conditions $\xi_{\vec x}$, $\vec x$ are independent of $\cF^0_{t_k}$. By Proposition \ref{prop-Nashlocal1}  we have: by using the superscript $^{t_k}$ is to indicate the initial time $t_k$,
\beaa
&&\dis |U^N(t_k, x_i, m^{N,i}_{\vec x})-V(t_k, x_i, {1\over N} \sum_{j=1}^N \d_{x_j})|^p \\
&&\dis \le {C\over N}\sum_{j=1}^N \mathbb\dbE\Big[|U^N-V|^p(t_{k+1}, X_{t_{k+1}}^{t_k, \vec x, j}, m_{ X_{t_{k+1}}^{t_k, \vec x,\rightarrow}}^{N,j})\Big] + {C_{\e}\over N^{p\over d_\e}}\big[1 + \|\vec x\|^p\big].
\eeaa
Since $V$ is Lipschitz, similar to \reff{mrho} we have
\beaa
|V(t_k, x_i, m^{N,i}_{\vec x})-V(t_k, x_i, {1\over N} \sum_{j=1}^N \d_{x_j})|\le C\cW_1(m^{N,i}_{\vec x}, {1\over N} \sum_{j=1}^N \d_{x_j})\le {C\over N}\big[|x_i|+\|\vec x\|\big].
\eeaa
Then
\beaa
&&\dis  |U^N-V|^p(t_k, x_i, m^{N,i}_{\vec x})\\
&&\dis \le {C\over N}\sum_{j=1}^N\mathbb\dbE\big[|U^N-V|^p(t_{k+1}, X_{t_{k+1}}^{t_k, \vec x, j}, m_{ X_{t_{k+1}}^{t_k, \vec x,\rightarrow}}^{N,j})\big] + {C_{\e}\over N^{p\over d_\e}}\big[1 +|x_i|^p+\|\vec x\|^p \big].
\eeaa
Denote 
\beaa
\G_k := \sup_{1\le i\le N}\sup_{\vec x} {|U^N-V|(t_k, x_i, m^{N,i}_{\vec x})\over  1 + |x_i|+ \|\vec x\|}.
\eeaa
Then, by \reff{EXlinear},
\beaa
&&\dis  |U^N-V|^p(t_k, x_i, m^{N,i}_{\vec x}) \\
&&\dis\le {C\G_{k+1}^p\over N} \sum_{j=1}^N \dbE\big[1 + |X_{t_{k+1}}^{t_k, \vec x, j}|^p+ \| X_{t_{k+1}}^{t_k, \vec x,\rightarrow}\|^p \big]  +{C_{\e}\over N^{p\over d_\e}}\big[1 + |x_i|^p+ \|\vec x\|^p \big]\\
&&\dis\le \big[C \G^p_{k+1} + {C_{\e}\over N^{p\over d_\e}}\big]\big[1 + |x_i|^p+ \|\vec x\|^p \big].
\eeaa
Thus
\beaa
\G_k^p \le C\G_{k+1}^p+ {C_{\e}\over N^{p\over d_\e}}, \q k=0,\cds, n-1.
\eeaa
We emphasize again that $n$ does not depend on $N$. Since $\G_n =0$, by backward induction on $k$ the above implies $\G_0 \le {C_{\e}\over N^{1\over d_\e}}$, which leads to \reff{uNconv} immediately.

(ii) Fix $t_k$. Consider FBSDE \reff{FBSDE1} on $[t_k, t_{k+1}]$ with initial condition $X^\xi_{t_k}$ and terminal condition $V(t_{k+1},\cd,\cd)$, and FBSDE system \reff{NashFBSDE2} on $[t_k, t_{k+1}]$ with initial condition $ X^{\vec x,\rightarrow}_{t_k}$ and terminal condition $U^N(t_{k+1},\cd,\cd)$. Note that  $V(t_{k+1},\cd,\cd)$ is uniformly Lipschitz continuous in $(x,\mu)$. Then, applying Proposition \ref{prop-Nashlocal1} conditionally on $\cF_{t_k}$, by \reff{uNconv} and \reff{EXlinear} we have
\bea
\label{rhokerror}
&&\dis \sup_{t_k\le t\le t_{k+1}} \dbE_{\cF_{t_k}}[\cW_1^p(\rho^N_t, \rho_t)]\nonumber \\
&&\dis \le C \cW_1^p(\rho^N_{t_k}, \rho_{t_k})+{C\over N}\sum_{j=1}^N \mathbb\dbE_{\cF_{t_k}}\big[|U^N-V|^p(t_{k+1},  X_{t_{k+1}}^{\vec x, j},m_{ X_{t_{k+1}}^{\vec x,\rightarrow}}^{N,j})\big] + {C_{\e}\over N^{p\over d_\e}}\big[1 + \| X^{\vec x,\rightarrow}_{t_k}\|^p\big]\nonumber\\
&&\dis \le C \cW_1^p(\rho^N_{t_k}, \rho_{t_k})+{C_{\e}\over N^{1+{p\over d_\e}}}\sum_{j=1}^N\mathbb\dbE_{\cF_{t_k}}\big[1+| X_{t_{k+1}}^{\vec x, j}|^p + \| X_{t_{k+1}}^{\vec x,\rightarrow}\|^p\big] + {C_{\e}\over N^{p\over d_\e}}\big[1 + \|X^{\vec x,\rightarrow}_{t_k}\|^p\big]\nonumber\\
&&\dis \le C \cW^p_1(\rho^N_{t_k}, \rho_{t_k})+ {C_{\e}\over N^{p\over d_\e}}\big[1 + \| X^{\vec x,\rightarrow}_{t_k}\|^p\big].
\eea
In particular, this implies
\beaa
\dbE\Big[\cW^p_1(\rho^N_{t_{k+1}}, \rho_{t_{k+1}})\Big] &\le&  \dbE\Big[C\cW^p_1(\rho^N_{t_k}, \rho_{t_k})+ {C_{\e}\over N^{p\over d_\e}}\big[1 + \| X^{\vec x,\rightarrow}_{t_k}\|^p\big]\Big] \\
&\le&  C\dbE\Big[\cW^p_1(\rho^N_{t_k}, \rho_{t_k})\Big] + {C_{\e}\over N^{p\over d_\e}}\big[1 + \|\vec x\|^p\big],
\eeaa
 By induction on $k=0,\cds, n-1$, we get
\beaa
\max_{0\le k\le n} \dbE\Big[\cW^p_1(\rho^N_{t_k}, \rho_{t_k})\Big] \le {C_{\e}\over N^{p\over d_\e}}\big[1 + \|\vec x\|^p\big].
\eeaa
This, together with \reff{rhokerror}, implies \reff{rhoNconv} immediately.
\qed

\subsection{Propagation of chaos} 
 In the literature quite often people consider a slight different convergence. Let $\xi_1,\cds, \xi_N\in \dbL^2(\cF_0, \mu)$ be independent copies of $\xi$ and denote $\vec \xi = (\xi_1,\cds, \xi_N)$. Consider the following systems of FBSDEs on $[0, T]$: for $i=1,...,N$ and denoting $X^{N, \vec \xi} = (X^{N, \vec \xi,1},\cds, X^{N, \vec \xi,N})$,
 \bea
\label{chaosFBSDE}
&&\dis\left\{\ba{lll}
\dis X^{\vec\xi, i}_t=\xi_i + \int_0^t\pa_pH(X_s^{\vec \xi,i},Z_s^{\vec\xi,i})ds+B^i_t+\b B_t^{0}; \\
\dis Y_t^{\vec\xi,i}=G(X_T^{\vec\xi,i},\rho_T)+\int_t^T[F(X_s^{\vec\xi,i},\rho_s) -\wh L(X_s^{\vec\xi,i},Z_s^{\vec\xi,i})]ds\\
\dis\qq\q - \int_t^TZ_s^{\vec\xi,i}\cd dB^i_s-\int_t^TZ_s^{0,\vec\xi,i}\cd dB_s^{0};
\ea\right.\\
\label{chaosNFBSDE}
&&\dis\left\{\ba{lll}
\dis X^{N, \vec\xi, i}_t=\xi_i + \int_0^t\pa_pH(X_s^{N,\vec \xi,i},Z_{i,s}^{N,\vec\xi,i})ds+B^i_t+\b B_t^{0}; \\
\dis Y_t^{N,\vec\xi,i}=G(X_T^{N,\vec\xi,i} m^{N,i}_{X^{N,\vec\xi}_T})+\int_t^T[F(X_s^{N,\vec\xi,i} m^{N,i}_{X^{N,\vec\xi}_s}) -\wh L(X_s^{N,\vec\xi,i},Z_{i,s}^{N,\vec\xi,i})]ds\\
\dis\qq\q - \sum_{j=1}^N\int_t^TZ_{j,s}^{N,\vec\xi,i}\cd dB^i_s-\int_t^TZ_s^{0,N,\vec\xi,i}\cd dB_s^{0};
\ea\right.
\eea
In \reff{chaosFBSDE}, $\rho_t := \cL_{X^{\vec\xi, i}_t|\cF^0_t}$. We emphasize that  \reff{chaosFBSDE} are conditionally independent copies of \reff{FBSDE1}, conditional on $\dbF^0$, in particular $\rho$ is the same as in \reff{FBSDE1} and does not depend on $i$.

\begin{thm}
\label{thm-chaos}
Let Assumptions \ref{assum-Lip1}, \ref{assum-Lip2},  \ref{assum-H2}, \ref{assum-mon}, and \ref{assum-H3} hold and fix $1<p<2$. Then
\bea
\label{chaosconv}
 \dbE\Big[\sup_{0\le t\le T} |X^{N, \vec \xi,i}_t - X^{\vec \xi,i}_t|^p\Big] \le  {C_{p,\e}\over N^{p\over d_\e}}\big[1 + \|\xi\|_2^p\big],\q \forall i.
 \eea
\end{thm}
\proof First, by Theorem \ref{thm-well-posedness} and Proposition \ref{prop-Nashwellposed} we know \reff{chaosFBSDE} and \reff{chaosNFBSDE} are wellposed. Next, for any $\vec x\in \dbR^{N\times d}$, conditional on $\vec \xi = \vec x$, we note that \reff{chaosNFBSDE} has a the same (conditional) distribution as \reff{NashFBSDE2}. Then by \reff{rhoNconv} we have
\beaa
\sup_{0\le t\le T}\dbE\Big[\cW_1^p({1\over N}\sum_{i=1}^N \d_{X^{N, \vec \xi,i}_t}, \rho_t)\big|\vec \xi = \vec x\Big]\le C_{\e}\cW_1^p({1\over N}\sum_{i=1}^N \d_{x_i}, \rho_0) + {C_{\e}\over N^{p\over d_\e}}\big[1 + \|\vec x\|^p\big].
\eeaa
Here $C_\e=C_{p,\e}$ may depend on $p$ as well. Then, by Lemma \ref{lem-empirical},
\bea
\label{chaosrhoN}
\left.\ba{c}
\dis\sup_{0\le t\le T}\dbE\Big[\cW_1^p({1\over N}\sum_{i=1}^N \d_{X^{N, \vec \xi,i}_t}, \rho_t)\Big]\\
\dis\le C_{\e}\dbE\Big[\cW_1^p({1\over N}\sum_{i=1}^N \d_{\xi_i}, \rho_0) + {1\over N^{p\over d_\e}}\big[1 + \|\vec \xi\|^p\big]\le  {C_{\e}\over N^{p\over d_\e}}\big[1 + \|\xi\|_2^p\big].
\ea\right.
\eea
Moreover, similar to \reff{mrho} we can show that, for any $i$,
\bea
\label{chaosmNi}
\sup_{0\le t\le T}\dbE\Big[\cW_1^p(m^{N,i}_{X^{N, \vec \xi}_t}, \rho_t)\Big]\le   {C_{\e}\over N^{p\over d_\e}}\big[1 + \|\xi\|_2^p\big].
\eea

Now let $\d>0$ and $0=t_0<\cds<t_n=T$ be as in the proof of Theorem \ref{thm-Nash1}. For each $k=0,\cds, n-1$, on $[t_k, t_{k+1}]$ we may rewrite \reff{chaosFBSDE} and \reff{chaosNFBSDE} as 
\beaa
&&\dis\left\{\ba{lll}
\dis X^{\vec\xi, i}_t= X^{\vec\xi, i}_{t_k} + \int_{t_k}^t\pa_pH(X_s^{\vec \xi,i},Z_s^{\vec\xi,i})ds+B^{i, t_k}_t+\b B_t^{0,t_k}; \\
\dis Y_t^{\vec\xi,i}=V(t_{k+1}, X_{t_{k+1}}^{\vec\xi,i},\rho_{t_{k+1}})+\int_t^{t_{k+1}}[F(X_s^{\vec\xi,i},\rho_s) -\wh L(X_s^{\vec\xi,i},Z_s^{\vec\xi,i})]ds\\
\dis\qq\q - \int_t^{t_{k+1}}Z_s^{\vec\xi,i}\cd dB^i_s-\int_t^{t_{k+1}}Z_s^{0,\vec\xi,i}\cd dB_s^{0};
\ea\right.\\
&&\dis\left\{\ba{lll}
\dis X^{N, \vec\xi, i}_t=X^{N, \vec\xi, i}_{t_k} + \int_{t_k}^t\pa_pH(X_s^{N,\vec \xi,i},Z_{i,s}^{N,\vec\xi,i})ds+B^{i, t_k}_t+\b B_t^{0,t_k}; \\
\dis Y_t^{N,\vec\xi,i}=U^N(t_{k+1}, X_{t_{k+1}}^{N,\vec\xi,i},m^{N,i}_{X^{N, \vec \xi}_{t_{k+1}}})+\int_t^{t_{k+1}}[F(X_s^{N,\vec\xi,i} m^{N,i}_{X^{N,\vec\xi}_s}) -\wh L(X_s^{N,\vec\xi,i},Z_{i,s}^{N,\vec\xi,i})]ds\\
\dis\qq\q - \sum_{j=1}^N\int_t^{t_{k+1}}Z_{j,s}^{N,\vec\xi,i}\cd dB^i_s-\int_t^{t_{k+1}}Z_s^{0,N,\vec\xi,i}\cd dB_s^{0};
\ea\right.
\eeaa
Since $V(t_{k+1},\cd,\cd)$ is uniformly Lipschitz continuous and note that $Z^{\vec \xi, i}$ is uniformly bounded, similar to the arguments \reff{uNest3}  we can easily show that, recalling that $t_{k+1}-t_k\le \d$, 
\beaa
&\dis\dbE\Big[\sup_{t_k\le t\le t_{k+1}}| X^{N, \vec\xi, i}_t- X^{\vec\xi, i}_t|^p\Big] \le C\dbE\Big[| X^{N, \vec\xi, i}_{t_k}- X^{\vec\xi, i}_{t_k}|^p\\
 &\dis + |U^N-V|^p(t_{k+1}, X_{t_{k+1}}^{N,\vec\xi,i},m^{N,i}_{X^{N, \vec \xi}_{t_{k+1}}})+\cW_1^p(m^{N,i}_{X^{N, \vec \xi}_{t_{k+1}}}, \rho_{t_{k+1}})+\int_{t_k}^{t_{k+1}}\cW_1^p(m^{N,i}_{X^{N, \vec \xi}_t}, \rho_t)dt\Big].
\eeaa 
Then, by \reff{uNconv}, \reff{EXlinear}, and \reff{chaosmNi} we have
\bea
\label{DXN1}
&&\dis\dbE\Big[\sup_{t_k\le t\le t_{k+1}}| X^{N, \vec\xi, i}_t- X^{\vec\xi, i}_t|^p\Big] \nonumber\\
&&\dis \le C\dbE\Big[| X^{N, \vec\xi, i}_{t_k}- X^{\vec\xi, i}_{t_k}|^p\Big] + {C_{\e}\over N^{p\over d_\e}}\dbE\Big[1 + |X_{t_{k+1}}^{N,\vec\xi,i}|^p + \|X^{N, \vec \xi}_{t_{k+1}}\|^p+\|\xi\|_2^p\Big]\nonumber\\
&&\dis \le C\dbE\Big[| X^{N, \vec\xi, i}_{t_k}- X^{\vec\xi, i}_{t_k}|^p\Big] + {C_{\e}\over N^{p\over d_\e}}\big[1 +\|\xi\|_2^p\big].
\eea
In particular, this implies
\beaa
\dbE\Big[| X^{N, \vec\xi, i}_{t_{k+1}}- X^{\vec\xi, i}_{t_{k+1}}|^p\Big]\le C\dbE\Big[| X^{N, \vec\xi, i}_{t_k}- X^{\vec\xi, i}_{t_k}|^p\Big] + {C_{\e}\over N^{p\over d_\e}}\big[1 +\|\xi\|_2^p\big].
\eeaa
Note that $X^{N, \vec\xi, i}_{t_0}= X^{\vec\xi, i}_{t_0}$, and $\d$ hence $n$ do not depend on $N$. By induction on $k$ we have 
\beaa
\sup_{k=0,\cds, n}\dbE\Big[| X^{N, \vec\xi, i}_{t_k}- X^{\vec\xi, i}_{t_k}|^2\Big] \le {C_{q,\e}\over N^{2\over d_\e}}\big[1 +\|\xi\|_q^2\big].
\eeaa
This, together with \reff{DXN1}, implies \reff{chaosconv} immediately.
\qed

\begin{rem}
\label{rem-chaos}
{\rm (i) In the literature, typically one considers the system \reff{chaosNFBSDE} with i.i.d. initial conditions, rather than the system \reff{NashFBSDE2} with deterministic initial conditions, for the propagation of chaos. As we saw in the proof of Theorem \ref{thm-chaos},  \reff{NashFBSDE2} can be viewed as a conditional version of \reff{chaosNFBSDE}, conditional on the values of $\vec\xi$. In this sense Theorem \ref{thm-Nash1} is slightly stronger than Theorem \ref{thm-chaos}. Moreover, \reff{NashFBSDE2} provides a pointwise representation for the Nash system \reff{Nash}.

(ii) A more fundamental difference between the two systems is the flow property, which is important when we study the problem dynamically. While \reff{NashFBSDE2} satisfies it, the flow property fails for the system \reff{chaosNFBSDE} in the sense that $\{X^{N, \vec \xi, i}_t\}_{1\le i\le N}$ do not seem to be conditionally i.i.d., conditional on any reasonable $\si$-algebra like $\cF^0_t$, even though the system starts with i.i.d. initial conditions. Consequently, if we insist on i.i.d. setting, our strategy of proving Theorem \ref{thm-Nash1} wouldn't work. That is, if we prove a version of Proposition \ref{prop-Nashlocal1} with i.i.d. initial conditions, we won't be able to apply it on $[t_k, t_{k+1}]$ for $k>0$. 
\qed}
\end{rem}

\section{Pointwise representation  for Wasserstein derivatives}
\label{sect-representation}
\setcounter{equation}{0}

In this section we provide pointwise representation formulas for the Wasserstein derivatives $\pa_\mu V, \pa_{\mu\mu} V$. These formulas are new in the literature, to our best knowledge. Such a representation is helpful for understanding the pointwise properties of the Wasserstein derivatives, for example their regularity under minimum conditions. Since the formulas are quite involved, to ease the presentation, we make the following simplifications.

$\bullet$ We assume all the involved processes are $1$-dimensional. All our results can be extended to the multidimensional cases without any significant difficulty.

 $\bullet$ We assume all the data are sufficiently smooth, and the involved derivatives are bounded and, if needed, Lipschitz continuous in $(x,\mu)$.
 
 $\bullet$ We assume all the involved (McKean-Vlasov) FBSDEs have a unique strong solution and the stability result holds true. In particular, this is true when $T$ is small.
 
$\bullet$ We restrict to first and second order derivatives only. But all the higher order derivatives can be expressed in the same manner. 

$\bullet$ We assume $F=0$. The $F$ term can be treated in exactly the same way as the $G$-term. 
 
  $\bullet$ We shall provide the formulas at $(0, x,\mu)$ only. 

\no Throughout this section, the above assumptions are always in force. We fix $\xi\in \dbL^2(\cF_0, \mu)$, and let $(X^\xi, Y^\xi, Z^\xi, Z^{0,\xi})$, $(X^{x}, Y^{x,\xi}, Z^{x,\xi}, Z^{0,x,\xi})$, $(X^{\xi,x}, Y^{\xi,x}, Z^{\xi,x}, Z^{0,\xi,x})$, $\rho$ be as in \reff{FBSDE1}, \reff{FBSDE2}, and \reff{FBSDE3} with $t_0=0$. We shall provide  pointwise representation formulas for the derivatives. We remark that our analysis here provides an alternative approach for the existence of classical solutions to the master equation \reff{master}, provided the data are smooth and the involved FBSDEs are wellposed, which is true when $T$ is small or when the monotonicity condition \reff{mon} holds.  However, since the classical solution theory is already established in \cite{CDLL, CD2}, in this section we will only focus on the derivation of the representation formulas, without providing the precise conditions for the wellposedness of the involved FBSDEs.

To facilitate the representation formulas, we  introduce the following operators:  for functions $\Phi(x, z)$ and $\Psi(x, \mu)$,
\bea
\label{tdH}
\left.\ba{c}
\td_x \Phi(x, z; x_1, z_1) := \pa_x \Phi(x,z) x_1 +\pa_p \Phi(x,z) z_1 ,\\
\td_{xx} \Phi(x, z; x_1, z_1; x_2, z_2; x_3,z_3) :=\td_x \Phi(x, z; x_3, z_3)\\
+\td_x (\pa_x \Phi) (x, z; x_1, z_1) x_2 + \td_x (\pa_p \Phi) (x, z; x_1, z_1) z_2;\\
\td_\mu \Psi(x, \mu, \tilde x, \tilde x'; x_1,x_1'):=  \pa_\mu \Psi(x,\mu, \tilde x) x_1 + \pa_\mu \Psi(x,\mu, \tilde x') x_1';\\
\ea\right.
\eea

We first state the representations for $\pa_x V, \pa_{xx} V$ without proof, which have already been used in the previous sections.

\begin{prop}
\label{prop-paxV}
It holds that
\bea
\label{paxV1}
\pa_x V(0,x,\mu) = \td_x Y^{x,\xi}_0,\q \pa_{xx} V(0,x,\mu) = \td_{xx} Y^{x,\xi}_0,
\eea
where, recalling that $\rho_t = \rho^\xi_t$ depends on $\xi$,
\bea
\label{tdxY}
\left.\ba{lll}
\dis \td_x  Y_t^{x,\xi}=\pa_x G(X_T^{x},\rho_T) -\int_t^T\td_x Z_s^{x,\xi}dB_s-\int_t^T\td_x Z_s^{0,x,\xi}dB_s^{0}\\
\dis \qq\qq +\int_t^T\Big[\pa_x H(X_s^x,Z_s^{x,\xi})+ \pa_p H(X_s^x,Z_s^{x,\xi}) \td_x Z_s^{x,\xi}\Big]ds;\\
\dis \td_{xx}  Y_t^{x,\xi}= \pa_{xx} G(X_T^{x},\rho_T)-\int_t^T\td_{xx}Z_s^{x,\xi}dB_s-\int_t^T\td_{xx} Z_s^{0,x,\xi}dB_s^{0}\\
\dis\qq\qq +\int_t^T\Big[\pa_{xx} H(X_s^x,Z_s^{x,\xi}) + 2\pa_{xp} H(X_s^x,Z_s^{x,\xi}) \td_x Z_s^{x,\xi} \\
\dis\qq\qq+\pa_{pp} H(X_s^x,Z_s^{x,\xi}) |\td_x Z_s^{x,\xi}|^2 +\pa_p H(X_s^x,Z_s^{x,\xi}) \td_{xx} Z_s^{x,\xi} \Big]ds.
\ea\right.
\eea

Alternatively, in light of \reff{FBSDE3} we also have
\bea
\label{paxV2}
\pa_x V(0,x,\mu) =\td_x Y^{\xi,x}_0,\q \pa_{xx} V(0,x,\mu) =\td_{xx} Y^{\xi,x}_0,
\eea
where, for $\Xi=(X, Z)$,
\bea
\label{tdYx}
\left\{\ba{lll}
\dis \td_x X_t^{\xi,x}=1 +\int_0^t \td_x (\pa_p H)\big(\Xi^{\xi,x}_s; \td_x \Xi^{\xi,x}_s\big) ds; \\
\dis \td_x Y_t^{\xi,x}= \pa_x G(X^{\xi,x}_T,\rho_T) \td_x X^{\xi,x}_T- \int_t^T \td_x \h L\big(\Xi^{\xi,x}_s; \td_x \Xi^{\xi,x}_s\big)ds\\
\dis \q -\int_t^T\td_x Z_s^{\xi,x}dB_s- \int_t^T \td_x Z_s^{0,\xi,x}dB_s^{0}; 
\ea\right.
\eea
\bea
\label{tdYxx}
\left\{\ba{lll}
\dis \td_{xx} X_t^{\xi,x}=\int_0^t \td_{xx} (\pa_p H)\big(\Xi^{\xi,x}_s; \td_x \Xi^{\xi,x}_s; \td_x \Xi^{\xi,x}_s; \td_{xx} \Xi^{\xi,x}_s\big)ds; \\
\dis \td_{xx} Y_t^{\xi,x}=\pa_{x} G(X^{\xi,x}_T,\rho_T)  \td_{xx} X^{\xi,x}_T + \pa_{xx} G(X^{\xi,x}_T,\rho_T) |\td_x X^{\xi,x}_T|^2 \\
\dis \q - \int_t^T \td_{xx} \h L\big(\Xi^{\xi,x}_s; \td_x \Xi^{\xi,x}_s; \td_{x} \Xi^{\xi,x}_s; \td_{xx} \Xi^{\xi,x}_s\big)ds \\
\dis\q-\int_t^T\td_{xx} Z_s^{\xi,x}dB_s- \int_t^T \td_{xx} Z_s^{0,\xi,x}dB_s^{0}.
\ea\right.
\eea
\end{prop}

To prepare for the representations of the other derivatives, we introduce the following systems of McKean-Valsov FBSDEs: again for $\Xi=(X, Z)$,
\bea
\label{tdYx-}
\left\{\ba{lll}
\dis \td_x X_t^{\xi,x-}= \int_0^t \td_x (\pa_p H)\big(\Xi^{\xi}_s; \td_x \Xi^{\xi,x-}_s\big) ds; \\
\dis \td_x Y_t^{\xi,x-}= \pa_x G(X^{\xi}_T,\rho_T)  \td_x X^{\xi,x-}_T  + \tilde \dbE\Big[ \td_\mu G\big(X_T^{\xi},\rho_T, \tilde X^{\xi,x}_T, \tilde X^{\xi}_T;  \td_x \tilde X^{\xi,x}_T,\td_x \tilde X^{\xi,x-}_T \big) \Big]\\
\dis\q  - \int_t^T  \td_x \h L\big(\Xi^{\xi}_s; \td_x \Xi^{\xi,x-}_s\big)ds -\int_t^T\td_x Z_s^{\xi,x-}dB_s- \int_t^T \td_x Z_s^{0,\xi,x-}dB_s^{0};
\ea\right.
\eea
\bea
\label{tdYxx-}
\left\{\ba{lll}
\dis \td_{xx} X_t^{\xi,x-}= \int_0^t \td_x (\pa_p H)\big(\Xi^{\xi}_s; \td_{xx} \Xi^{\xi,x-}_s\big) ds; \\
\dis \td_{xx} Y_t^{\xi,x-}= \pa_x G(X^{\xi}_T,\rho_T)  \td_{xx} X^{\xi,x-}_T   + \tilde \dbE\Big[ \pa_{\tilde x\mu} G(X_T^{\xi},\rho_T, \tilde X^{\xi,x}_T) |\td_x \tilde X^{\xi,x}_T|^2\\
\dis\q  +\td_\mu G\big(X_T^{\xi},\rho_T, \tilde X^{\xi,x}_T,  \tilde X^{\xi}_T; \td_{xx} \tilde X^{\xi,x}_T,  \td_{xx} \tilde X^{\xi,x-}_T\big)   \Big]\\
\dis\q -\int_t^T \td_x \h L\big(\Xi^{\xi}_s; \td_{xx}\Xi^{\xi,x-}_s\big)ds -\int_t^T\td_{xx} Z_s^{\xi,x-}dB_s- \int_t^T \td_{xx} Z_s^{0,\xi,x-}dB_s^{0}.
\ea\right.\qq\q
\eea

\subsection{Representation of $\pa_\mu V$}
Recall that, for fixed $\mu$, $\pa_\mu V(t,x,\mu, \tilde x)$ is well defined only for $\mu$-a.e. $\tilde x$. However, when $\pa_\mu V$ is continuous in all variables, it  is unique for all $\tilde x\in \dbR^d$, see Remark \ref{rem-compact}. In this subsection we provide a representation formula for this continuous version.  
\begin{thm}
\label{thm-pamuV}
It holds that
\bea
\label{pamuV}
\pa_\mu V(0,x,\mu, \tilde x)= \td_\mu Y^{x, \xi, \tilde x}_0,
\eea
where,  by using $\tilde \cd$  to denote conditionally independent copies as in \reff{Ito},   
\bea
\label{tdmuYx}
\left.\ba{c}
 \td_\mu Y^{x, \xi, \tilde x}_t  = \tilde \dbE_{\cF^0_T}\Big[\td_\mu G\big(X_T^x,\rho_T, \tilde X^{\xi, \tilde x}_T,  \tilde X^\xi_T;  \td_x \tilde X^{\xi,\tilde x}_T, \td_x \tilde X^{\xi, \tilde x-}_T\big)\Big]
 \\
\dis +\int_t^T \pa_p H(X_s^{x},Z_s^{x,\xi}) \td_\mu Z^{x,\xi,\tilde x}_sds -\int_t^T\td_\mu Z_s^{x,\xi, \tilde x}dB_s-\int_t^T\td_\mu Z_s^{0,x,\xi, \tilde x}dB_s^0.
\ea\right.
 \eea
 \end{thm}
\proof We proceed in four steps. 

{\it Step 1.} For any $\xi\in \dbL^2(\cF_0, \mu)$ and $\eta\in \dbL^2(\cF_0)$, following standard arguments and by our assumption of the stability property of the involved systems we have
\bea
\label{tdXconv}
 \lim_{\e\to 0} \dbE\Big[\sup_{0\le t\le T} \big| {X^{\xi + \e \eta}_t - X^\xi_t\over \e} - \td X^{\xi, \eta}_t\big|^2\Big] = 0,
\eea
where $(\td X^{\xi, \eta}, \td Y^{\xi, \eta}, \td Z^{\xi, \eta},\td Z^{0,\xi, \eta})$ satisfies the   linear McKean-Vlasov FBSDE:
\bea
\label{tdmuFBSDE}
\left.\ba{lll}
\dis  \td X^{\xi, \eta}_t=\eta + \int_0^t\td_x (\pa_p H)\big(\Xi^{\xi}_s; \td \Xi_s^{\xi,\eta}\big)ds, \\
\dis \td Y_t^{\xi, \eta}= \pa_x G(X_T^{\xi},\rho_T) \td X^{\xi, \eta}_T + \tilde \dbE_{\cF^0_T}\big[\pa_\mu G(X_T^{\xi},\rho_T, \tilde X^\xi_T)\td \tilde X^{\xi, \eta}_T\big]\\
\dis\qq - \int_t^T  \td_x \h L\big(\Xi^{\xi}_s; \td \Xi^{\xi,\eta}_s\big)ds - \int_t^T\td Z_s^{\xi, \eta}dB_s- \int_t^T\td Z_s^{0,\xi,\eta}d B_s^0.
\ea\right.
\eea
Similarly, by \reff{tdXconv} and \reff{FBSDE2}, one can show that
\bea
\label{tdYconv}
 \lim_{\e\to 0} \dbE\Big[\sup_{0\le t\le T} \big| {Y^{x, \xi + \e \eta}_t - Y^{x,\xi}_t\over \e} - \td Y^{x, \xi, \eta}_t\big|^2\Big] = 0,
\eea
where $(\td Y^{x,\xi, \eta}, \td Z^{x, \xi, \eta}, \td Z^{0,x, \xi, \eta})$ satisfies the linear (standard) BSDE:
\bea\label{eq:nabla mu X}
\left.\ba{c}
\dis \td Y_t^{x, \xi, \eta}=   \tilde \dbE_{\cF^0_T}\big[\pa_\mu G(X_T^x,\rho_T, \tilde X^\xi_T) \td \tilde X^{\xi, \eta}_T\big] +\int_t^T \pa_p H(X_s^{x},Z_s^{x,\xi}) \td Z^{x,\xi,\eta}_sds \\
\dis -\int_t^T\td Z_s^{x,\xi, \eta}dB_s-\int_t^T\td Z_s^{0,x,\xi, \eta}dB_s^0. 
\ea\right.
\eea
In particular, \reff{tdYconv} implies,
\bea
\label{pamuVconv}
 \lim_{\e\to 0}  \big| {V(0,x, \cL_{\xi + \e \eta}) - V(0,x,\cL_\xi)\over \e} - \td Y^{x, \xi, \eta}_0\big|^2\Big] = 0.
\eea
Thus, by the definition of $\pa_\mu V$,
\bea
\label{pamuV1}
\dbE\big[\pa_\mu V(0,x, \mu, \xi) \eta\big] =  \td Y^{x, \xi, \eta}_0.
\eea

{\it Step 2.} In this step we assume $\xi$ (or say, $\mu$) is discrete: $p_i = \dbP(\xi = x_i)$, $i=1,\cds, n$. Fix $i$, consider the following system of McKean-Vlasov FBSDEs: for $j=1,\cds, n$,
\bea
\label{tdmuYij}
\left.\ba{lll}
\dis  \td X^{i,j}_t= \1_{\{i=j\}} + \int_0^t \td_x (\pa_p H)\big(\Xi^{\xi,x_j}_s; \td \Xi_s^{i,j}\big)ds, \\
\dis \td Y_t^{i,j}= \pa_x G(X_T^{\xi,x_j},\rho_T) \td X^{i,j}_T+\sum_{k=1}^n p_k \tilde \dbE_{\cF^0_T}\Big[\pa_\mu G(X_T^{\xi,x_j},\rho_T, \tilde X^{\xi,x_k}_T) \td \tilde X^{i,k}_T\Big] \\
\dis\qq - \int_t^T  \td_x \h L\big(\Xi^{\xi,x_j}_s; \td \Xi^{i,j}_s\big)ds - \int_t^T\td Z_s^{i,j}dB_s-\b \int_t^T\td Z_s^{0,i,j}d B_s^0.
\ea\right.
\eea
Denote, for $\Phi=X, Y, Z, Z^0$,
\beaa
\td \Phi^{\xi, x_i} := \td \Phi^{i,i},\q  \td \Phi^{\xi, x_i-}:= {1\over p_i} \sum_{j\neq i} \td \Phi^{i,j} \1_{\{\xi=x_j\}}.
\eeaa
 Note that $\Phi^\xi =  \sum_{j=1}^n \Phi^{\xi, x_j} \1_{\{\xi=x_j\}}$. Since \reff{tdmuYij} is linear, one can easily check that
\bea
\label{tdXxi-}
\left.\ba{lll}
\dis  \td X^{\xi, x_i}_t=1+ \int_0^t\td_x (\pa_p H)\big(\Xi^{\xi, x_i}_s; \td \Xi_s^{\xi,x_i}\big)ds;\\
\dis  \td X^{\xi, x_i-}_t= \int_0^t\td_x (\pa_p H)\big(\Xi^{\xi}_s; \td \Xi_s^{\xi,x_i-}\big)ds;\\
\dis \td Y_t^{\xi, x_i}= \pa_x G(X_T^{\xi,x_i},\rho_T) \td X^{\xi, x_i}_T - \int_t^T\td Z_s^{\xi, x_i}dB_s- \int_t^T\td Z_s^{0,\xi,x_i}d B_s^0\\
\dis \q - \int_t^T  \td_x \h L\big(\Xi^{\xi,x_i}_s; \td \Xi^{\xi,x_i}_s\big)ds + p_i\tilde \dbE_{\cF^0_T}\big[\td_\mu G(X_T^{\xi},\rho_T, \tilde X^{\xi,x_i}_T, \tilde X^\xi_T; \td\tilde X_s^{\xi,x_i}, \td \tilde X^{\xi, x_i-}_T)\big];\\
\dis \td Y_t^{\xi, x_i-}= \pa_x G(X_T^{\xi},\rho_T) \td X^{\xi, x_i-}_T   - \int_t^T  \td_x \h L\big(\Xi^{\xi}_s; \td \Xi^{\xi,x_i-}_s\big)ds \\
\dis\q - \int_t^T\td Z_s^{\xi, x_i-}dB_s- \int_t^T\td Z_s^{0,\xi,x_i-}d B_s^0\\
\dis\q+ \tilde \dbE_{\cF^0_T}\big[\td_\mu G(X_T^{\xi},\rho_T, \tilde X^{\xi,x_i}_T, \tilde X^\xi_T; \td\tilde X_s^{\xi,x_i}, \td \tilde X^{\xi, x_i-}_T)\big]\1_{\{\xi\neq x_i\}};
\ea\right.
\eea
Moreover, note that
\beaa
\left.\ba{c}
\dis  \tilde \dbE_{\cF^0_T}\big[\pa_\mu G(X_T^\xi,\rho_T, \tilde X^\xi_T)\td \tilde X^{\xi, \1_{\{\xi=x_i\}}}_T\big]\ms\\
\dis =   \tilde \dbE_{\cF^0_T}\Big[\pa_\mu G(X_T^\xi,\rho_T, \tilde X^\xi_T)\td \tilde X^{\xi, \1_{\{\xi=x_i\}}}_T\big[\1_{\{\tilde \xi=x_i\}} + \1_{\{\tilde \xi\neq x_i\}}\big]\Big] \ms\\
\dis= p_i \tilde \dbE_{\cF^0_T}\Big[\pa_\mu G(X_T^{\xi},\rho_T, \tilde X^{\xi, x_i}_T) \td \tilde X^{\xi, \1_{\{\xi=x_i\}}}_T\Big] +  \tilde \dbE_{\cF^0_T}\Big[\pa_\mu G(X_T^\xi,\rho_T, \tilde X^\xi_T)\td \tilde X^{\xi, \1_{\{\xi=x_i\}}}_T \1_{\{\tilde \xi\neq x_i\}}\Big].
\ea\right.
\eeaa
Since \reff{tdmuFBSDE} is also linear, one can easily check that, for $\Phi=X, Y, Z, Z^0$,
\bea
\label{tdXdecom}
\td  \Phi^{\xi, \1_{\{\xi=x_i\}}} = \td \Phi^{\xi, x_i} \1_{\{\xi=x_i\}}+ p_i \td   \Phi^{\xi, x_i-}.
\eea
Plug this into \reff{eq:nabla mu X}, we obtain 
\bea
\label{tdYxipi}
\td \Phi_t^{x, \xi, \1_{\{\xi=x_i\}}}=  p_i \td \Phi^{x,\xi, x_i}_t,
\eea
 where
\bea
\label{eq:nabla mu X2}
\left.\ba{c}
\dis  \td Y^{x,\xi, x_i}_t= \tilde \dbE_{\cF^0_T}\big[\td_\mu G(X_T^x,\rho_T, \tilde X^{\xi,x_i}_T, \tilde X^\xi_T;  \td \tilde X^{\xi, x_i}_T, \td \tilde X^{\xi, x_i-}_T)\big]  \\
\dis +\int_t^T \pa_p H(X_s^{x},Z_s^{x,\xi}) \td Z^{x,\xi,x_i}_sds-\int_t^T\td Z_s^{x,\xi, x_i}dB_s-\int_t^T\td Z_s^{0,x,\xi, x_i}dB_s^0. 
\ea\right.
\eea
In particular, by setting $\eta = \1_{\{\xi=x_i\}}$ in \reff{pamuV1} we obtain:
\bea
\label{pamuVdiscrete}
\pa_\mu V(0, x, \mu, x_i) = \td Y^{x,\xi, x_i}_0.
\eea
We shall note that \reff{tdXxi-} is different from \reff{tdYx} and \reff{tdYx-}, so \reff{pamuVdiscrete} provides an alternative representation in the discrete case. 

{\it Step 3.} We now prove \reff{pamuV} in the case that $\mu$ is continuous. For each $n\ge 1$, let $x^n_i := \frac{i}{2^n}$, $i=-n2^{n},\cds, n2^{n}$, and 
\bea
\label{xin}
\xi_n :=\sum_{i=- n2^{n}}^{n2^{n}-1} x_i^n\1_{[x_{i}^n,x_{i+1}^n)}(\xi)-n\1_{(-\infty,-n)}(\xi)+n\1_{[n,+\infty)}. 
\eea
 It is clear that $\lim_{n\to+\infty}\mathbb E|\xi_n - \xi|^2=0$ and thus $\lim_{n\to\infty} \cW_2(\cL_{\xi_n}, \cL_\xi)=0$. Then for any $\eta$, by stability of FBSDE \reff{tdmuFBSDE} and  BSDE \reff{eq:nabla mu X}, we derive from \reff{pamuV1} that
 \bea
 \label{tdmuYconv}
 \dbE\Big[\pa_\mu V(0,x,\mu, \xi) \eta\Big] = \td Y^{x,\xi, \eta}_0 = \lim_{n\to\infty} \td Y^{x,\xi_n, \eta}_0.
  \eea

 For each $\tilde x\in \dbR$, let $i_n(\tilde x)$ be the $i$ such that $\tilde x \in [x_{i_n(\tilde x)}^n, x_{i_n(\tilde x)+1}^n)$, which is well defined when $n> |\tilde x|$. Then $\dis\lim_{n\to\infty} (\cL_{\xi_n}, x^n_{i_n(\tilde x)}) =(\mu, \tilde x)$.  By the stability of FBSDEs \reff{FBSDE1}-\reff{FBSDE2}, we have $(X^{\xi_n, x^n_{i_n(\tilde x)}}, Z^{\xi_n, x^n_{i_n(\tilde x)}}) \to (X^{\xi, \tilde x}, Z^{\xi, \tilde x})$ under appropriate norm. Moreover, since $\xi$ is continuous, $\dbP(\xi_n = x^n_{i_n(\tilde x)}) = \dbP(\xi \in [x^n_{i_n(\tilde x)}, x^n_{i_n(\tilde x)+1})) \to 0$, as $n\to \infty$. Then by the stability of \reff{tdXxi-} and \reff{eq:nabla mu X2} we can check that
 \bea
 \label{xinconv}
 \lim_{n\to\infty}\Big(\td \Phi^{\xi_n,x^n_{i_n(\tilde x)}}, \td \Phi^{\xi_n,x^n_{i_n(\tilde x)-}}, \td\Phi^{x,\xi_n,x^n_{i_n(\tilde x)}}\Big) =\Big(\td_x \Phi^{\xi,\tilde x}, \td_x\Phi^{\xi,\tilde x-}, \td_\mu\Phi^{x,\xi,\tilde x}\Big).
  \eea
 Now for any bounded and continuous function $\f: \dbR\to \dbR$, by setting $\eta=\f(\xi)$ in \reff{tdmuYconv}, we derive from  \reff{tdYxipi} that
 \beaa
 &\dis\dbE\Big[\pa_\mu V(0,x,\mu, \xi) \f(\xi)\Big] =  \lim_{n\to\infty} \td Y^{x,\xi_n, \f(\xi_n)}_0 =  \lim_{n\to\infty} \sum_i \f(x^n_i)\td Y^{x,\xi_n, \1_{\{\xi_n=x^n_i\}}}_0\\
&\dis =  \lim_{n\to\infty} \sum_i \f(x^n_i)\td Y^{x,\xi_n, x^n_i}_0 \dbP(\xi\in [x^n_i, x^n_{i+1}))  = \int_\dbR \f(\tilde x)\td Y^{x,\xi, \tilde x}_0 \mu(d\tilde x). 
 \eeaa
 This implies \reff{pamuV} immediately.

 {\it Step 4.} We finally prove the general case. Denote $\psi(x,\mu, \tilde x):= \td_\mu Y^{x,\xi, \tilde x}_0$. By the stability of FBSDEs, $\psi$ is continuous in all the variables. Fix an arbitrary $(\mu, \xi)$. One can easily construct continuous $\xi_n$ such that $\lim_{n\to\infty}\dbE[|\xi_n-\xi|^2]=0$. Then, for any $\eta= \f(\xi)$ as in Step 3, by \reff{pamuV1} and Step 3 we have
 \beaa
&\dis \dbE\big[\pa_\mu V(0,x, \mu, \xi) \f(\xi)\big] =  \lim_{n\to \infty}  \td Y^{x, \xi_n, \f(\xi_n)}_0 \\
&\dis =  \lim_{n\to \infty}\dbE\big[ \psi(x, \cL_{\xi_n}, \xi_n)\f(\xi_n)\big] = \dbE\big[ \psi(x, \mu, \xi)\f(\xi)\big],
\eeaa
which implies \reff{pamuV} in the general case and hence  completes the proof.
\qed

\begin{rem}
\label{rem-pamuVrep}
{\rm (i) By using the linearized system of SPDE \reff{SPDE}, \cite[Corollary 3.9]{CDLL}  provided a pointwise representation formula for the gradient ${\d V\over \d \mu}(t,x,\mu, \tilde x)$.  Note that $\pa_\mu V (t,x,\mu, \tilde x) = \pa_{\tilde x} {\d V\over \d \mu}(t,x,\mu, \tilde x)$, so \cite{CDLL} implies a representation formula for $\pa_\mu V (t,x,\mu, \tilde x)$ as well, by involving an FBSPDE system whose initial value is the derivative of the Dirac measure. Our representation formula \reff{pamuV} involves strong solutions of FBSDEs and holds under weaker technical conditions. We note that, unlike the connection between \reff{SPDE} and \reff{FBSDE1}-\reff{FBSDE2}, the forward PDE in \cite{CDLL} does not represent the density of the forward SDEs in \reff{FBSDE1}, so the connection between \reff{pamuV} and their representation formula is not clear to us. 

(ii) Rigorously speaking the derivative $\pa_\mu V$ is defined through Fr\'{e}chet derivative, see \reff{pamu}. Since our focus here is the representation formula, we content ourselves with using the G\^{a}teux derivative in \reff{pamuVconv}, which is slightly easier.  However, we can easily extend our arguments to the Fr\'{e}chet derivative, then our arguments indeed lead to the classical solutions of the master equations, provided that the involved FBSDEs are wellposed.
\qed}\end{rem}

\subsection{Representation of the second order derivatives}
First, based on Theorem \ref{thm-pamuV}, we have the following representations immediately. 
\begin{prop}
\label{prop-paxmuV}
 It holds that
\bea
\label{paxmuV}
\dis \pa_{x\mu} V(0,x,\mu, \tilde x)= \td_{x\mu} Y^{x, \xi, \tilde x}_0,\q\pa_{\tilde x\mu} V(0,x,\mu, \tilde x)= \td_{\tilde x\mu} Y^{x, \xi, \tilde x}_0,
 \eea
where, recalling \reff{tdYx} and \reff{tdYx-} again, 
\bea
\label{tdxtildexmuY}
\left.\ba{lll}
 \td_{x\mu} Y^{x, \xi, \tilde x}_t  = \tilde \dbE_{\cF^0_T}\Big[\td_{\mu} (\pa_x G)\big(X_T^x,\rho_T, \tilde X^{\xi, \tilde x}_T, \tilde X^\xi_T; \td_x \tilde X^{\xi,\tilde x}_T,\td_x \tilde X^{\xi, \tilde x-}_T\big)\Big]
 \\
\dis \q+\int_t^T \Big[ [\pa_{xp} H(X_s^{x},Z_s^{x,\xi})  + \pa_{pp} H(X_s^{x},Z_s^{x,\xi}) \td_x Z_s^{x,\xi}] \td_{\mu} Z^{x,\xi,\tilde x}_s\\
\dis \q +\pa_p H(X_s^{x},Z_s^{x,\xi}) \td_{x\mu} Z^{x,\xi,\tilde x}_s \Big]ds  -\int_t^T\td_{x\mu} Z_s^{x,\xi, \tilde x}dB_s-\int_t^T\td_{x\mu} Z_s^{0,x,\xi, \tilde x}dB_s^0;\\
\dis \td_{\tilde x\mu} Y^{x, \xi, \tilde x}_t  = \tilde \dbE_{\cF^0_T}\Big[\pa_{\tilde x\mu} G(X_T^x,\rho_T, \tilde X^{\xi, \tilde x}_T) |\td_{x} \tilde X^{\xi, \tilde x}_T|^2\\
 \dis \q+\td_{\mu} G\big(X_T^x,\rho_T, \tilde X^{\xi, \tilde x}_T,  \tilde X^\xi_T; \td_{xx} \tilde X^{\xi,\tilde x}_T, \td_{xx} \tilde X^{\xi, \tilde x-}_T\big)\Big]
 \\
\dis \q+\int_t^T \pa_p H(X_s^{x},Z_s^{x,\xi}) \td_{\tilde x\mu} Z^{x,\xi,\tilde x}_sds -\int_t^T\td_{\tilde x\mu} Z_s^{x,\xi, \tilde x}dB_s-\int_t^T\td_{\tilde x\mu} Z_s^{0,x,\xi, \tilde x}dB_s^0.
\ea\right.
 \eea
 \end{prop}

Note that $\pa_{\mu x} V=\pa_{x\mu}V$ when the derivatives are continuous, so the above provides a representation for $\pa_{\mu x}V$ as well.  However, we remark that $\pa_{\mu \tilde x} V(0, x, \mu, \tilde x)$ is not meaningful because $\tilde x$ is not a variable of $V$ itself.

We finally investigate $\pa_{\mu\mu} V(0,x, \mu, \tilde x, \bar x)$, which is unfortunately very involved.  Introduce the following function: for any $\h x\in \dbR$ and  random variables $X_1, X_2, X_3, X_4$, 
\bea
\label{I9.28} 
\left.\ba{lll}
\dis I^{\xi,x}(\h x; X_1, X_2, X_3, X_4):= \tilde \dbE_{\cF^0_T}\Big[ \pa_\mu G\big(\h x,\rho_T, \tilde X^{\xi,x}_T) \tilde X_1  + \pa_\mu G\big(\h x,\rho_T, \tilde X^{\xi}_T)\tilde X_2 \\
\dis\q   + \pa_{\tilde x\mu} G\big(\h x,\rho_T, \tilde X^{\xi,x}_T) \td_x \tilde X^{\xi,x}_T \tilde X_3+ \pa_{\tilde x\mu} G\big(\h x,\rho_T, \tilde X^{\xi}_T) \td_x \tilde X^{\xi,x-}_T\tilde X_4 \\
\dis \q + \bar\dbE_{\cF^0_T}\big[ [\pa_{\mu\mu} G\big(\h x,\rho_T, \tilde X^{\xi,x}_T, \bar X^\xi_T) \td_x \tilde X^{\xi,x}_T  + \pa_{\mu\mu} G\big(\h x,\rho_T, \tilde X^{\xi}_T, \bar X^\xi_T) \td_x \tilde X^{\xi,x-}_T] \bar X_4 \big]\Big].
\ea\right.
\eea
where, as usual, $\tilde X_i$, $\bar X_i$ denote the conditionally independent copy of $X_i$, conditional on $\cF^0_T$. Consider the following systems of McKean-Vlasov FBSDEs: for $\Xi = (X, Z)$,
\bea
\label{tdmuYxx'}
\left\{\ba{lll}
\dis \td_\mu X_t^{\xi,x,x'-}=\int_0^t \td_x (\pa_p H)\big(\Xi^{\xi,x}_s; \td_\mu \Xi^{\xi,x, x'-}_s\big) ds; \\
\dis \td_\mu Y_t^{\xi,x,x'-}= \pa_x G(X^{\xi,x}_T,\rho_T) \td_\mu X^{\xi,x, x'-}_T - \int_t^T \td_x \h L\big(\Xi^{\xi,x}_s; \td_\mu \Xi^{\xi,x,x'-}_s\big)ds\\
\dis \q+ \tilde \dbE_{\cF^0_T}\big[\td_\mu G(X^{\xi, x}_T, \rho_T, \tilde X^{\xi,x_i}_T, \tilde X^\xi_T; \td_x \tilde X^{\xi, x'}_T,\td_x \tilde X^{\xi, x'-}_T\big) \big]\\
\dis \q -\int_t^T\td Z_s^{\xi,x,x'-}dB_s- \int_t^T \td Z_s^{0,\xi,x, x'-}dB_s^{0}; 
\ea\right.\qq~
\eea
\bea
\label{tdmuxY}
\left\{\ba{lll}
\dis \td_{\mu x} X_t^{\xi,x,x'-}=\int_0^t \td_{xx} (\pa_p H)\big(\Xi^{\xi,x}_s; \td_\mu \Xi^{\xi,x,x'-}_s;  \td_x \Xi^{\xi,x}_s; \td_{\mu x} X_s^{\xi,x,x'-} \big) ds; \\
\dis \td_{\mu x} Y_t^{\xi,x, x'-}= \pa_x G(X^{\xi,x}_T, \rho_T) \td_{\mu x} X_s^{\xi,x,x'-} + \pa_{xx} G(X^{\xi,x}_T, \rho_T) \td_x X^{\xi,x}_T\td_\mu X_T^{\xi,x,x'-} \\
\dis\q +\tilde \dbE_{\cF^0_T}\Big[\td_\mu(\pa_x G)(X^{\xi,x}_T, \rho_T, \tilde X^{\xi,x'}_T, \tilde X^{\xi}_T; \td_x \tilde X^{\xi,x'}_T ,\td_x \tilde X^{\xi,x'-}_T)\Big] \td_x X^{\xi,x}_T\\
\dis\q  - \int_t^T \td_{xx} \h L\big(\Xi^{\xi,x}_s; \td_\mu \Xi^{\xi,x,x'-}_s;  \td_x \Xi^{\xi,x}_s; \td_{\mu x} X_s^{\xi,x,x'-}\big)ds\\
\dis \q -\int_t^T\td_{\mu x} Z_s^{\xi,x,x'}dB_s- \int_t^T \td_{\mu x} Z_s^{0,\xi,x,x'}dB_s^{0}; 
\ea\right.
\eea
\bea
\label{tdmuxYx-}
\left\{\ba{lll}
\dis \td_{\mu x} X_t^{\xi,x-, x'}= \int_0^t \td_x (\pa_p H)\big(\Xi^{\xi}_s; \td_x \Xi^{\xi,x'}_s; \td_x \Xi^{\xi,x-}_s; \td_{\mu x} \Xi^{\xi,x-, x'}_s\big) ds; \\
\dis \td_{\mu x} X_t^{\xi,x-, x'-}= \int_0^t \td_x (\pa_p H)\big(\Xi^{\xi}_s; \td_x \Xi^{\xi,x'-}_s; \td_x \Xi^{\xi,x-}_s; \td_{\mu x} \Xi^{\xi,x-, x'-}_s\big) ds; \\
\dis \td_{\mu x} Y_t^{\xi,x-,x'}= \pa_x G(X^{\xi}_T, \rho_T) \td_{\mu x} X_s^{\xi,x-,x'} + \pa_{xx} G(X^{\xi}_T, \rho_T) \td_x X^{\xi,x-}_T \td_x X_T^{\xi,x'} \\
\dis\q +  \tilde \dbE_{\cF^0_T}\Big[\td_\mu(\pa_x G)\big(X_T^{\xi},\rho_T, \tilde X^{\xi,x}_T, \tilde X^{\xi}_T; \td_x \tilde X^{\xi,x}_T , \td_x \tilde X^{\xi,x-}_T\big)\Big]  \td_x X^{\xi, x'}_T \\
\dis\q  - \int_t^T  \td_x \h L\big(\Xi^{\xi}_s; \td_x \Xi^{\xi,x'}_s; \td_x \Xi^{\xi,x-}_s; \td_{\mu x} \Xi^{\xi,x-, x'}_s\big)ds \\
\dis\q-\int_t^T\td_{\mu x} Z_s^{\xi,x-, x'}dB_s- \int_t^T \td_{\mu x} Z_s^{0,\xi,x-,x'}dB_s^{0};\\
\dis \td_{\mu x} Y_t^{\xi,x-,x'-}= \pa_x G(X^{\xi}_T, \rho_T) \td_{\mu x} X_s^{\xi,x-,x'-} + \pa_{xx} G(X^{\xi}_T, \rho_T) \td_x X^{\xi,x-}_T \td_x X_T^{\xi,x'-} \\
\dis\q +  \tilde \dbE_{\cF^0_T}\Big[\td_\mu(\pa_x G)\big(X_T^{\xi},\rho_T, \tilde X^{\xi,x}_T, \tilde X^{\xi}_T; \td_x \tilde X^{\xi,x}_T , \td_x \tilde X^{\xi,x-}_T\big)\Big]  \td_x X^{\xi, x'-}_T \\
\dis\q + \tilde \dbE_{\cF^0_T}\Big[\td_\mu(\pa_x G)\big(X^{\xi}_T, \rho_T, \tilde X^{\xi,x'}_T, \tilde X^{\xi}_T; \td_x \tilde X^{\xi,x'}_T, \td_x \tilde X^{\xi,x'-}_T\big)\Big]  \td_x X^{\xi,x-}_T\\
\dis\q + I^{\xi,x} \Big(X^\xi_T; \td_{\mu x}X^{\xi,x,x'-}_T, \td_{\mu x}  X^{\xi,x-,x'}_T +\td_{\mu x}  X^{\xi,x-,x'-}_T,\\
\dis\qq\qq \td_\mu  X^{\xi,x,x'-}_T, \td_x X^{\xi,x'}_T + \td_x X^{\xi, x'-}\Big)\\
\dis\q  - \int_t^T  \td_x \h L\big(\Xi^{\xi}_s; \td_x \Xi^{\xi,x'-}_s; \td_x \Xi^{\xi,x-}_s; \td_{\mu x} \Xi^{\xi,x-, x'-}_s\big)ds \\
\dis\q -\int_t^T\td_{\mu x} Z_s^{\xi,x-, x'-}dB_s- \int_t^T \td_{\mu x} Z_s^{0,\xi,x-,x'-}dB_s^{0}.
\ea\right.
\eea

\begin{thm}
\label{thm-pamumuV}
It holds that
\bea
\label{pamumuV}
\dis \pa_{\mu\mu} V(0,x,\mu, \tilde x,\bar x)= \td_{\mu\mu} Y^{x, \xi, \tilde x,\bar x}_0,
 \eea
 where
\bea
\label{tdmumuYx}
\left.\ba{c}
\dis \td_{\mu\mu} Y^{x, \xi, \tilde x,\bar x}_t  = I^{\xi, \tilde x}\Big(X^x_T;   \td_{\mu x}  X^{\xi,\tilde x,\bar x-} _T, \td_{\mu x}  X^{\xi,\tilde x-, \bar x} _T+\td_{\mu x}  X^{\xi,\tilde x-, \bar x-} _T, \\
\dis\qq\qq  \td_{\mu}  X^{\xi,\tilde x,\bar x-} _T,  \td_x X^{\xi,\bar x}_T + \td_x X^{\xi, \bar x-}\Big)\\
\dis +\int_t^T \Big[\pa_p H(X_s^{x},Z_s^{\xi,x}) \td_{\mu\mu} Z^{x,\xi,\tilde x,\bar x}_s + \pa_{pp} H(X_s^{x},Z_s^{\xi,x}) \td_{\mu} Z^{x,\xi,\tilde x}_s\td_{\mu} Z^{x,\xi,\bar x}_s\Big]ds\\
\dis -\int_t^T\td_{\mu\mu} Z_s^{x,\xi, \tilde x,\bar x}dB_s-\int_t^T\td_{\mu\mu} Z_s^{0,x,\xi, \tilde x,\bar x}dB_s^0.
\ea\right.
 \eea
\end{thm}
\proof  We shall differentiate \reff{pamuV} with respect to $\mu$. Note that the right side of \reff{tdmuYx} involves the following terms related to $\xi$: $X^\xi$, $X^{\xi,x}$, $\td_x X^{\xi, x}$, $\td_x X^{\xi, x-}$, $\rho_T = \rho_T^\xi$. The idea of Theorem \ref{thm-pamuV} is as follows. Denote $\td X^{\xi,\eta}:= \lim_{\e\to 0} {1\over \e}[X^{\xi+\e\eta} - X^\xi]$. When $\xi$ is discrete, we have $\td X^{\xi, \1_{\{\xi=x_i\}}} = \td X^{\xi, x_i} \1_{\{\xi=x_i\}} + p_i \td X^{\xi, x-}$, where $(\td X^{\xi, x_i}, \td X^{\xi, x_i-})$ satisfies \reff{tdXxi-}. When $\xi$ is continuous and approximated by discrete $\xi_n$, we have $ (\td X^{\xi_n, x^n_{i_n(x)}}, \td X^{\xi_n, x^n_{i_n(x)}-})$ converges to $(\td_x X^{\xi, x}, \td_x X^{\xi, x-})$. We shall apply the same arguments on the other terms involving $\xi$. Since the calculation is lengthy but quite straightforward, we shall skip the details and only report the results. Let $\Phi = X, Y, Z, Z^0$ and $\Xi = (X, Z)$ as usual.

(i) Recall \reff{FBSDE3} and denote $\td \Phi^{\xi,x, \eta}:= \lim_{\e\to 0} {1\over \e}[\Phi^{\xi+\e\eta, x} - \Phi^{\xi,x}]$. When $\xi$ is discrete as in Theorem \ref{thm-pamuV} Step 2, we have $\td \Phi^{\xi,x, \1_{\{\xi=x_i\}}} = \td \Phi^{\xi,x, x_i} \1_{\{\xi=x_i\}} + p_i \td \Phi^{\xi, x, x_i-}$, where:
\beaa
\dis\left.\ba{lll}
\dis \td X_t^{\xi,x,x_i}=\int_0^t \td_x (\pa_p H)\big(\Xi^{\xi,x}_s; \td \Xi^{\xi,x, x_i}_s\big) ds; \\
\dis \td X_t^{\xi,x,x_i-}=\int_0^t \td_x (\pa_p H)\big(\Xi^{\xi,x}_s; \td \Xi^{\xi,x, x_i-}_s\big) ds; \\
\dis \td Y_t^{\xi,x,x_i}= \pa_x G(X^{\xi,x}_T,\rho_T) \td X^{\xi,x, x_i}_T + p_i\tilde \dbE_{\cF^0_T}\big[\td_\mu G(X^{\xi, x}_T, \rho_T, \tilde X^{\xi,x_i}_T, \tilde X^\xi_T; \td \tilde X^{\xi, x_i}_T,\td \tilde X^{\xi, x_i-}_T\big) \big]\\
\dis \q - \int_t^T \td_x \h L\big(\Xi^{\xi,x}_s; \td \Xi^{\xi,x,x_i}_s\big)ds-\int_t^T\td Z_s^{\xi,x,x_i}dB_s- \int_t^T \td Z_s^{0,\xi,x, x_i}dB_s^{0}; \\
\dis \td Y_t^{\xi,x,x_i-}= \pa_x G(X^{\xi,x}_T,\rho_T) \td X^{\xi,x, x_i-}_T - \int_t^T \td_x \h L\big(\Xi^{\xi,x}_s; \td \Xi^{\xi,x,x_i-}_s\big)ds\\
\dis\q + \tilde \dbE_{\cF^0_T}\big[\td_\mu G(X^{\xi, x}_T, \rho_T, \tilde X^{\xi,x_i}_T, \tilde X^\xi_T; \td \tilde X^{\xi, x_i}_T,\td \tilde X^{\xi, x_i-}_T\big) \big]\1_{\{\xi\neq x_i\}} \\
\dis \q -\int_t^T\td Z_s^{\xi,x,x_i-}dB_s- \int_t^T \td Z_s^{0,\xi,x, x_i-}dB_s^{0}; 
\ea\right.
\eeaa
Now for continuous $\xi$, let $\xi_n$ and $x^n_{i_n(x')}$ be as in Theorem \ref{thm-pamuV} Step 3. We can show that $( \td \Phi^{\xi_n,x, x^n_{i_n(x')}}, \td \Phi^{\xi_n, x, x^n_{i_n(x')}-})$  converges to $(0, \td_\mu \Phi^{\xi, x, x'-})$. 

(ii) Recall \reff{tdYx} and denote $\td^2_x \Phi^{\xi,x, \eta}:= \lim_{\e\to 0} {1\over \e}[\td_x\Phi^{\xi+\e\eta, x} - \td_x\Phi^{\xi,x}]$. When $\xi$ is discrete, we have $\td^2_x \Phi^{\xi,x, \1_{\{\xi=x_i\}}} = \td^2_x \Phi^{\xi,x, x_i} \1_{\{\xi=x_i\}} + p_i \td^2_x X^{\xi, x, x_i-}$, where:
\beaa
\dis\left.\ba{lll}
\dis \td^2_x X_t^{\xi,x,x_i}=\int_0^t \td_{xx} (\pa_p H)\big(\Xi^{\xi,x}_s; \td \Xi^{\xi,x,x_i}_s;  \td_x \Xi^{\xi,x}_s; \td^2_x X_s^{\xi,x,x_i} \big) ds; \\
\dis \td^2_x X_t^{\xi,x,x_i-}=\int_0^t \td_{xx} (\pa_p H)\big(\Xi^{\xi,x}_s; \td \Xi^{\xi,x,x_i-}_s;  \td_x \Xi^{\xi,x}_s; \td^2_x X_s^{\xi,x,x_i-} \big) ds; \\
\dis \td^2_x Y_t^{\xi,x, x_i}= \pa_x G(X^{\xi,x}_T, \rho_T) \td^2_x X_s^{\xi,x,x_i} + \pa_{xx} G(X^{\xi,x}_T, \rho_T) \td_x X^{\xi,x}_T\td X_T^{\xi,x,x_i} \\
\dis\q +p_i \tilde \dbE_{\cF^0_T}\Big[\td_\mu(\pa_x G)(X^{\xi,x}_T, \rho_T, \tilde X^{\xi,x_i}_T, \tilde X^{\xi}_T; \td \tilde X^{\xi,x_i}_T,\td \tilde X^{\xi,x_i-}_T)\Big] \td_x X^{\xi,x}_T\\
\dis\q  - \int_t^T \td_{xx} \h L\big(\Xi^{\xi,x}_s; \td \Xi^{\xi,x,x_i}_s;  \td_x \Xi^{\xi,x}_s; \td^2_x X_s^{\xi,x,x_i}\big)ds\\
\dis \q -\int_t^T\td^2_x Z_s^{\xi,x,x_i}dB_s- \int_t^T \td^2_x Z_s^{0,\xi,x,x_i}dB_s^{0}; \\
\dis \td^2_x Y_t^{\xi,x, x_i-}= \pa_x G(X^{\xi,x}_T, \rho_T) \td^2_x X_s^{\xi,x,x_i-} + \pa_{xx} G(X^{\xi,x}_T, \rho_T) \td_x X^{\xi,x}_T, \td X_T^{\xi,x,x_i-} \\
\dis\q +\tilde \dbE_{\cF^0_T}\Big[\td_\mu(\pa_x G)(X^{\xi,x}_T, \rho_T, \tilde X^{\xi,x_i}_T,\tilde X^{\xi}_T; \td \tilde X^{\xi,x_i}_T,\td \tilde X^{\xi,x_i-}_T)\Big] \td_x X^{\xi,x}_T\1_{\{\xi\neq x_i\}}\\
\dis\q  - \int_t^T \td_{xx} \h L\big(\Xi^{\xi,x}_s; \td \Xi^{\xi,x,x_i}_s;  \td_x \Xi^{\xi,x}_s; \td^2_x X_s^{\xi,x,x_i}\big)ds\\
\dis \q -\int_t^T\td^2_x Z_s^{\xi,x,x_i}dB_s- \int_t^T \td^2_x Z_s^{0,\xi,x,x_i}dB_s^{0}.
\ea\right.
\eeaa
Now for continuous $\xi$ with corresponding  $\xi_n$, $x^n_{i_n(x')}$, by the desired convergence in (i) and \reff{xinconv}, we can show that $( \td^2_x \Phi^{\xi_n,x, x^n_{i_n(x')}}, \td^2_x \Phi^{\xi_n, x, x^n_{i_n(x')}-})$  converges to $(0, \td_{\mu x} \Phi^{\xi, x, x'-})$.

(iii) Recall \reff{tdYx-} and denote $\td^2_x \Phi^{\xi,x-, \eta}:= \lim_{\e\to 0} {1\over \e}[\td_x\Phi^{\xi+\e\eta, x-} - \td_x\Phi^{\xi,x-}]$. When $\xi$ is discrete, we have $\td^2_x \Phi^{\xi,x-, \1_{\{\xi=x_i\}}} = \td^2_x \Phi^{\xi,x-, x_i} \1_{\{\xi=x_i\}} + p_i \td^2_x X^{\xi, x-, x_i-}$, where:
\beaa
\dis\left.\ba{lll}
\dis \td^2_x X_t^{\xi,x-, x_i}= \int_0^t \td_x (\pa_p H)\big(\Xi^{\xi}_s; \td \Xi^{\xi,x_i}_s; \td_x \Xi^{\xi,x-}_s; \td^2_x \Xi^{\xi,x-, x_i}_s\big) ds; \\
\dis \td^2_x X_t^{\xi,x-, x_i-}= \int_0^t \td_x (\pa_p H)\big(\Xi^{\xi}_s; \td \Xi^{\xi,x_i-}_s; \td_x \Xi^{\xi,x-}_s; \td^2_x \Xi^{\xi,x-, x_i-}_s\big) ds; \\
\dis \td^2_x Y_t^{\xi,x-,x_i}= \pa_x G(X^{\xi}_T, \rho_T) \td^2_x X_s^{\xi,x-,x_i} + \pa_{xx} G(X^{\xi}_T, \rho_T) \td_x X^{\xi,x-}_T \td X_T^{\xi,x_i} \\
\dis\q +  \tilde \dbE_{\cF^0_T}\Big[\td_\mu(\pa_x G)\big(X_T^{\xi},\rho_T, \tilde X^{\xi,x}_T, \tilde X^{\xi}_T; \td_x \tilde X^{\xi,x}_T , \td_x \tilde X^{\xi,x-}_T\big)\Big]  \td X^{\xi, x_i}_T \\
\dis\q + p_i \tilde \dbE_{\cF^0_T}\Big[\td_\mu(\pa_x G)\big(X^{\xi}_T, \rho_T, \tilde X^{\xi,x_i}_T, \tilde X^{\xi}_T; \td \tilde X^{\xi,x_i}_T, \td \tilde X^{\xi,x_i-}_T\big)\Big]  \td_x X^{\xi,x-}_T\\
\dis\q + p_iI^{\xi,x} \Big(X^\xi_T; \td^2_xX^{\xi,x,x_i}_T +\td^2_xX^{\xi,x,x_i-}_T, \td^2_x  X^{\xi,x-,x_i}_T +\td^2_x  X^{\xi,x-,x_i-}_T, \\
\dis\qq \td X^{\xi,x, x_i}_T + \td  X^{\xi,x,x_i-}_T, \td X^{\xi,x_i}_T + \td X^{\xi, x_i-}\Big) \\
\dis\q  - \int_t^T  \td_x \h L\big(\Xi^{\xi}_s; \td \Xi^{\xi,x_i}_s; \td_x \Xi^{\xi,x-}_s; \td^2_x \Xi^{\xi,x-, x_i}_s\big)ds \\
\dis\q-\int_t^T\td^2_x Z_s^{\xi,x-, x_i}dB_s- \int_t^T \td^2_x Z_s^{0,\xi,x-,x_i}dB_s^{0};\\
\dis \td^2_x Y_t^{\xi,x-,x_i-}= \pa_x G(X^{\xi}_T, \rho_T) \td^2_x X_s^{\xi,x-,x_i-} + \pa_{xx} G(X^{\xi}_T, \rho_T) \td_x X^{\xi,x-}_T \td X_T^{\xi,x_i-} \\
\dis\q +  \tilde \dbE_{\cF^0_T}\Big[\td_\mu(\pa_x G)\big(X_T^{\xi},\rho_T, \tilde X^{\xi,x}_T, \tilde X^{\xi}_T; \td_x \tilde X^{\xi,x}_T , \td_x \tilde X^{\xi,x-}_T\big)\Big]  \td X^{\xi, x_i-}_T 
\ea\right.
\eeaa
\beaa
\dis\left.\ba{lll}
\dis\q + \tilde \dbE_{\cF^0_T}\Big[\td_\mu(\pa_x G)\big(X^{\xi}_T, \rho_T, \tilde X^{\xi,x_i}_T, \tilde X^{\xi}_T; \td \tilde X^{\xi,x_i}_T, \td \tilde X^{\xi,x_i-}_T\big)\Big]  \td_x X^{\xi,x-}_T\1_{\{\xi\neq x_i\}}\\
\dis\q + I^{\xi,x} \Big(X^\xi_T; \td^2_xX^{\xi,x,x_i}_T +\td^2_xX^{\xi,x,x_i-}_T, \td^2_x  X^{\xi,x-,x_i}_T +\td^2_x  X^{\xi,x-,x_i-}_T, \\
\dis\qq \td X^{\xi,x, x_i}_T + \td  X^{\xi,x,x_i-}_T, \td X^{\xi,x_i}_T + \td X^{\xi, x_i-}\Big)\1_{\{\xi\neq x_i\}}\\
\dis\q  - \int_t^T  \td_x \h L\big(\Xi^{\xi}_s; \td \Xi^{\xi,x_i-}_s; \td_x \Xi^{\xi,x-}_s; \td^2_x \Xi^{\xi,x-, x_i-}_s\big)ds \\
\dis\q -\int_t^T\td^2_x Z_s^{\xi,x-, x_i-}dB_s- \int_t^T \td^2_x Z_s^{0,\xi,x-,x_i-}dB_s^{0}.
\ea\right.
\eeaa
Now for continuous $\xi$ with corresponding  approximations $\xi_n$, $x^n_{i_n(x')}$, by the desired convergence in (i), (ii), and \reff{xinconv}, We can show that $( \td^2_x \Phi^{\xi_n,x-, x^n_{i_n(x')}}, \td^2_x \Phi^{\xi_n, x-, x^n_{i_n(x')}-})$  converges to $(\td_{\mu x} \Phi^{\xi, x-, x'}, \td_{\mu x} \Phi^{\xi, x-, x'-})$. 

(iv) Recall \reff{tdmuYx} and denote $\td^2_\mu \Phi^{x, \xi,\tilde x, \eta}:= \lim_{\e\to 0} {1\over \e}[\td_\mu \Phi^{x,\xi+\e\eta, \tilde x} - \td_\mu\Phi^{x,\xi,\tilde x}]$. When $\xi$ is discrete, we have $\td^2_\mu \Phi^{x, \xi,\tilde x, \1_{\{\xi=x_i\}}} = p_i \td^2_\mu \Phi^{x, \xi,\tilde x,x_i}$, where: recalling \reff{tdYxipi},
\bea
\label{td2muYx}
\left.\ba{c}
 \td^2_\mu Y^{x, \xi, \tilde x, x_i}_t  = I^{\xi,\tilde x} \Big(X^x_T; \td^2_xX^{\xi,\tilde x,x_i}_T +\td^2_xX^{\xi, \tilde x,x_i-}_T, \td^2_x  X^{\xi,\tilde x-,x_i}_T +\td^2_x  X^{\xi,\tilde x-,x_i-}_T, \\
\dis\qq \td X^{\xi,\tilde x, x_i}_T + \td  X^{\xi,\tilde x,x_i-}_T, \td X^{\xi,x_i}_T + \td X^{\xi, x_i-}\Big)
 \\
\dis +\int_t^T \Big[\pa_p H(X_s^{x},Z_s^{x,\xi}) \td^2_\mu Z^{x,\xi,\tilde x, x_i}_s + \pa_{pp} H(X_s^{x},Z_s^{x,\xi}) \td Z^{x,\xi, x_i}_s \td^2_\mu Z^{x,\xi,\tilde x, x_i}_s\Big] ds \\
\dis\q -\int_t^T\td^2_\mu Z_s^{x,\xi, \tilde x, x_i}dB_s-\int_t^T\td^2_\mu Z_s^{0,x,\xi, \tilde x, x_i}dB_s^0.
\ea\right.
 \eea
 Now for continuous $\xi$ with corresponding  approximations $\xi_n$, $x^n_{i_n(\bar x)}$, by the desired convergence in (i), (ii), (iii), and \reff{xinconv}, we can show that $\td^2_\mu \Phi^{x, \xi_n, \tilde x, x^n_{i_n(\bar x)}}$  converges to $\td_{\mu\mu} \Phi^{x, \xi, \tilde x, \bar x}$. 

(v) Finally, it is obvious that $\dbE\big[\pa_{\mu\mu} V(0,x, \mu, \tilde x, \xi) \eta\big] = \td^2_\mu \Phi^{x, \xi,\tilde x, \eta}_0$. Then the rest of the proof follows similar arguments as  in Theorem \ref{thm-pamuV}, and we skip the details.
\qed

\section{Appendix}
\label{sect-proof}
\setcounter{equation}{0}
This Appendix consists of three types of materials:

\begin{itemize}
\item{} Some examples, especially counterexamples, to illustrate some points in the paper;

\item{} Proofs of some related results which are not used in the rest of the paper but nevertheless are interesting in their own rights;

\item{} Some proofs which are more or less standard but are provided for completeness. 
\end{itemize}

\subsection{Some results in Section \ref{sect-game}}
The first example shows that under our conditions the master equation typically does not have a classical solution.

\begin{eg}
\label{eg-nonclassical}
Let $d=1$, $T=1$, $F=0$, $H(x, z) = {1\over 2} |z|^2$, and, for $\cL_\xi = \mu$ as usual,
\beaa
G(x,\mu) = G(\mu) := \Big|\dbE\big[\big|\xi - \dbE[\xi]\big|\big]-{2\over \sqrt{\pi}}\Big|.
\eeaa
Then $V(0,x,\mu) =  V_0(\mu):= \Big|\dbE\big[\big|\xi - \dbE[\xi] + B_1\big|\big]-{2\over \sqrt{\pi}}\Big|$ is not differentiable in $\mu$.
\end{eg}
\proof We first note that the data here satisfy all our assumptions, including the monotonicity condition \reff{mon}. We next show that $V(0,x,\mu)= V_0(\mu)$. Fix $\xi \in \dbL^2(\cF_0, \mu)$ and set $t_0=0$, then \reff{FBSDE1} becomes:
\bea
\label{nonclassical-FBSDE1}
\left.\ba{c}
\dis X^{\xi}_t=\xi + \int_0^tZ_s^{\xi}ds+B_t+\b B_t^{0};\q \rho_t := \cL_{X_t^\xi|\cF^0_t}; \\
\dis Y_t^{\xi}=G( \rho_1)- {1\over 2}\int_t^1|Z_s^{\xi}|^2ds- \int_t^1Z_s^{\xi}dB_s-\int_t^1Z_s^{0,\xi}dB_s^{0}.
\ea\right.
\eea
Since $ G(\rho_1)$ is independent of $B$, from the BSDE above we see that $Z^\xi_t \equiv 0$, and thus
\bea
\label{nonclassical-FBSDE2}
X^{\xi}_t=\xi +B_t+\b B_t^{0};\q \rho_t := \cL_{X_t^\xi|\cF^0_t}; \qq Y_t^{\xi}=G(\rho_1)-\int_t^1Z_s^{0,\xi}dB_s^{0}.
\eea
Note further that
\beaa
&& G(\rho_1) =\Big| \dbE_{\cF^0_1}\big[\big| \xi +B_1+\b B_1^{0} - \dbE_{\cF^0_1}[\xi +B_1+\b B_1^{0}]\big|\big] - {2\over \sqrt{\pi}}\Big|\\
&&= \Big|\dbE_{\cF^0_1}\big[\big| \xi +B_1 - \dbE[\xi]\big|\big]  - {2\over \sqrt{\pi}}\Big|= \Big|\dbE\big[\big| \xi +B_1 - \dbE[\xi]\big|\big] -{2\over \sqrt{\pi}}\Big| = V_0(\mu),
\eeaa
which is deterministic. Plug this into \reff{nonclassical-FBSDE2}, we have $V(0,x,\mu) = V_0(\mu)$.  

Finally, we show that $V_0$ is not differentiable at $\mu=Normal(0, 1)$. Indeed, let $\xi = \eta \in \dbL^2(\cF_0, \mu)$. Then, for any $\e\in \dbR$, noting that $\xi+\e \eta + B_1 \sim Normal(0, 1+(1+\e)^2)$, 
\beaa
V_0(\cL_{\xi+\e\eta})= \Big|\dbE\big[\big|\xi+ \e\eta +B_1\big|\big] -{2\over \sqrt{\pi}}\Big| =\Big|\sqrt{1+(1+\e)^2}-\sqrt{2}\Big| \sqrt{2\over \pi}.
\eeaa
In particular, this implies that $V_0(\cL_\xi) = 0$. Then
\beaa
&&\lim_{\e\downarrow 0} {1\over \e}\Big[V_0(\cL_{\xi+\e\eta}) - V_0(\cL_\xi)\Big] =\lim_{\e\downarrow 0} {2\over \sqrt{\pi}\e}\Big|\sqrt{1+\e +{\e^2\over 2}}-1\Big| = {1\over \sqrt{\pi}};\\
&&\lim_{\e\uparrow 0} {1\over \e}\Big[V_0(\cL_{\xi+\e\eta}) - V_0(\cL_\xi)\Big] =\lim_{\e\uparrow 0} {2\over \sqrt{\pi}\e}\Big|\sqrt{1+\e +{\e^2\over 2}}-1\Big| =- {1\over \sqrt{\pi}}.
\eeaa
Recalling \reff{pamu}, this implies that $\pa_\mu V_0(\mu, \cd)$ does not exist. 
\qed

The next example shows that the comparison principle fails for master equations. 
\begin{eg}
\label{eg-comparison}
Let $d=1$, $\b=0$, $F=0$, $H(x, z) = {1\over 2} |z|^2$, and,  for $\cL_\xi = \mu$,
\beaa
G_i(x, \mu) = g_i(x)  + C_0 \dbE[\xi],\q i=1,2,
\eeaa 
where $C_0>0$ is a constant, $g_1\equiv 0$, and $g_2: \dbR\to (0, 1)$ is smooth and strictly decreasing. Then the master equation \reff{master} with terminal $G_i$ has a classical solution $V_i$. However, $G_1(x,\mu) < G_2(x,\mu)$,  but $V_1(0, 0, \d_0) > V_2(0, 0, \d_0)$ for $C_0$ large enough. 
\end{eg}
\proof We first solve the master equation. Let $u_i$ solves the following PDE:
\bea
\label{comparison-PDE}
\pa_t u_i(t,x) + {1\over 2} \pa_{xx} u_i + {1\over 2} |\pa_x u_i|^2 =0,\q u_i(T,x) = g_i(x).
\eea
It is straightforward to show that 
\bea
\label{comparison-ui}
u_i(t,x) = \ln \Big(\dbE\big[e^{g_i(x+ B_T-B_t)}\big]\Big),\q\mbox{and then}\q u_1 =0, \q u_2>0, \q \pa_x u_2 <0.
\eea
Moreover, for $\xi\in \dbL^2(\cF_0,\mu)$ and $x\in \dbR$,
\beaa
X^{i,\xi}_t = \xi +\int_0^t \pa_x u_i(s, X^{i,\xi}_s) ds + B_t,\q X^{i,x}_t = x + \int_0^t \pa_x u_i(s, X^{i,x}_s) ds + B_t.
\eeaa
Then one can easily see that 
\beaa
V_i(0, x, \mu) = u_i(0,x)+C_0\dbE[X^{i,\xi}_T].
\eeaa
Thus, by \reff{comparison-ui}, 
\beaa
V_1(0, 0, \d_0) = 0,\q V_2(0,0, \d_0) = u_2(0,0) + C_0\int_0^T \dbE[\pa_x u_2(s, X^{2,0}_s)]ds.
\eeaa
Since $\pa_x u_2< 0$, and note that $u_2$ and $X^{2,0}$ do not depend on $C_0$, then we see that $V_2(0,0, \d_0) <  0= V_1(0,0,\d_0)$ when $C_0$ is large enough. 
\qed

\subsection{Some results in Section \ref{sect-mollifier}}

\no {\bf Proof of Theorem \ref{thm-mollifier} (iii)}.   The key idea is to express $\pa_\mu U_n$ in terms of $\pa_\mu U$. We shall focus only on the first component: $\pa_\mu U \cd e_1$, where $e_1 := (1,0,\cds, 0)^\top\in \dbR^d$. Fix $\xi\in\mathcal{P}_2$ with $\cL_\xi = \mu$ and $\eta\in\mathcal{P}_2(\mathbb R^1)$. Recall \reff{pamu} and consider $U_n(\cL_{\xi + \e\eta e_1}) - U_n(\cL_\xi)$ for small $\e>0$. We proceed in three steps.

{\it Step 1.}  Recall \reff{Deltan}, \reff{muny}, and \reff{Un}. For each $y\in \D_n$, let $\xi_n(y)$ be a discrete random variable such that  $\cL_{\xi_n(y)} = \mu_n(y)$, namely $\dbP(\xi_n(y) = {\vec i\over n}) =\wh \psi_{\vec i}(\mu, y)$  for any $\vec i\in\mathbb Z^d$.  Note that we may construct $\xi_n(y)$ in a way such that $y\mapsto \xi_n(y)$ is measurable. In this step, we shall construct a random variable $\eta^\e_n(y)$ such that 
\bea
\label{etan}
\dbP(\xi_n(y) + \eta^\e_n(y) = {\vec i\over n}) = \wh\psi_{\vec i}(\cL_{\xi+\e \eta e_1}, y).
\eea

Since the perturbation of $\xi$ in the right side of \reff{etan} is only along $e_1$, we rewrite $\vec{i}=(i_1,\bar i)$ for some $\bar i\in \dbZ^{d-1}$. We note that rigorously we shall write $({\vec i})^\top = (i_1, (\bar i)^\top)$. However, this notation is really heavy, so in this subsection we abuse the notations for elements of $\dbZ^d_n$ and do not distinguish row  and column vectors. For each $\bar i\in \dbZ^{d-1}_n$, one can easily show that 
\beaa
\sum_{ |k|\le 2n^2}  |\psi_{(k, \bar i)}(\cL_{\xi+\e\eta e_1}) - \psi_{(k, \bar i)}(\mu)| \le C_n \e,
\eeaa
for some constant $C_n>0$ which may depend on $n$ and $\eta$, but independent of $\epsilon$. We next introduce a function $p^\e_{\vec i}(y)$ for $\vec i\in \dbZ^d_n$:
\bea
\label{pity}
p_{\vec i}^\e(y) :=   C_n\e  - {N_n\over N_n+1}\sum_{ k=-2n^2}^{i_1} [\psi_{(k, \bar i)}(\cL_{\xi+\e\eta e_1}) - \psi_{(k, \bar i)}(\mu)].
\eea
Then clearly $p_{\vec i}^\e(y)\ge 0$. Moreover, note that $|y_{\vec i}|\le {1\over N_n^3}$ for $\vec i\in\dbZ^d_n\setminus \{0\}$ and $|y_{\vec 0}| \le {N_n-1\over N_n^3}$,  by \reff{muny} we have
\beaa
\wh\psi_{\vec i}(\mu, y) \ge {N_n\over N_n+1} [{1\over N_n^2} - {N_n-1\over N_n^3}] = {1\over N_n^2(N_n+1)}.
\eeaa
Then $0\le p_{\vec i}^\e(y) \le 2C_n \e \le \wh\psi_{\vec i}(\mu, y)$ for all $\vec i\in \dbZ^d_n$ and all $\e \le {1\over 2C_nn N_n^2(N_n+1)}$. 

We now construct $\eta_n^\e(y)$. Note that both $\xi_n(y)$ and $\xi_n(y) + \eta_n^\e(y)$ take values in ${1\over n}\dbZ^d_n$. On $\{\xi_n(y) = {(i_1, \bar i)\over n}\}$, when $i_1 < 2n^2$, we set $\eta^\e_n(y)$ to take values $0$ and ${e_1\over n}$, so $\xi_n(y) + \eta_n^\e(y)$ take values ${(i_1, \bar i)\over n}$ and ${(i_1+1, \bar i)\over n}$; and when $i_1 = 2n^2$, we set $\eta^\e_n(y)$ to take values $0$ and $-4ne_1$, so $\xi_n(y) + \eta_n^\e(y)$ take values ${(2n^2, \bar i)\over n}$ and ${(-2n^2, \bar i)\over n}$. Moreover, we set their joint distribution as follows: recalling $\dbP(\xi_n(y) = {\vec i\over n}) =   \wh \psi_{\vec i}(\mu, y) \ge p_{\vec i}^\e(y)$, 
\bea
\label{lamdai}
&\dis\dbP(\xi_n (y)= {\vec i\over n}, \eta_n^\e(y) = \l(i_1)e_1) = p_{\vec i}^\e(y),\q \dbP(\xi_n(y)={\vec i\over n},\eta_n^\e(y)=0)=\wh \psi_{\vec i}(\mu, y)-p_{\vec i}^\e(y), \nonumber\\
&\dis\mbox{where}\q \l(i_1) := {1\over n}\1_{\{i_1 < 2n^2\}} - 4n \1_{\{i_1=2n^2\}}.
\eea
Then we can see, when $i_1 > -2n^2$,
\beaa
&&\dbP(\xi_n(y) + \eta^\e_n(y) = {\vec i\over n}) = \dbP\big(\xi_n(y)  = {\vec i\over n},  \eta^\e_n(y)=0\big) + \dbP\big(\xi_n(y)  = {(i_1-1, \bar i)\over n},  \eta^\e_n(y)={e_1\over n}\big)\\
&&\q =\big[\wh \psi_{\vec i}(\mu, y)-p_{\vec i}^\e(y)\big] + p_{(i_1-1, \bar i)}^\e(y) = \wh \psi_{\vec i}(\mu, y) + {N_n\over N_n+1} [\psi_{\vec i}(\cL_{\xi+\e\eta e_1}) - \psi_{\vec i}(\mu)]\\
&&\q = \wh\psi_{\vec i}(\cL_{\xi+\e \eta e_1}, y),
\eeaa
where the last equality thanks to \reff{muny}. Similarly we may verify \reff{etan} when $i_1=-2n^2$.

{\it Step 2.} We next compute $\pa_\mu U_n$. Since $U\in C^1(\cP_2)$, by   \reff{muny}, \reff{Un} and \reff{etan} we have
\beaa
&&U_n(\cL_{\xi+\e\eta e_1}) - U_n(\cL_\xi) = \int_{\D_n} \zeta_n(y) [U(\cL_{\xi_n(y) + \eta_n^\e(y)}) - U(\cL_{\xi_n})] dy\\
&&=\int_{\D_n} \zeta_n(y) \int_0^1 \dbE\big[\pa_\mu U(\cL_{\xi_n(y) + \th \eta_n^\e(y)}, \xi_n + \th \eta^\e_n(y)) \eta_n^\e(y)\big]d\th dy\\
&&=\int_{\D_n} \zeta_n(y) \int_0^1 \sum_{\vec i\in \dbZ^d_n} \l(i_1) \Big[\pa_\mu U\big(\cL_{\xi_n(y) + \th \eta_n^\e(y)}, ~ \l_n(\vec i, \th) \big)\cd e_1  p_{\vec i}^\e(y)\Big]d\th dy,\\
&&\mbox{where}\q \l_n(\vec i, \th) := {(i_1 + n \l(i_1) \th, \bar i)\over n}.
\eeaa
Note that we already know $\pa_\mu U_n$ exists, so it is determined by the Gateux derivative. Thus, by the continuity of $\pa_\mu U$ we have
\beaa
&&\dbE\Big[\pa_\mu U_n(\mu, \xi)\cd e_1 \eta\Big] = \lim_{\e\to 0} {1\over \e}\Big[U_n(\cL_{\xi + \e \eta e_1}) - U_n(\mu)\Big] \\
&&=\int_{\D_n} \zeta_n(y) \int_0^1 \sum_{\vec i\in \dbZ^d_n}\l(i_1) \Big[\pa_\mu U\big(\mu_n(y), ~  \l_n(\vec i, \th) \big)\cd e_1  \lim_{\e\downarrow 0} {p_{\vec i}^\e(y)\over \e}\Big]d\th dy.
\eeaa
Note that, by \reff{pity} and \reff{mun}
\beaa
 &&\lim_{\e\downarrow 0} {p_{\vec i}^\e(y)\over \e}=C_n - {N_n\over N_n+1} \sum_{ k=-2n^2}^{i_1} \dbE\big[\pa_\mu \psi_{(k, \bar i)}(\mu, \xi) \cd e_1\eta\big]\\
 &&=C_n - {N_n\over N_n+1} \sum_{ k=-2n^2}^{i_1} \dbE\big[\pa_x\big( \phi_{(k, \bar i)}\ch\big)(\xi) \cd e_1\eta\big] +  {N_n\over N_n+1}\dbE\big[\pa_x\ch(\xi) \cd e_1\eta\big] \1_{\{i_1 \ge 0, \bar i= \bar 0\}}.
\eeaa
Then
\beaa
&&\dbE\Big[\pa_\mu U_n(\mu, \xi)\cd e_1 \eta\Big] = C_n \int_{\D_n} \zeta_n(y) \int_0^1 \sum_{\vec i\in \dbZ^d_n}\l(i_1) \Big[\pa_\mu U\big(\mu_n(y), ~  \l_n(\vec i, \th) \big)\cd e_1\Big]d\th dy\\
&&-{N_n\over N_n+1}\int_{\D_n} \zeta_n(y) \int_0^1 \sum_{\vec i\in \dbZ^d_n}\l(i_1) \times\\
&&\qq\qq \Big[\pa_\mu U\big(\mu_n(y), ~  \l_n(\vec i, \th) \big)\cd e_1 \sum_{ k=-2n^2}^{i_1}\dbE\big[\pa_x\big( \phi_{(k, \bar i)}\ch\big)(\xi) \cd e_1\eta\big]\Big]d\th dy\\
&&+{N_n\over N_n+1}\int_{\D_n} \zeta_n(y) \int_0^1 \sum_{i_1=0}^{2n^2}\l(i_1) \Big[\pa_\mu U\big(\mu_n(y), ~  \l_n( (i_1,\bar 0) , \th)\big)\cd e_1 \dbE\big[\pa_x\ch(\xi) \cd e_1\eta\big]\Big]d\th dy.
\eeaa
Note that in \reff{pity} we can choose different $C_n$, while at above only the first term in the right side depends on $C_n$. So we must have
\beaa
\int_{\D_n} \zeta_n(y) \int_0^1 \sum_{\vec i\in \dbZ^d_n}\l(i_1) \Big[\pa_\mu U\big(\mu_n(y), ~ \l_n(\vec i, \th) \big)\cd e_1\Big]d\th dy =0,
\eeaa
hence
\beaa
&&\dbE\Big[\pa_\mu U_n(\mu, \xi)\cd e_1 \eta\Big] = -{N_n\over N_n+1}\int_{\D_n} \zeta_n(y) \int_0^1 \sum_{\vec i\in \dbZ^d_n}\l(i_1) \times\\
&&\qq\qq \Big[\pa_\mu U\big(\mu_n(y), ~\l_n(\vec i, \th)\big)\cd e_1 \sum_{ k=-2n^2}^{i_1}\dbE\big[\pa_x\big( \phi_{(k, \bar i)}\ch\big)(\xi) \cd e_1\eta\big]\Big]d\th dy\\
&&\q+{N_n\over N_n+1}\int_{\D_n} \zeta_n(y) \int_0^1 \sum_{i_1=0}^{2n^2}\l(i_1) \Big[\pa_\mu U\big(\mu_n(y), ~ \l_n( (i_1,\bar 0) , \th) \big)\cd e_1 \dbE\big[\pa_x\ch(\xi) \cd e_1\eta\big]\Big]d\th dy.
\eeaa
Since $\eta$ is arbitrary, we obtain
 \bea
 \label{paUn}
&&\pa_\mu U_n(\mu, x)\cd e_1  = -{N_n\over N_n+1}\int_{\D_n} \zeta_n(y) \int_0^1 \sum_{\vec i\in \dbZ^d_n}\l(i_1) \times\nonumber\\
&&\qq\qq \Big[\pa_\mu U\big(\mu_n(y), ~ \l_n(\vec i, \th) \big)\cd e_1 \sum_{ k=-2n^2}^{i_1}\pa_x\big( \phi_{(k, \bar i)}\ch\big)(x) \cd e_1\Big]d\th dy\\
&&\q+{N_n\over N_n+1}\int_{\D_n} \zeta_n(y) \int_0^1 \sum_{i_1=0}^{2n^2}\l(i_1) \Big[\pa_\mu U\big(\mu_n(y), ~ \l_n( (i_1,\bar 0) , \th)\big)\cd e_1 \pa_x\ch(x) \cd e_1\Big]d\th dy.\nonumber
\eea

{\it Step 3.} Finally we prove the convergence. Given $K\subset\subset \dbR^d$,  for $n$ large enough we have  $Q_n\supseteq K$ and thus $\ch\equiv 1$ on $K$. Moreover, for $x\in K\cap \D_{\vec i}$, we must have $|i_1|\leq n^2$ and thus $\l(i_1)=\l(i_1+1)={1\over n}$, and therefore  \reff{paUn}
becomes
\beaa
&&\pa_\mu U_n(\mu, x)\cd e_1  = -{N_n\over N_n+1}\int_{\D_n} \zeta_n(y) \int_0^1 \sum_{\vec j\in \dbZ^d_n}\l(j_1) \times\\
&&\qq \Big[\pa_\mu U\big(\mu_n(y), ~\l_n(\vec j, \th) \big)\cd e_1 \sum_{ k=-2n^2}^{j_1}\pa_x\phi_{(k, \bar j)}(x) \cd e_1\Big]d\th dy\\
&&= -{N_n\over n(N_n+1)}\int_{\D_n} \zeta_n(y) \int_0^1 \sum_{\bar j\in J_{\bar i}} \Big[\pa_\mu U\big(\mu_n(y), ~ {(i_1+\th, \bar j) \over n} \big)\cd e_1 \pa_x\phi_{(i_1, \bar j)}(x) \cd e_1\\
&&\qq + \pa_\mu U\big(\mu_n(y), ~ {(i_1+1+\th, \bar j) \over n} \big)\cd e_1 \big[\pa_x\phi_{(i_1, \bar j)}(x) + \pa_x\phi_{(i_1+1, \bar j)}(x) \big]\cd e_1\Big]d\th dy
\eeaa
By \reff{fi} and \reff{phi}, we see that 
\beaa
\pa_x\phi_{(i_1, \bar j)}(x) \cd e_1 = \phi_{i_1}'(x_1) \Pi_{l=2}^d \phi_{j_l}(x_l),\q  \big[\pa_x\phi_{(i_1, \bar j)}(x) + \pa_x\phi_{(i_1+1, \bar j)}(x) \big]\cd e_1=0.
\eeaa
 Then
\beaa
&&\pa_\mu U_n(\mu, x)\cd e_1  \\
&&= -{N_n\over n(N_n+1)}\int_{\D_n} \zeta_n(y) \int_0^1 \sum_{\bar j\in J_{\bar i}} \Big[\pa_\mu U\big(\mu_n(y), ~ {(i_1+\th, \bar j) \over n} \big)\cd e_1  \phi_{i_1}'(x_1) \Pi_{l=2}^d \phi_{j_l}(x_l)\Big]d\th dy\\
&&=-{N_n\over n(N_n+1)}\int_{\D_n} \zeta_n(y)  \pa_\mu U(\mu_n(y),  x) \cd e_1  \phi_{i_1}'(x_1) dy -{N_n\over n(N_n+1)}\int_{\D_n} \zeta_n(y) \times\\
&& \int_0^1\sum_{\bar j\in J_{\bar i}}\Big[ \big[\pa_\mu U\big(\mu_n(y), ~ {(i_1+\th, \bar j) \over n} \big)- \pa_\mu U(\mu_n(y),  x)\big]\cd e_1 \phi_{i_1}'(x_1)\Pi_{l=2}^d \phi_{j_l}(x_l)\Big]d\th dy.
\eeaa
By the uniform continuity of $\pa_\mu U$ and by \reff{munconv} (we assume $\pa_\mu U$ is continuous under $\cW_1$), we have
\beaa
\pa_\mu U_n(\mu, x)\cd e_1  &=& -{N_n\over n(N_n+1)}\int_{\D_n} \zeta_n(y)  [\pa_\mu U(\mu,  x)  \cd e_1  + o(1)] \phi_{i_1}'(x_1)  dy \\
&& -{N_n\over n(N_n+1)}\int_{\D_n} \zeta_n(y) \sum_{\bar j\in J_{\bar i}} o(1) \phi_{i_1}'(x_1)\Pi_{l=2}^d \phi_{j_l}(x_l) dy
\eeaa
By  \reff{I} and \reff{fi}, we see that $|\phi_{i_1}'(x_1)|\le n$ and thus $|\pa_\mu U_n(\mu, x)\cd e_1|\le C$. Moreover, for $ {i_1\over n} + {1\over n^2}\le x_1 \le {i_1+1\over n} - {1\over n^2}$, we have $\phi_{i_1}'(x_1)  = -n$, then 
\beaa
\pa_\mu U_n(\mu, x)\cd e_1  = {N_n\over N_n+1}\int_{\D_n} \zeta_n(y)  \pa_\mu U(\mu,  x)  \cd e_1  dy +o(1)=\pa_\mu U(\mu,  x)  \cd e_1  + o(1).
\eeaa
Therefore,
\beaa
 &&\int_{\D_{\vec i}} \Big|[\pa_\mu U_n(\mu, x)-\pa_\mu U(\mu, x)]\cd e_1\Big| dx \\
&&= \int_{\D_{\bar i}} \Big[\int_{{i_1\over n}}^{{i_1\over n}+{1\over n^2}} + \int_{{i_1\over n}+{1\over n^2}}^{{i_1+1\over n}-{1\over n^2}} + \int_{{i_1+1\over n}-{1\over n^2}}^{i_1+1\over n}\Big]\Big|[\pa_\mu U_n(\mu, x)-\pa_\mu U(\mu, x)]\cd e_1\Big| d x_1 d\bar x\\
&&\le  \int_{\D_{\bar i}} \Big[\int_{{i_1\over n}}^{{i_1\over n}+{1\over n^2}}  C dx_1+ \int_{{i_1\over n}+{1\over n^2}}^{{i_1+1\over n}-{1\over n^2}}  o(1) dx_1 + \int_{{i_1+1\over n}-{1\over n^2}}^{i_1+1\over n} C dx_1\Big] d\bar x\\
&&\le  \int_{\D_{\bar i}} \Big[{C\over n^2} + {o(1)\over n}\Big] d\bar x \le o(1) |\D_{\vec i}|.
\eeaa
 Then
 \beaa
 \int_K \Big|[\pa_\mu U_n(\mu, x)-\pa_\mu U(\mu, x)]\cd e_1\Big| dx &=& \sum_{\vec i\in \dbZ^d_n} \int_{K\cap \D_{\vec i}} \Big|[\pa_\mu U_n(\mu, x)-\pa_\mu U(\mu, x)]\cd e_1\Big| dx\\
 &\le& o(1) \sum_{\vec i\in \dbZ^d_n}|\D_{\vec i}| \1_{\{K\cap \D_{\vec i} \neq \emptyset\}} \le o(1),
  \eeaa
  where the $o(1)$ is uniform for $\mu\in \cM$. This completes the proof.
\qed

The following example shows that our mollifier does not keep the Lipschitz continuity under $\cW_2$ uniformly, as pointed out in Remark \ref{rem-mon} (i).
\begin{eg}
\label{eg-Lipschitz}
Let $d=1$, $\mu^m:=\delta_{\frac{m+2}{2m^2}}$, $\nu^m:=\delta_{\frac{m-2}{2m^2}}$, and $\mu^m_n(y), \nu^m_n(y)$ be defined by \reff{muny}. Set $U^m(\mu):= \cW_2(\mu, \nu^m_m(0))$, and $U^m_n$ the mollifier of $U^m$ with the $\z_n$ in \reff{Un} satisfying
\bea
\label{Lipschitzzetan}
\supp\zeta_n\subset  \big\{y=(y_{\vec i})_{\vec i\in \dbZ^d_n\setminus \{0\}} : |y_{\vec i}|\le N_n^{-4}\big\}.
 \eea
 Then $U^m$ is  uniformly Lipschitz continuous under $\cW_2$ with Lipschitz constant $1$, but
 \bea
 \label{Umm}
 |U^m_m(\mu^m) - U^m_m(\nu^m)|\ge {\sqrt{m}\over C}\cW_2(\mu^m, \nu^m),
 \eea
  and thus the Lipschitz constant of $\{U^m_n\}_{m,n\ge 1}$ under $\cW_2$ is not uniform.
 \end{eg}
 \proof  First, by \reff{mun} and \reff{muny} we have
 \beaa
\mu^m_n(0)=\frac{N_n}{N_n+1}\Big[\phi_0(\frac{m+2}{2m^2})\delta_0+[1-\phi_0](\frac{m+2}{2m^2})\delta_{\frac{1}{n}}\Big]+\frac{1}{N_n(N_n+1)}\sum_{i\in\mathbb Z_n}\delta_{\frac{i}{n}};\\
\nu^m_n(0)=\frac{N_n}{N_n+1}\Big[\phi_0(\frac{m-2}{2m^2})\delta_0+[1-\phi_0](\frac{m-2}{2m^2})\delta_{\frac{1}{n}}\Big]+\frac{1}{N_n(N_n+1)}\sum_{i\in\mathbb Z_n}\delta_{\frac{i}{n}}.
\eeaa
Recall \reff{fi0} and note particularly that $\phi_i$ depends on $n$. We can easily see that, for $m$ large,
\bea
\label{munum}
\left.\ba{c}
\dis \cW_2(\mu^m, \nu^m) = {2\over m^2},\\
\dis \cW^2_2(\mu^m_m(0), \nu^m_m(0)) = \frac{N_m}{m^2[N_m+1]}\big[\phi_0(\frac{m-2}{2m^2})-\phi_0(\frac{m+2}{2m^2})\big]=\frac{2N_m}{[N_m+1] m^3}.
\ea\right.
\eea
Next, recall  \reff{Un}, \reff{muny}, \reff{mun}, and \reff{tildeUn}, \reff{Lipschitzzetan}, we have
\beaa
U^m_m(\nu^m) &=&\int_{\D_m} \zeta_m(y) \mathcal{W}_2(\nu^m_m(y), \nu^m_m(0)) dy\\
&=&\int_{\D_m} \zeta_m(y) \mathcal{W}_2(\nu^m_m(0) +\frac{N_m}{N_m+1}\sum_{i\in\mathbb Z_m}y_i\delta_\frac{i}{m} , \nu^m_m(0)) dy\\
&\leq&\int_{\D_m}\zeta_m(y)\sqrt{\sum_{i\in\mathbb Z_m}|y_i|\frac{i^2}{m^2}}dy\le \int_{\D_m}\zeta_m(y)\sqrt{\sum_{|i|\le 2m^2}N_m^{-4}\frac{i^2}{m^2}}dy\leq\frac{C}{m^2}.
\eeaa
Similarly $\int_{\D_m} \zeta_m(y) \mathcal{W}_2(\mu^m_m(y), \mu^m_m(0)) dy\le {C\over m^2}$. Then by \reff{munum} we have
\beaa
U^m_m(\mu^m) &=&\int_{\D_m} \zeta_m(y) \mathcal{W}_2(\mu^m_m(y), \nu^m_m(0)) dy \\
&\ge&\mathcal{W}_2(\mu^m_m(0), \nu^m_m(0)) - \int_{\D_m} \zeta_m(y) \mathcal{W}_2(\mu^m_m(y), \mu^m_m(0)) dy \\
&\ge&\sqrt{\frac{2N_m}{[N_m+1] m^3}} - \frac{C}{m^2}\ge {1\over Cm^{3\over 2}}.
\eeaa
Thus, by  \reff{munum} again we have, for $m$ large enough,
\beaa
U^m_m(\mu^m) - U^m_m(\nu^m) \ge {1\over Cm^{3\over 2}} - {C\over m^2} \ge {1\over Cm^{3\over 2}}  = {\sqrt{m}\over C} \cW_2(\mu^m,\nu^m). 
\eeaa
This implies \reff{Umm} immediately.
\qed

The next example shows that the convergences  in Remark \ref{rem-mon} (ii)-(iii) do not hold.
\begin{eg}
\label{eg-mollifier}
Let $d=1$, $U(\mu) = \int_\dbR g(x) \mu(dx)$ for some function $g\in C^\infty_c(\dbR)$ with $g'(0)\ge 1$. Then, for $\mu = \d_0$ and $0\in K \subset\subset \dbR$, we have
\bea
\label{pamuUndiv}
\int_K \big|\pa_\mu U_n(\mu, x) -  \pa_\mu U(\mu, x)\big|\mu(dx) \ge 1,\q\mbox{for all $n$ satisfying $\supp(g) \subset [-n, n]$}.
\eea
\end{eg}
\proof By \reff{Un} we have
\beaa
U_n(\mu) &=& \int_{\D_n} \zeta_n(y) \sum_{|i|\le 2n^2}  {N_n\over N_n+1}[\psi_i(\mu) + {1\over N_n^2}+   y_i] g({i\over n})dy \\
&=&   {N_n\over N_n+1} \sum_{|i|\le 2n^2}  g({i\over n}) \Big[[\psi_i(\mu) + {1\over N_n^2}]+    \int_{\D_n} y_i\zeta_n(y) dy\Big].
\eeaa
Note  that $\pa_\mu U(\mu, x) = g'(x)$ and, for $n$ large as in \reff{pamuUndiv},
\beaa
 &&\pa_\mu U_n(\mu, x) =  {N_n\over N_n+1} \sum_{|i|\le 2n^2}  g({i\over n}) \pa_\mu\psi_i(\mu,x) \\
 &&={N_n\over N_n+1} \sum_{|i|\le 2n^2}  g({i\over n}) [(\phi_i\ch)'(x) - \1_{\{i=0\}}\ch'(x)] ={N_n\over N_n+1} \sum_{|i|\le 2n^2}  g({i\over n}) \phi_i'(x).
 \eeaa
By \reff{I} and \reff{fi0}, we see that $\phi_i'({j\over n}) =0$ for all $i, j$, and thus   $\pa_\mu U_n(\mu, {j\over n})=0$.  Therefore,  $\big|\pa_\mu U_n(\mu, 0) -  \pa_\mu U(\mu, 0)\big| = |g'(0)| \ge 1$. This implies \reff{pamuUndiv} immediately. 
\qed

The example below shows that our mollifier does not keep the monotonicity property \reff{mon}. For simplicity, we use a smooth function.  Of course in this case there is no need to mollify it, but we nevertheless use it for illustration purpose, and we will consider only the mollification in $x$, which is conceivably much simpler than the mollification in $\mu$, but already destroys the monotonicity property.
\begin{eg}
\label{eg-mollifier-mon} Let $U(x, \mu) := \big[|x|^2 - m_\mu^{(2)}\big]^2$, where $m_\mu^{(2)}:= \int_{\dbR^d} |x|^2\mu(dx)$. Then $U$ satisfies \reff{mon}, but the mollification of $U$ with respect to $x$ already violates \reff{mon}. 
\end{eg}
\proof First, similar to \reff{mon-eg2} one can easily check that $U$ satisfies \reff{mon}:
\beaa
\int_{\dbR^d} \big[U(x, \mu_1) - U(x, \mu_2)\big]\big[\mu_1(dx) - \mu_2(dx)\big] = - 2[m^{(2)}_{\mu_1} - m^{(2)}_{\mu_2}]^2\le 0.
\eeaa
Next, let $\z$ be a smooth kernel and consider the mollification of $U$ with respect to $x$:
\beaa
U_{x,\e}(x,\mu) := \int_{\dbR^d} U(x-\e y,\mu) \z(y) dy = \int_{\dbR^d} \big[|x-\e y|^2 - m_\mu^{(2)}\big]^2 \z(y) dy.
\eeaa
Then, recalling the $m_\mu$ in \reff{mon-eg1},
\beaa
&&\dis \int_{\dbR^d} \big[U_{x,\e}(x, \mu_1) - U_{x,\e}(x, \mu_2)\big]\big[\mu_1(dx) - \mu_2(dx)\big] \\
&&\dis=  -2[m^{(2)}_{\mu_1} - m^{(2)}_{\mu_2}] \int_{\dbR^d} \int_{\dbR^d} |x-\e y|^2  \z(y) dy\big[\mu_1(dx) - \mu_2(dx)\big]\\
&&\dis = -2[m^{(2)}_{\mu_1} - m^{(2)}_{\mu_2}]^2 + 4\e m_\z [m_{\mu_1} - m_{\mu_2}] [m^{(2)}_{\mu_1} - m^{(2)}_{\mu_2}].
\eeaa
This is not always negative.
\qed

\subsection{Some results in Section \ref{sect-regularity}}
\no {\bf Proof of Proposition \ref{prop-classical-x}.}  (i) We first note that $H$ is only locally Lipschitz continuous. For this purpose,  let $R>0$ be a constant which will be specified later. Let $I_{R}\in C^\infty(\dbR^d)$ be a truncation function such that $I_{R}(z)= z$ for $|z|\le R$,   $|\pa_z I_{R}(z)|=0$ for $|z|\ge R+1$, and $|\pa_z I_{R}(z)|\le 1$ for $z\in \dbR^d$. Denote $H_R(x,z)= H(x, I_R(z))$. Then clearly $|\pa_p H_R(x, z)|\le L^H_1(R+1)$ and $|\pa_x H_R(x,z)|\le \tilde L^H_1(R+1)$ for all $(x, z)\in \dbR^{d\times 2}$, where $\tilde L^H_1(R):= \sup_{(x, z)\in D_R}|\pa_x H_R(x,z)|$. Fix an arbitrary $(t_0,x)\in [0, T]\times \dbR^d$. Consider the following BSDE  on $[t_0, T]$ (abusing the notation here):
\bea
\label{FBSDE-x-modified}
\!\!\! Y_t^{x}=G(X_T^{x},\rho_T)+\!\!\int_t^T\!\! [F( X_s^{x},\rho_s)+H_{R}(X_s^x,Z_s^x)]ds-\!\!\int_t^T\!\! Z_s^{x} \!\cd\! dB_s^{t_0}-\!\!\int_t^T\!\! Z_s^{0,x}\!\cd\! dB_s^{0,t_0},
\eea
and denote $u(t_0,x) := Y_{t_0}^{x}$, which is $\cF^{0}_{t_0}$-measurable. By standard BSDE arguments, clearly the above system is wellposed, and it holds $ Y^x_t = u(t,  X^x_t)$, $Z^x_t = \pa_x u(t,  X^x_t)$.  Moreover, we have  $\pa_x u(t_0, x) = \td Y^{x}_{t_0}$, where
\bea
\label{FBSDE-pax-modified}
  \left.\begin{array}{ll} 
\dis \nabla Y_t^{x}=\pa_x G( X_T^{x},\rho_T)+\int_t^T[\pa_x F( X_s^{x},\rho_s)+\pa_x H_{R}( X_s^{x}, Z_s^{x})+\nabla Z_s^{x}~\pa_p H_{R}( X_s^{x}, Z_s^{x})]ds\\
\dis \qquad\qquad\qquad\qquad\quad\, -\int_t^T\nabla Z_s^{x}dB_s^{t_0}-\int_t^T\nabla Z_s^{0,x}dB_s^{0,t_0},\q t_0\le t\le T.
\end{array}
  \right.
\eea
Note that $|\pa_x G|\le L_1$, $|\pa_x F|\le L_1$, and $|\pa_x H_R|\le \tilde L^H_1(R+1)$,  one can easily see that
\beaa
|\pa_x u(t_0, x)| = |\td  Y^{x}_{t_0}| \le L_1[1+T] + T\tilde L^H_1(R+1).
\eeaa
Note that
\beaa
\limsup_{R\to\infty} {L_1[1+T] + T\tilde L^H_1(R+1)\over R} =\limsup_{R\to\infty} { T\tilde L^H_1(R+1)\over R+1} = Tc^H_1 < 1.
\eeaa
We may choose $R>0$ large enough  such that 
\beaa
|\pa_x u(t_0, x)|  \le L_1[1+T] + T\tilde L^H_1(R+1) \le R.
\eeaa
This proves \reff{pauest1} by setting $C_1=R$.   Moreover, since $|Z^x_t|= |\pa_x u(t, X^x_t)|\le R$, we see that $H_R(X^x_t, Z^x_t)= H(X^x_t, Z^x_t)$. Thus $(X^x, Y^x, Z^x,Z^{0,x})$ actually satisfies \reff{FBSDE-x}.

(ii) First by (i) we see that \reff{pauest1} holds and \reff{FBSDE-pax-modified} is the same as \reff{FBSDE-pax}. Next, by \cite[Theorem 6.1] {MYZ}  the above $u$ is a weak solution to the BSPDE in \reff{SPDE} with coefficient $H_R$ instead of $H$. However, since $|\pa_x u|\le R$, so $H_R(x,\pa_x u) = H(x, \pa_x u)$, and thus $u$ satisfies the BSPDE in \reff{SPDE} with coefficient $H$. The relation \reff{FBSDE-xu} also follows from \cite{MYZ}.

(iii) Fix $R$ as in (i). First, applying standard BSDE estimates on \reff{FBSDE-pax-modified} we see that
\bea
\label{tdZest}
\dbE_{t_0}\Big[ \Big(\int_{t_0}^T |\nabla Z^x_s|^2ds\Big)^2\Big] \le C,\q \mbox{a.s.}
\eea
Next, by standard stability arguments, we may assume without loss of generality that $F, G$ and $H$ are twice differentiable in $x$. Then we have $\pa_{xx} u(t_0, x) = \td^2  Y^x_{t_0}$, where, by differentiating \reff{FBSDE-pax-modified} formally in $x$:
\bea
\label{FBSDE-paxx-modified}
  \left.\begin{array}{ll} 
\dis \nabla^2 Y_t^{x}=\pa_{xx} G( X_T^{x},\rho_T)  - \int_t^T\sum_{i=1}^d \big[\nabla^2 Z_s^{i,x}dB_s^{i, t_0} +\nabla^2Z_s^{0,i, x}dB_s^{0,i, t_0}\big] \\
\dis\q +\int_t^T\Big[\pa_{xx} F(X_s^{x},\rho_s) + \pa_{xx} H_{R}(\cd) + 2\nabla Z_s^{x} \pa_{xp} H_{R}(\cd)  \\
\dis \q +  \nabla  Z_s^{x} \pa_{pp} H_{R}(\cd)  [\nabla Z_s^{x}]^\top+  \sum_{i=1}^d \nabla^2 Z_s^{i, x} \pa_{p_i} H_{R}(\cd) \Big](X_s^{x},Z_s^{x})ds.
\end{array}
  \right.
\eea
Denote $M^x_T := \exp\big(\int_{t_0}^T \pa_{p} H_{R}( X_s^{x},Z_s^{x}) \cd dB_s^{t_0} - {1\over 2} \int_{t_0}^T| \pa_{p} H_{R}(X_s^{x}, Z_s^{x})|^2 ds\Big)$. Then
\beaa
\nabla^2Y_{t_0}^{x} &=& \dbE_{t_0} \Big[M^x_T \pa_{xx} G( X_T^{x},\rho_T)  + M^x_T\int_{t_0}^T\big[\pa_{xx} F( X_s^{x},\rho_s)\\
&&+ \pa_{xx} H_{R}(\cd) + 2\nabla  Z_s^{x} \pa_{xp} H_{R}(\cd) +  \nabla Z_s^{x} \pa_{pp} H_{R}(\cd)  [\nabla Z_s^{x}]^\top \big]( X_s^{x},Z_s^{x})ds\Big].
\eeaa
Thus, by \reff{tdZest},
\beaa
|\pa_{xx} u(t_0, x)| &=& |\nabla^2  Y_{t_0}^{x}| \le C\dbE_{t_0} \Big[M^x_T   + M^x_T \int_{t_0}^T\big[1+ |\nabla  Z_s^{x}|^2\big]ds\Big]\\
&\le& C + C\Big( \dbE_{t_0} [|M^x_T|^2 ]\Big)^{1\over 2} \Big( \dbE_{t_0}\Big[ \big(\int_{t_0}^T |\nabla  Z^x_s|^2ds\big)^2\Big] \Big)^{1\over 2} \le C.
\eeaa
This is the required estimate. 
\qed

 {\bf Proof of Proposition \ref{prop-local}.}  (i) Let $R$ and $H_R, \h L_R$ be as in the proof of Proposition \ref{prop-classical-x} (i) and (ii). Note that $\pa_p H_R$ and $\h L_R$ are uniformly Lipschitz continuous. Then by the standard contraction mapping arguments we see that the FBSDE  \reff{FBSDE1} has a unique solution $(X^\xi, Y^\xi, Z^\xi,Z^{0,\xi})$, whenever $T\le \d_1$. Now denote $\rho_t := \cL_{X^\xi_t |\cF^{0}_t}$,  then the rest of the results follow immediately from Proposition \ref{prop-classical-x}.

(ii) Again it suffices to prove the result for $H_R$. In this case the existence of classical solution $V$ follows directly from  \cite[Theorems 5.10 and 5.11]{CD2}. 
\qed

\subsection{Some results in Section \ref{sect-weak}}
\begin{prop}
\label{prop-weak}
Assume $b: [0, T]\times \dbR^d \times \cP_2\to \dbR^d$ is continuous in all variables and bounded. Then the following PDE has a weak solution:
\bea
\label{SPDEb}
d\rho(t,x) = \Big[\frac{1}{2}\tr\big( \pa_{xx} \rho(t,x)\big) - div(\rho(t,x) b(t,x,\rho_t))\Big]dt,\q \rho(0,\cd) =\rho_0.
\eea
\end{prop}
\proof Fix $\xi\in \dbL^2(\cF_0, \rho_0)$. For any $\rho\in C^0([0, T]; \cP_2)$ with $\rho(0) = \rho_0$, set 
\beaa
&\dis X_t := \xi + B_t,\q B^\rho_t := B_t - \int_0^t b(s, X_s, \rho_s) ds,\\
&\dis {d\dbP^\rho\over d\dbP}:= M^\rho_T:= \exp\Big(\int_0^T  b(s, X_s, \rho_s)\cd  dB_s -{1\over 2} | b(s, X_s, \rho_s) |^2 ds\Big).
\eeaa
Then we may introduce a mapping $\Phi$ on $C^0([0, T]; \cP_2)$ by: $\Phi_t(\rho) := \dbP^\rho\circ X_t^{-1}$. By the continuity of $b$, it is clear that $\Phi$ is continuous. Moreover, since $b$ is bounded, by \cite[Lemma 4.1]{WZ}  the set $\{\Phi_t(\rho): \rho \in C^0([0, T]; \cP_2)\}$ is compact under $\cW_2$, for any $t\in [0, T]$. It is clear that $\cW_2(\Phi_s(\rho), \Phi_t(\rho)) \le C\sqrt{t-s}$ for all $0\le s<t\le T$.  Then the set $\{\Phi(\rho): \rho \in C^0([0, T]; \cP_2)\} \subset C^0([0, T]; \cP_2)$ is compact, under the metric $d(\rho, \rho') := \sup_{0\le t\le T} \cW_2(\rho_t, \rho'_t)$. Thus by Schauder fixed-point theorem we see that $\Phi$ has a fixed point $\rho$. It is clear that this $\rho$ is a weak solution to PDE \reff{SPDEb}. 
 \qed

\subsection{Some results in Section \ref{sect-convergence}}

\no{\bf Proof of Proposition \ref{prop-Nashwellposed}.}  We proceed in two steps. 

{\it Step 1.} Recall the truncation function $I_R$ and $H_R$ in the proof of Proposition \ref{prop-classical-x}. Denote $F_i(\vec x) := F(x_i, m^{N,i}_{\vec x})$, $G_i(\vec x) := G(x_i, m^{N,i}_{\vec x})$. For $n\ge 1$, let $F^n_i, G^n_i, H^n$ be the standard mollifier of $F_i, G_i, H$, which satisfy Assumptions \ref{assum-Lip1} and \ref{assum-H3}  uniformly in $n$.  Fix $(t_0, \vec x)$, recall  \reff{NashFBSDE1}, and consider the following system of BSDEs: $i=1,\cds, N$,
\bea
\label{NashFBSDEn}
 \left.\ba{lll}
\dis Y_{t}^{n, i,\vec x}=G^n_i(X_T^{\rightarrow,\vec x}) -\sum_{j=1}^N\int_t^TZ_{j,s}^{n,i,\vec x}\cd dB_s^{j}-\int_t^TZ_{s}^{n,0,i,\vec x}\cd dB_s^{0}\\
\dis+\int_t^T\Big[F^n_i(X_s^{\rightarrow,\vec x}) +H^n_n(X_{s}^{i,\vec x}, Z_{i,s}^{n,i,\vec x})+\sum_{j\neq i}I_n(Z_{j,s}^{n,i,\vec x})\cd\pa_{p}H_n^n(X_{s}^{j,\vec x},Z_{j, s}^{n, j, \vec x})\Big]ds.
\ea\right.
\eea
Obviously the above system is wellposed, and there exists a smooth function $u_n^i$ such that 
\beaa
Y_{t}^{n, i,\vec x} = u_n^i(t, X^{\rightarrow,\vec x}_t),\q Z_{j, t}^{n, i,\vec x} = \pa_{x_j} u_n^i(t, X^{\rightarrow,\vec x}_t),\q t\in [t_0, T].
\eeaa

We now derive the estimate:
\bea
\label{paun}
|\pa_{x_k} u_n^i| \le C_N,
\eea
 where  $C_N$  depends on $N$ and the parameters in the Assumptions, but not on $n$. By \reff{NashFBSDEn} we have
\beaa
&\dis u_n^i(t_0, \vec x) = Y_{t_0}^{n, i,\vec x} = \dbE\Big[G^n_i(X_{T}^{\rightarrow,\vec x})+\int_{t_0}^T\big[\sum_{j\neq i}I_n(\pa_{x_j} u_n^i(s, X^{\rightarrow,\vec x}_s))\cd\pa_{p}H_n^n(X_{s}^{j,\vec x},\pa_{x_j} u_n^j(s, X^{\rightarrow,\vec x}_s))  \\
&\dis + H^n_n(X_{s}^{i,\vec x}, \pa_{x_i} u_n^i(s, X^{\rightarrow,\vec x}_s))-H^n_n(X_{s}^{i,\vec x}, 0) +H^n_n(X_{s}^{i,\vec x}, 0) +F^n_i(X^{\rightarrow,\vec x}_s)  \big]ds\Big].
\eeaa
Note that $X^{\rightarrow,\vec x}_s$ has normal distribution and its components are conditionally independent, conditional on $\cF^0_s$. By integration by parts formula one can easily show that,
\beaa
\pa_{x_k}\dbE_{\cF^0_s}[\f(X^{\rightarrow,\vec x}_s)]= \dbE_{\cF^0_s}\Big[\f(X^{\rightarrow,\vec x}_s){B^{k,t_0}_s\over s-t_0}\Big],
\eeaa
 for any bounded and measurable function $\f$. Then
\bea
\label{paxun}
\left.\ba{c}
\dis \pa_{x_k} u_n^i(t_0, \vec x) =\dbE\Big[\pa_{x_k} G^n_i(X_{T}^{\rightarrow,\vec x})+  \int_{t_0}^T\big[\pa_{x_k}F^n_i(X^{\rightarrow,\vec x}_s) +\pa_{x_k}H^n_n(X_{s}^{i,\vec x}, 0)\\
\dis + [H^n_n(X_{s}^{i,\vec x}, \pa_{x_i} u_n^i(s, X^{\rightarrow,\vec x}_s))-H^n_n(X_{s}^{i,\vec x}, 0)]{B^{k, t_0}_s\over s-t_0} \\
\dis + \sum_{j\neq i}I_n(\pa_{x_j} u_n^i(s, X^{\rightarrow,\vec x}_s))\cd\pa_{p}H_n^n(X_{s}^{j,\vec x},\pa_{x_j} u_n^j(t, X^{\rightarrow,\vec x}_t)){B^{k, t_0}_s\over s-t_0}  \big]ds\Big].
\ea\right.
\eea
Denote $\G^n_s := \sup_{i, j}\sup_{\vec x} |\pa_{x_j} u_n^i (s,\vec x)|$. 
Then, by our assumptions,
\beaa
 |\pa_{x_k} u_n^i(t_0, \vec x)| \le C \dbE\Big[1+  \int_{t_0}^TN\G^n_s {|B^{t_0}_s|\over s-t_0} ds\Big] \le C + CN \int_{t_0}^T {\G^n_s \over \sqrt{s-t_0}}ds.
\eeaa
That is,
\beaa
\G^n_{t_0} \le C + CN \int_{t_0}^T {\G^n_s \over \sqrt{s-t_0}}ds,\q 0\le t_0\le T.
\eeaa
Then one can easily see  that $\sup_{0\le t\le T} \G^n_t \le C_N$, and hence \reff{paun} holds.

{\it Step 2.}  Now by \reff{NashFBSDEn}, we may view $u^i_n$ as a solution to the following heat equation:
\beaa
\pa_t u^i_n(t, \vec x) +{1\over 2} \sum_{j=1}^N\tr( \pa_{x_jx_j} u^i_n) + {\b^2\over 2}\sum_{j,k=1}^N\tr( \pa_{x_jx_k} u^i_n) + \tilde f^i_n(t, \vec x)=0,~ u^i_n(T,\vec x) = G^n_i(\vec x),
\eeaa
where, 
\beaa
\tilde f^i_n(t, \vec x) := \sum_{j\neq i}I_n(\pa_{x_j} u_n^i(t, \vec x))\cd\pa_{p}H_n^n(x_i,\pa_{x_j} u_n^j(t, \vec x))  + H^n_n(x_i, \pa_{x_i} u_n^i(t, \vec x)) +F^n_i(\vec x)
\eeaa
satisfies $|\tilde f^i_n|\le C_N[1+|\vec x|]$, thanks to \reff{paun}. Since $|\pa_{x_j}G^n_i|\le C_N$, then by standard PDE result we see that 
\bea
\label{paun2}
|\pa_{x_jx_k} u^i_n(t,\vec x)| \le {C_N\over \sqrt{T-t}}.
\eea
Now send $n\to \infty$, by \reff{paun} and \reff{paun2} it is clear that $u^i_n \to u^i$, $\pa_{x_j} u^i_n \to \pa_{x_j} u^i$ for some function $u^i$ such that $|\pa_{x_j} u^i|\le C_N$. Note that $I_n(\pa_{x_j} u^i) = \pa_{x_j} u^i$  for $n\ge C_N$, we see that 
\bea
\label{ui}
\left.\ba{c}
\dis \pa_t u^i(t, \vec x) +{1\over 2} \sum_{j=1}^N\tr( \pa_{x_jx_j} u^i )+ {\b^2\over 2}\sum_{j,k=1}^N\tr( \pa_{x_jx_k} u^i) + \tilde f^i(t, \vec x)=0, \\
\dis u^i(T,\vec x) = G_i(\vec x),\q\mbox{where}\\ 
\dis\tilde f^i(t, \vec x) := \sum_{j\neq i}\pa_{x_j} u^i(t, \vec x))\cd\pa_{p}H(x_i,\pa_{x_j} u^j(t, \vec x))  + H(x_i, \pa_{x_i} u^i(t, \vec x)) +F_i(\vec x)
\ea\right.
\eea
satisfies $|\tilde f^i|\le C_N[1+|\vec x|]$. Then we still have $u^i\in C^{1,2}([0, T)\times \dbR^{N\times d})$ and it satisfies \reff{paun2}. Note that \reff{ui} is exactly the Nash system \reff{Nash}, then $v_N^i:= u^i$ is a classical solution and \reff{Nashpaxv} holds. The uniqueness of classical solution satisfying \reff{Nashpaxv} is obvious.   

Finally, given the classical solution $v^i_N$, the wellposedness of \reff{NashFBSDE1}, \reff{NashFBSDE2}, and the relation \reff{YvN} are standard. 
 \qed

\ms
\no{\bf Proof of Lemma \ref{lem-empirical}.} First, by otherwise rescaling the problem, we may assume without loss of generality that $\|\cX\|_q =1$. Denote $\mu := \cL_\cX$, $\mu_i := \cL_{\cX_i}$, and $\mu^N := {1\over N} \sum_{i=1}^N \d_{\cX_i}$, which is a random measure. For any Borel set $A\subset \dbR^d$, note that
\[
\mu^N(A) = {1\over N} \sum_{i=1}^N \1_A(\cX_i).
\]
Since $\cX_1,\cds, \cX_N$ are independent,   then 
\beaa
&\dis\!\! \dbE\big[\mu^N(A)\big]= {1\over N}\sum_{i=1}^N\dbE\big[\1_A(\cX_i)\big]= {1\over N} \sum_{i=1}^N\mu_i(A) = \mu(A);\\
&\dis \!\! {\rm Var}\big[\mu^N(A)\big]={1\over N^2}\sum_{i=1}^N{\rm Var}\big[\1_A(\cX_i)\big] = {1\over N^2}\sum_{i=1}^N \mu_i(A)[1-\mu_i(A)]\leq {1\over N^2} \sum_{i=1}^N\mu_i(A) = {\mu(A)\over N}.
\eeaa
This implies:
\beaa
&\dbE\Big[\big|\mu^N(A) - \mu(A)\big|\Big] \le\dbE[\mu^N(A)] + \mu(A) =  2\mu(A);\\
&\dbE\Big[\big|\mu^N(A) - \mu(A)\big|\Big] \le \Big(\dbE\Big[\big|\mu^N(A) - \mu(A)\big|^2\Big]\Big)^{1\over 2} =  \Big({\rm Var}\big[\mu^N(A)\big]\Big)^{1\over 2} \le  \sqrt{\mu(A)\over N}.
\eeaa
Put together, we have
\begin{equation}\label{eq:minmin}
\dbE\Big[\big|\mu^N(A) - \mu(A)\big|\Big] \le \min\Big\{2\mu(A), \sqrt{\mu(A)\over N}\Big\}.
\end{equation}

We next introduce a partition $(B_n)_{n\geq 0}$ of $\mathbb R^d$: 
\beaa
B_0=(-1,1]^d,\q B_n=(-2^n,2^n]^d\setminus (-2^{n-1},2^{n-1}]^d,\quad n\geq 1,
\eeaa
and a sequence of partitions $(P_n)_{n\geq 0}$ of the hypercube $(-1,1]^d$ into $2^{dn}$ translations of the hypercube $(-2^{-n},2^{-n}]^d$. It is obvious that (recalling that we are assuming $\|\cX\|_q=1$),
\beaa
\mu(B_n) \le \mu\Big(\dbR^{Nd}\setminus (-2^{n-1},2^{n-1}]^d\Big) \le \dbP(|\cX|\ge 2^{-(n-1)})\le 2^{-q(n-1)}.
\eeaa
Moreover, using Cauchy-Schwarz inequality and the fact that the partition $P_m$ has exactly $2^{dm}$ elements, we deduce from \reff{eq:minmin} that for all $n,m\geq 0$
\bea
\label{eq:Che}
\left.\ba{c}
\dis \sum_{A\in P_m}\dbE\big[|\mu^N((2^n A)\cap B_n)-\mu((2^n A)\cap B_n)|\big]\\
\dis \le \min\Big[2\mu(B_n),\frac{2^{\frac{dm}{2}}}{\sqrt{N}}\sqrt{\mu(B_n)}\Big]\le \min\Big[2^{1-q(n-1)}, ~\sqrt{2^{dm-q(n-1)}\over N}\Big].
\ea\right.
\eea

Define
\begin{eqnarray*}
&\dis\mathcal{D}_p(\mu^N,\mu):=\sum_{n\ge 0}2^{pn}\Big[\big|\mu^N(B_n)-\mu(B_n)\big|\\
&\dis+{2^p-1\over 2}[\mu^N(B_n)\wedge\mu(B_n)]\sum_{m\ge 1}2^{-pm}\sum_{A\in P_m}\big|{\mu^N(2^nA\cap B_n)\over \mu^N(B_n)}-{\mu(2^nA\cap B_n)\over \mu(B_n)}\big|\Big].
\end{eqnarray*}
By \cite[Lemmas 5 and 6]{FG} there exists a constant $C>0$, depending only on $p, d$, such that
\beaa
&\dis \dbE\Big[\cW_p^p(\mu^N, \mu) \Big] \le C\dbE\Big[\mathcal{D}_p(\mu^N,\mu)\Big]\\
&\dis \le C\sum_{n\ge 0} 2^{pn} \sum_{m\ge 0} 2^{-pm} \sum_{A\in P_m}\dbE\big[|\mu^N((2^n A)\cap B_n)-\mu((2^n A)\cap B_n)|\big].
\eeaa
Plug \reff{eq:Che} into it, we obtain
\bea
\label{eq:DpmuNmu}
\dbE\Big[\cW_p^p(\mu^N, \mu) \Big] \ \le C\sum_{n\ge 0} 2^{pn} \sum_{m\ge 0} 2^{-pm}\min\Big[2^{-qn}, ~\sqrt{2^{dm-qn}\over N}\Big].
\eea
This is the formula (4) in \cite{FG}. Now following exactly line by line from Step 1 to Step 4 in the proof of  \cite[Theorem 1]{FG} we obtain the desired estimate \reff{empiricalConv}.
\qed

\bs
\no{\bf Proof of \reff{empiricald1}.} When $d=1$, we have a representation for $\cW_1$ (see, e.g. \cite{DGM}),
\beaa
\cW_1(\mu^N, \mu)= \int_{\dbR} \big|{1\over N} \sum_{i=1}^N \1_{\{\cX_i \le x\}} - \dbP(\cX\le x)\big| dx = {1\over N} \int_{\dbR} \big|\sum_{i=1}^N\eta^i(x) \big| dx,
\eeaa
where $\eta^i_x:=  \1_{\{\cX_i \le x\}} - \dbP(\cX_i\le x)$, $i=1,\cds, N$, are independent with $\dbE[\eta^i(x)]=0$. Then
\beaa
&&\dis \dbE\Big[\cW^2_1(\mu^N, \mu)\Big] ={1\over N^2} \dbE\Big[ \int_{\dbR^2} \big| \sum_{i=1}^N\eta^i(x_1) \big| \big|\sum_{i=1}^N\eta^i(x_2) \big| dx_1dx_2\Big]\\
&&\dis\le {1\over N^2}  \int_{\dbR^2} \Big(\dbE\big[\big| \sum_{i=1}^N\eta^i(x_1) \big|^2\big]\Big)^{1\over 2} \Big(\dbE\big[\big| \sum_{i=1}^N\eta^i(x_2) \big|^2\big]\Big)^{1\over 2} dx_1dx_2\Big]\\
&&\dis ={1\over N^2}  \Big(\int_{\dbR} \Big(\dbE\big[\big| \sum_{i=1}^N\eta^i(x) \big|^2\big]\Big)^{1\over 2} dx\Big)^2={1\over N^2}  \Big(\int_{\dbR} \Big(\sum_{i=1}^N\dbE[|\eta^i(x)|^2]\Big)^{1\over 2} dx\Big)^2\\
&&\dis={1\over N^2}  \Big(\int_{\dbR} \Big(\sum_{i=1}^N\dbP(\cX_i \le x)\dbP(\cX_i>x)\Big)^{1\over 2} dx\Big)^2\\
&&\dis \le {1\over N^2}  \Big(\int_{\dbR} \Big(\sum_{i=1}^N\dbP(|\cX_i| \ge |x|)\Big)^{1\over 2} dx\Big)^2 ={1\over N}  \Big(\int_{\dbR} \Big(\dbP(|\cX| \ge |x|)\Big)^{1\over 2} dx\Big)^2\\
&&\dis \le {2\over N}  \Big( 1+\int_1^\infty \Big(\dbE\big[{|\cX|^q \over |x|^q}\big]\Big)^{1\over 2} dx\Big)^2\le {C\over N}\big[1+\dbE[|\cX|^q]\big].
\eeaa
This completes the proof.
\qed

\end{document}